\documentclass{amsart}
\usepackage{graphicx,amssymb, verbatim, tikz-cd}
\usepackage[dvipsnames]{xcolor}
\definecolor{DarkPurple}{RGB}{88, 41, 123}
\usepackage[
    colorlinks=true,
    linkcolor=OliveGreen,
    citecolor=DarkPurple,
    urlcolor=NavyBlue
]{hyperref}
\usepackage{enumerate}
\newcommand{\Hecke}{\mathcal{H}}
\newcommand{\Z}{\mathbb{Z}}

\newcommand{\Ring}{\mathsf{R}}
\newcommand{\bone}{\mathbf{1}}
\newcommand{\Ha}{\mathcal{H}_{W}^a}

\newcommand{\gr}{\operatorname{gr}}
\newcommand{\B}{\mathcal{B}}
\newcommand{\mhm}{\mathrm{MHM}}
\newcommand{\mc}[1]{\mathcal{#1}}
\newcommand{\mb}[1]{\mathbb{#1}}
\newcommand{\rank}{\mathop{\mathrm{rank}}}
\newtheorem{Thm}{Theorem}[section]
\newtheorem{Lem}[Thm]{Lemma}

\newtheorem{Prop}[Thm]{Proposition}
\newtheorem{Cor}[Thm]{Corollary}
\theoremstyle{definition}
\newtheorem{defi}[Thm]{Definition}

\newtheorem{Rem}[Thm]{Remark}

\newtheorem{Ex}[Thm]{Example}

\numberwithin{equation}{section}

\title[A bar operation on the double affine Hecke algebra]{A bar operation on the\\ double affine Hecke algebra}
\author{Dougal Davis}
\author{Ivan Losev}
\author{Calder Morton-Ferguson}

\begin{document}

\begin{abstract}
    We construct a bar-type operation for the affine Hecke algebra $\Ha$ attached to a symmetrizable Kac--Moody group. Our main tool is a geometric bimodule defined using the equivariant $K$-theory of a suitable version of the Steinberg variety for this Kac--Moody group. Using elements of this bimodule corresponding to standard and costandard Hodge $D$-modules on the Kashiwara flag variety, we define a bar {operation on $\Ha$ that takes values in a suitable completion} by relating the left and right actions of $\Ha$ on this bimodule. A central issue is that the resulting formulas for the bar operation do not preserve the algebra itself and naturally produce infinite expansions. To address this, we introduce certain completions of $\Ha$ on which the bar operation is well-defined. For this construction, we prove a number of combinatorial finiteness results for products in $\Ha$. In affine type, we further analyze the level-zero part of the algebra and obtain an explicit combinatorial formula for a bar operation on a suitable localization of Cherednik's double affine Hecke algebra. In this case the bar operation itself is not an involution but we also show that this bar operation becomes an involution when the lattice parameter $q$ is set to $1$.
\end{abstract}

\maketitle

\begin{center}
{\it Dedicated to George Lusztig on his 80th birthday with admiration.}
\end{center}

\section{Introduction}

For the Hecke algebra associated to a Weyl group $W$, the bar involution is a basic structural operation, first defined in \cite{KL79}. Writing $T_{w}$ for the standard basis of the Hecke algebra over $\mathbb{Z}[v^{\pm 1}]$, { the bar involution is the $\mathbb{Z}$-linear operator on the Hecke algebra  characterized by}
\begin{align}
    \overline{v}=v^{-1},\qquad \overline{T_w}=T_{w^{-1}}^{-1},\label{eqn:bar-tw}
\end{align}
for all $w \in W$.

In the affine case, if one writes $X_\lambda$ for the Bernstein generator corresponding to a weight $\lambda$ (we establish conventions in Definition \ref{def:affine-hecke}), one obtains (for example, from the discussion preceding Lemma 7.1 in \cite{L83}) the formula
\begin{align}
\overline{X_{\lambda}} = T_{w_0}X_{w_0\lambda} T_{w_0}^{-1},\label{eqn:bar-ylambda}
\end{align}
{ where $w_0$ is the longest element in the finite Weyl group.}

The bar involution is fundamental in the construction of Kazhdan--Lusztig bases and arises from Verdier duality in the geometric realization of the Hecke algebra via constructible sheaves on the flag variety. 

It is natural to ask for an analogous operation in the double affine setting.
The purpose of the present paper is to construct such an operation for the affine Hecke algebra $\mathcal{H}_W^a$ attached to a Kac--Moody group in the sense of Garland--Grojnowski \cite{GG}, in the case when the Kac-Moody group is affine we get a version of the double affine Hecke algebra. This algebra is defined by a Bernstein presentation: it is generated by the Hecke algebra of $W$ { (with basis $T_w, w\in W$)} together with the group algebra of the weight lattice { $P$ (with basis $X_\lambda, \lambda\in P$)}, subject to the usual quadratic, braid, and Bernstein relations (which we recall in equations \eqref{eq:hecke_quadratic}--\eqref{eq:Bernstein_alternative}). 

When the Weyl group $W$ is finite, the algebra defined in loc.\ cit.\ is an affine Hecke algebra, and our bar operation recovers the ordinary bar involution. { When $W$ is infinite, we obtain a new bar operation. A basic difficulty here is as follows: one can define $\overline{T_w}$ by (\ref{eqn:bar-tw}), but
(\ref{eqn:bar-ylambda}) does not even make sense for infinite $W$ as there is no longest element.}


Our construction is motivated by geometry. Let $G$ be an extended symmetrizable Kac--Moody group with positive and negative Borel subgroups $B$ and $B^-$ (a setup described in \S \ref{sec:kacmoodysetup}), let $Y=G/B^-$ be Kashiwara's  flag variety as introduced in \cite{K}, and consider the quotient stack
\[
\mathcal Y=[B\backslash G/B^-].
\]

For example, suppose $G=G_0(\!(z)\!)$ for $G_0$ a connected reductive complex algebraic group with chosen opposite Borel subgroups $B_0^+$ and $B_0^-$, with $B=I^+$ the standard Iwahori subgroup and $B^-=I^-$ the co-Iwahori subgroup. Then
\[
Y=G/I^-
\]
is the thick affine flag variety, equivalently the moduli space of $G_0$-bundles on $\mathbb P^1$ with full level structure at $0$ and a $B_0^-$-reduction at $\infty$ (as recalled in \cite{Yun}), and
\[
\mathcal Y=[I^+\backslash G_0(\!(z)\!)/I^-].
\]

Returning to the general case, the cotangent stack $T^*\mathcal Y$ carries a natural $\mathbb G_m$-action { by fiberwise dilations}, and we show that its equivariant $K$-group
\[
\B = K_0^{\mathbb G_m}(T^*\mathcal Y)
\]
inherits commuting left and right actions of $\Ha$. Inside $\B$ there are natural standard and costandard classes, denoted $\bone_!^w$ and $\bone_*^w$, coming from the $B$-orbit stratification on $Y$. { They are constructed as the associated graded sheaves of the corresponding Hodge $D$-modules on $\mathcal{Y}$.} The central idea behind our construction is to use the commutation of an element of $\Ha$ past the bimodule element $\bone_!^1$ as a double-affine analog of the twisted conjugation by $T_{w_0}$ occurring in the formula (\ref{eqn:bar-ylambda}).

The equation we obtain from this commutation gives a candidate for the bar operation, but by itself it does not explain
why this operation should respect products. To prove multiplicativity one has to
compare two different ways of moving a product through the bimodule, and this
comparison naturally involves infinite expansions both in the Hecke algebra and in $\B$. We thus devote part of the paper to the construction of appropriate
completions in which our operation is a well-defined ring automorphism.

{  Now we digress to present a  way to justify the necessity of  completions. In the case of affine Hecke algebras, the bar involution comes from Verdier duality on the affine flag stack $[I^+\backslash G^\vee_0(\!(z)\!)/I^+]$, where $G^\vee_0$ is the Langlands dual of $G_0$. The partial order of $I^+\times I^+$-orbits on $G_0^\vee(\!(z)\!)$, i.e., the Bruhat order $\preceq$ on the affine Weyl group $W^a$, is ``ideal-finite'' in the sense that for every $y\in W^a$ the set $\{x\in W^a| x\preceq y\}$ is finite. In particular, Verdier duality sends the class of the shriek-extended constant local system on every orbit to a finite iterated cone of such objects associated to smaller orbits. In the case of the general (extended) Kac--Moody group $G$ one can still form the Langlands dual group $G^\vee$, the loop group $G^\vee(\!(z)\!)$, and its positive Iwahori subgroup $I^{da}$ (``da'' stands for ``double affine'', we use this notation to distinguish from the usual Iwahori subgroup in $G^\vee$ itself). It was observed by Garland, \cite{Garland}, in the case of untwisted affine $G^\vee$ that in order to have the Cartan decomposition one needs to replace $G^\vee(\!(z)\!)$ with a certain sub-monoid $G^\vee(\!(z)\!)^+\subset G^\vee(\!(z)\!)$. The same submonoid has the Bruhat decomposition, \cite{Braverman_Kazhdan_Patnaik}, where the strata are indexed by $W\ltimes \mathcal{T}$ with $\mathcal{T}\subset P$ being the Tits cone. Correspondingly, one can consider the $\Z[v^{\pm 1}]$-subalgebra  $\mathcal{H}_W^+\subset\mathcal{H}_W^a$ spanned by $T_w X_\lambda$ with $w\in W,\lambda\in \mathcal{T}$. It was proved in \cite{Braverman_Kazhdan_Patnaik} that $\mathcal{H}^+_W$ is essentially the Hecke algebra associated to the $\mathbb{F}_q$-points of  $I^{da}\backslash G^\vee(\!(z)\!)^+/I^{da}$. It should be noted that as a geometric object (as opposed to the set of $\mathbb{F}_q$ points),  $I^{da}\backslash G^\vee(\!(z)\!)^+/I^{da}$ is far more complicated than the usual affine flag stack, in particular, the theory of sheaves on this object has not been well understood, although a related category was studied in \cite{Bouthier_Vasserot}. In any case, one can imagine that the bar operation we define should have something to do with Verdier duality for sheaves on $I^{da}\backslash G^\vee(\!(z)\!)^+/I^{da}$.

To finish this part of the discussion let us mention that results and constructions of \cite{Braverman_Kazhdan_2011, Braverman_Kazhdan_Patnaik} were extended to the case of general Kac--Moody group $G$ by Gaussent and Rousseau, \cite{Gaussent_Rousseau} and by Bardy-Panse, Gaussent and Rousseau, \cite{BPGR}.

Now we get back to the discussion of completions. The set $W\ltimes \mathcal{T}$ comes with a partial order, to be called the double affine Bruhat order, that generalizes the usual affine Bruhat order in the case of finite $G$, see \cite{Muthiah_2018} and \cite{Muthiah_Orr}. If $G$ is infinite, the order is not ideal finite. Muthiah  in \cite{Muthiah_2018} and Bardy-Panse, Gaussent, Rousseau, \cite{BPGR} defined the standard basis in $\mathcal{H}_W^+$ (over $\Z[v^{\pm 1}]$,  the case of specialization to $v=q$ was handled already in \cite{Braverman_Kazhdan_Patnaik}) generalizing the usual affine case. The discussion of the ``double affine flag stack'' above implies that one should expect that applying the bar operation to the standard basis element labeled by $x\in W\ltimes \mathcal{T}$ one gets an infinite sum of standard basis elements, namely, those
labeled by elements $y\preceq x$.

However, it is not clear to us whether the set of infinite sums of standard basis elements going down with respect to the double affine Bruhat order actually has an algebra structure extending that on $\mathcal{H}^+_W$. In \cite{HM}, Hebert and Muthiah constructed a somewhat larger algebra completion containing all such infinite sums. We consider a related but easier completion of an ideal  $\mathcal{H}_W^>\subset\mathcal{H}_W^+$ that is the span of $T_w X_\lambda$ with $\lambda$ being ``strictly Tits'' (in the affine case this means that the level is strictly positive, the general definition will be given in Section  \ref{sec:productsinh}). The completion we need is introduced in Section \ref{SS_large_completion} under the name of the ``large completion''. We will see that this large completion still has a well-defined product, that the bar operation extends to the large completion and respects the product there.}


The combinatorial input needed for { establishing the product on the large completion, extending the bar-operation there, and checking the compatibility between the two} is an analysis of products of the form
\[
T_w^{-1}X_\mu T_u.
\]
When $\mu$ lies in the strict Tits cone, the possible weights that can occur are still constrained strongly enough for one to obtain useful finiteness statements.

{
To finish our discussion of the bar operation on the large completion, we mention several related works. In the case of the Hecke algebras associated to Weyl groups, the bar operation is recovered from the so-called $R$-polynomials. Deodhar in \cite{Deodhar_1985} found a combinatorial formula to compute these polynomials. Muthiah in the untwisted affine case, \cite{Muthiah_2019}, and then Hebert and Philippe, \cite{Hebert_Philippe} in the general Kac--Moody case found generalizations of Deodhar's formula, hence producing a version of the bar operation on the completion of $\mathcal{H}_W^+$ coming from the double affine Bruhat order. A relation between the bar operation defined in this way and the bar operation that we produce remains quite mysterious, at least to us.
}

{We now restrict to the case when $G$ is untwisted affine (and, for technical reasons, not of types $E_8,F_4,G_2$). So far we have discussed the Hecke algebra associated to the submonoid $W\ltimes \mathcal{T}$. The case of $W\ltimes (-\mathcal{T})$ can be handled similarly. So, what is missing is $W\ltimes P_0$, where $P_0$ is the sublattice of level $0$ weights, the case originally considered by Cherednik, \cite{Cherednik}.
We  study this case separately; the advantage of this setting is that for the algebra in this case, which we identify with Cherednik's double affine Hecke algebra, we give explicit combinatorial formulas (on generators) for a bar operation that is now defined on a suitable localization. We would like to emphasize that this construction is not, strictly speaking, a special case of the construction using the K-theory of an affine Steinberg variety, rather, it is inspired by a version of that construction}. The motivation for the formula we provide in the level-zero setting is closely tied to expected relations between the twisted standard and costandard generators originally introduced from the Hodge theory point of view in \cite[\S 3.2--3.3]{DMB}; computations with these elements provide a heuristic model for our construction.

{ The motivation for the bar operation on (a completion of) $\mathcal{H}^+_W$ comes from the Verdier duality for the affine flag variety of the Langlands dual Kac--Moody group $G^\vee$, so should arise from Langlands duality for Kac--Moody groups, currently in its nascency. A conceptual reason for our level 0 duality is even more mysterious. There is, however, a hint that it can be connected to the Dolbeault form of the usual Langlands duality. Namely, the localization we take is with respect to the lattice inside the original affine Hecke algebra $\Hecke_W$. So, on the Dolbeault Langlands dual side, this will be essentially the familiar localization to torus fixed points on the affine Steinberg variety for the usual Langlands dual group.}

{ A notable feature of the bar operations on various versions of $\mathcal{H}^a_W$ is that it is not clear whether they are involutions -- and in the level 0 affine situation we show that it is not.
Rather, we believe the resulting bar operation is best viewed as a deformation of an involution. In the level-zero Cherednik setting we show that after specializing the lattice parameter $q$ to $1$, the bar becomes an honest involution, while in rank $1$ we compute explicitly that for generic $q$ its square is nontrivial.} 

A possibly related failure of involutivity is observed in \cite{PPThesis}, which was { mentioned} to us by Dinakar Muthiah; a related phenomenon also occurs in forthcoming work of Muthiah and Tolmachov.

{ We now briefly describe the structure of the paper.}
We begin in \S \ref{sec:setup} by recalling the Kac--Moody affine Hecke algebra of \cite{GG} and continue in \S \ref{sec:geometric} by constructing the bimodule $\B$ geometrically from the Kashiwara flag variety. In \S \ref{sec:products}, we then develop the combinatorics of products in $\Ha$ and in $\B$, which leads to the two completions mentioned above {in \S\ref{sec:completions}}. With these tools in place, we define in \S \ref{sec:barinvolution} the bar operation on the appropriate completed algebra and prove that it is multiplicative. In \S \ref{sec:level0}, we treat the level-zero setting, compare with Cherednik's double affine Hecke algebra, and obtain an explicit formula for the resulting localized bar operation there.

{ {\bf Acknowledgements}. I.L. has been partially supported by the NSF under grant DMS-2501558. D.D. has been supported by the ARC grants DP250100824 and FL200100141. We would like to thank Dinakar Muthiah, Manish Patnaik, and Quan Situ for stimulating discussions.}

{
{\bf AI use.} The AI use for the mathematical content of the paper was limited to the following. We collaborated with Codex, GPT 5.5 and GPT 5.6 Sol, to come up with a generalization of Yun's construction from the affine to the general Kac-Moody case in  Section \ref{SSS_KM_general}. The authors had a proof of Theorem \ref{thm:levelzero-q1-involution} for type $A_1$ and used Codex to extend it to the general case. The authors also used Codex to self-referee the paper. It pointed out the following errors and gaps:
\begin{enumerate}
\item An incorrect topology in Section \ref{SS_large_completion}.
\item Conventional inconsistencies in Sections \ref{SSS_lattice_actions} and \ref{sec:coxeter-actions}.
\item An error in the proof of Proposition \ref{prop:omega0-endomorphism}. 
\end{enumerate}
It also pointed out a few minor omissions and missing assumptions. All of these were fixed by the authors. We used suggestions from Codex in fixing (2). We also used Codex to proofread and point out typos. All arguments and revisions were verified by the authors.
}
\section{Setup and Background}\label{sec:setup}\label{sec:kacmoodysetup}

Fix an extended symmetrizable Kac--Moody group $G$ over $k = \mathbb{C}$,
together with a Borel subgroup $B \subset G$ containing a Cartan subgroup $T$. Let $W=N_G(T)/T$ be the Weyl group, with simple reflections $\{s_i\}_{i\in I}$.
Let $P = X^*(T)$ be the weight lattice, let $Q=\bigoplus_{i\in I}\mathbb Z\alpha_i\subset P$
be the root lattice, and let $P^\vee=X_*(T)$ be the coweight lattice. Let $\Delta$
be the set of roots, with $\Delta_+$ and $\Delta_-$ the positive and negative roots
respectively, and $\{\alpha_i\}_{i \in I}$ the simple roots.
Set \(P_{\mathbb R}:=P\otimes_{\mathbb Z}\mathbb R\).

Let $B^-$ denote the negative Borel appearing in Kashiwara's construction of the
thick flag variety $G/B^-$ \cite[\S 5]{K}.
In finite type, $B^-$ is the ordinary opposite Borel
algebraic group. In the case of $G=G_0(\!(z)\!)$ for a reductive group
$G_0$, which we will call the \emph{untwisted affine case}, this negative
Borel is the co-Iwahori subgroup
\[I^-=\{g\in G_0[z^{-1}]\mid g(\infty)\in B_0^-\}\]
for the opposite Borel $B_0^- \subset G_0$. In general \(G/B^-\) is not a finite-type flag variety but Kashiwara's thick
flag variety; in \S\ref{sec:unionofstacks}, we will work with it through
finite-type approximations.

Let $\Ring=\mathbb Z[v^{\pm 1}]$ and set $\hbar=v-v^{-1}$.
Let $\Hecke_W$ be the Iwahori--Hecke algebra of $W$ over $\Ring$,
generated by $\{T_i\}_{i\in I}$ with the braid relations and
\begin{equation}\label{eq:hecke_quadratic}
(T_i-v)(T_i+v^{-1})=0 \qquad\text{(equivalently }T_i^{-1}=T_i-\hbar\text{)}.
\end{equation}
Let $\Ring P$ be the group algebra of $P$ over $\Ring$, with basis elements
$\{X_\lambda\}_{\lambda\in P}$ and multiplication $X_\lambda X_\mu=X_{\lambda+\mu}$.
The action of $W$ on $P$ extends $\Ring$-linearly to $\Ring P$; for $f\in \Ring P$ we
write $f^{s_i}$ for its image under $s_i$.

The following definition was first made in the Kac--Moody setting in \cite{GG}, although we use different normalizations in the present paper; see Remark \ref{rem:comparison} for a comparison between the two setups.
\begin{defi}\label{def:affine-hecke}
The \emph{affine Hecke algebra} $\Ha$ is the $\Ring$-algebra generated by
$\Hecke_W$ and $\Ring P$ along with the Bernstein relations:
\begin{equation}\label{eq:Bernstein_general}
fT_i-T_i f^{s_i}=\hbar\,\frac{f-f^{s_i}}{1-X_{\alpha_i}}
\qquad (f\in \Ring P).
\end{equation}
Equivalently, for $\mu\in P$ with $k=\langle \mu,\alpha_i^\vee\rangle$, one has
\begin{equation}
\label{eq:Bernstein_alternative}
X_\mu T_i - T_i X_{s_i\mu}
=
\begin{cases}
-\hbar\,(X_{s_i\mu}+ X_{s_i\mu+\alpha_i}+ \cdots + X_{\mu-\alpha_i}) & k>0,\\[2pt]
0 & k=0,\\[2pt]
\hbar\,(X_{\mu}+ X_{\mu+\alpha_i}+ \cdots + X_{s_i\mu-\alpha_i}) & k<0.
\end{cases}
\end{equation}
\end{defi}

\begin{Rem}\label{rem:comparison}
In \cite{GG} the Hecke parameter is denoted \(q\) and the quadratic relation is \((T_s+1)(T_s-q)=0\), whereas we use the more symmetric normalization over \(\Ring=\mathbb Z[v^{\pm1}]\) in which \((T_s-v)(T_s+v^{-1})=0\); these are related by the rescaling together with the identification \(q=v^2\). Likewise, both \cite{L83} and \cite{GG} write the Bernstein relation with denominator \(1-X_{-\alpha_i}\) while we write \(1-X_{\alpha_i}\) in \(\Ring P\). One can see this either as a flip from $X_\lambda$ to $X_{-\lambda}$ in $\Ring P$, or equivalently as a flip between $T_i$ and $T_i^{-1}$. We make this choice so that in Section \ref{sec:level0}, our conventions match the standard ones used for Cherednik's double affine Hecke algebra.
\end{Rem}

We will need a consequence of (\ref{eq:Bernstein_alternative}). For a $W$-orbit $W\lambda$ in $P$, let $\operatorname{Conv}(W\lambda)$ denote the intersection of the convex hull of $W\lambda$ in $P\otimes_{\mathbb{Z}}\mathbb{R}$ with $\lambda+Q$.


\begin{Lem}\label{Lem:Bernstein_consequence}
For all $\mu\in P$ and $w\in W$, we have
$$X_\mu T_w=T_w X_{w^{-1}\mu}+\sum_{u\in W, \nu\in P}a_u^w T_u X_\nu,$$
where $a_{u,\nu}^w=0$ unless $u$ is less than $w$ in the Bruhat order, and $\nu\in \operatorname{Conv}(W\mu)$.
\end{Lem}
\begin{proof}
If $w$ is a simple reflection, the claim follows from (\ref{eq:Bernstein_alternative}). The general case is handled using the case of simple reflections and induction on $\ell(w)$.
\end{proof}

\section{Geometric construction of the bimodule $\B$}\label{sec:geometric}

The main constructions of the present paper are obtained by means of a bimodule $\B$ over $\Ha$, which will be an object of central importance for our results. In this section, we define $\B$ geometrically, giving also a geometric realization of $\Ha$, and we explain how this perspective naturally endows $\B$ with the structure of an $\Ha$-bimodule.

\subsection{$B\backslash G / B^-$ as a union of stacks}\label{sec:unionofstacks}
In \cite{K}, Kashiwara defined $G/B^-$, the \emph{Kashiwara flag variety}. In the untwisted affine case, this is called the \emph{thick affine flag variety}. In \cite[\S 5.2]{Yun}, it is explained how to express $B\backslash G/B^-$ as a union of finite type smooth stacks in the untwisted affine case. The goal of this section will be to give a similar expression for $B\backslash G/B^-$ in the general case of symmetrizable Kac--Moody $G$. We first review the construction in loc.\ cit., which in turn is based on Kashiwara's results from \cite{K}, and we then explain how to generalize it beyond this untwisted affine case.

\subsubsection{Recollections from \cite{Yun}}\label{sec:yun-recollections}

We begin by recalling the setup for the affine flag variety as discussed in \cite[\S 5.2]{Yun}. Suppose, for this subsection, that $G$ is untwisted affine. For the purposes of this section we can take $G = G_0(\!(z)\!)$. Let $I^+ \subset G_0[[z]]$ be the standard Iwahori subgroup, defined as the preimage of the Borel subgroup $B^+_0 \subset G_0$ under the reduction map $G_0[[z]] \to G_0$. Similarly, let $I^- \subset G_0[z^{-1}]$ be the co-Iwahori subgroup, defined as the preimage of the opposite Borel $B^-_0$ under evaluation at $z = \infty$.

The affine flag variety is the ind-scheme $X = G/I^+$, while the \emph{thick affine flag variety} is the infinite-type scheme $Y = G/I^-$, which parametrizes $G_0$-torsors on $\mathbb{P}^1$ with full level structure at $0$ and a $B^-_0$-reduction at $\infty$. We let $X_{\preceq w} = \overline{X_{w}} \subset X$ denote the closure of the Schubert stratum $X_w = I^+w I^+/I^+$ considered with its reduced scheme structure.

The orbits of $I^+$ on $Y$ are indexed by $W$. For any $w \in W$, we denote the corresponding infinite-dimensional orbit by $Y_{w} = I^+ w I^- / I^-$. The closure relations are determined by the opposite Bruhat order: $Y_{w'} \subset \overline{Y_{w}}$ if and only if $w' \succeq w$. For a fixed element $u \in W$, we consider the open subscheme $Y_{\preceq u} = \bigsqcup_{w \preceq u} Y_{w}$.

In \cite[\S 5.2]{Yun} a principal congruence subgroup $K \subset G[[z]]$ is introduced which acts freely on $Y_{\preceq u}$ and trivially on the corresponding Schubert variety $X_{\preceq u}$. We now make this choice of $K$ explicit.

Let
\[
\Phi_{\preceq u}=\bigcup_{w\preceq u}\{\alpha\in\Delta_+\mid w^{-1}\alpha<0\}.
\]
This is finite, since the Bruhat interval $\{w\mid w\preceq u\}$ is finite and each
inversion set is finite. For an affine root $\alpha+n\delta$, the associated root
{ subalgebra} is contained in $z^n\mathfrak g_\alpha\subset \mathfrak g_0(\!(z)\!)$; we
refer to this integer $n$ as its $z$-degree. Let $m$ be an integer strictly greater
than the $z$-degree of any root in $\Phi_{\preceq u}$. We define $K = K_m$ to be the principal
congruence subgroup of level $m$:
\[
K_m = \{ g \in G[[z]] \mid g \equiv 1 \pmod{z^m} \}.
\]
With this choice, $K_m$ acts freely on $Y_{\preceq u}$ and trivially on $X_{\preceq u}$. The quotient $Z = K \backslash Y_{\preceq u}$ is a scheme of finite type, and since $K$ is normal in the Iwahori $I^+$, the quotient group $I^+/K$ acts on $Z$, stratifying it into finitely many orbits.

\subsubsection{The general Kac--Moody case}\label{SSS_KM_general}

We now generalize this construction to the setting of a general symmetrizable Kac--Moody
group $G$. Formally, we use the ``maximal'' Kac--Moody group defined in \cite[\S 6.1.16]{Kumar}.
For example, in the untwisted affine case, $G=T_\ell\ltimes \hat{G}$, where $T_\ell$ stands for the loop rotation $\mathbb{G}_m$ and $\hat{G}$ is a central extension of $G_0(\!(z)\!)$ by $\mathbb{G}_m$.

We again denote the Kashiwara flag variety by $Y = G/B^-$, where $B^-$ is the
negative Borel subgroup. In this geometric subsection, the left positive Borel $B$ is understood in
Kashiwara's completed sense when tail subgroups are used, see \cite[\S5]{K}. The $B$-orbits on $Y$
are indexed by $W$; for $w\in W$ we write
\[
Y_w = BwB^-/B^- \subset Y,
\qquad
Y_{\preceq w} = \bigsqcup_{v\preceq w} Y_v.
\]
(As in the untwisted affine case, the closure relations on $\{Y_w\}$ are opposite to
Bruhat order, hence each $Y_{\preceq w}$ is a $B$-stable open subscheme of $Y$; see
\cite[\S 5]{K} and cf.\ \cite[\S 5.2]{Yun}.)

For any $w \in W$, we wish to define a quotient $Z_{\preceq w}$, analogous to the
finite-type scheme constructed in \S\ref{sec:yun-recollections}, and then prove
that it is again a finite-type scheme. To do so, we must construct a specific
subgroup $K_w \subset B$ that generalizes the principal congruence subgroup $K_m$
from the untwisted affine case.

Following \cite[\S 5]{K}, we use the filtration of the completed positive
unipotent radical of $B$ by ``tail" subgroups. More precisely, if
$\widehat{\mathfrak n}^+=\prod_{\alpha\in\Delta_+}\mathfrak g_\alpha$ denotes the completed
positive nilpotent Lie algebra appearing in Kashiwara's construction, then a cofinite
ideal $S\subset\Delta_+$ determines the closed pro-unipotent subgroup
$U_S\subset B$ with Lie algebra
\[
\widehat{\mathfrak n}_S=\prod_{\alpha\in S}\mathfrak g_\alpha .
\]

\begin{defi}
A subset $S \subset \Delta_+$ is called a \emph{cofinite ideal} if it satisfies:
\begin{enumerate}
    \item The complement $\Delta_+ \setminus S$ is a finite set.
    \item The ideal property: $(S + \Delta_+) \cap \Delta_+ \subset S$.
\end{enumerate}
If $S$ is a cofinite ideal, the group $U_S$ is a closed normal subgroup of $B$, and
$B/U_S$ is a finite-dimensional solvable algebraic group \cite[\S 3.4, pp.~171--172]{K}.
\end{defi}

We now specify how to choose $S$ depending on the Weyl group element $w$. We replace the notion of ``degree" from \S\ref{sec:yun-recollections} with
the height of a root. For $\alpha = \sum k_i \alpha_i \in \Delta_+$, let
$\mathrm{ht}(\alpha) = \sum k_i$.

Let $w \in W$. For $v\in W$, let
$\Phi_v = \{ \alpha \in \Delta_+ \mid v^{-1}\alpha < 0 \}$ be the inversion set of $v$, and set
\[
\Phi_{\preceq w}=\bigcup_{v\preceq w}\Phi_v .
\]
This is finite, since the Bruhat interval below $w$ is finite. Let
\[
H_{\mathrm{max}} =
\begin{cases}
\max \{ \mathrm{ht}(\alpha) \mid \alpha \in \Phi_{\preceq w} \}, & \Phi_{\preceq w}\ne\varnothing,\\
0, & \Phi_{\preceq w}=\varnothing .
\end{cases}
\]
Define
\[
S_w = \{ \beta \in \Delta_+ \mid \mathrm{ht}(\beta) > H_{\mathrm{max}} \},
\qquad
K_w = U_{S_w}\subset B.
\]

\begin{Lem}\label{lem:Kw}
The subset $S_w$ is a cofinite ideal (so $K_w$ is a normal subgroup of $B$, hence $\mathcal{G}_w=B/K_w$ is a
finite-dimensional solvable algebraic group), and the subgroup $K_w$ satisfies the following properties:
\begin{enumerate}
    \item $K_w$ acts trivially on the Schubert variety $X_{\preceq w} \subset X=G/B$.
    \item $K_w$ acts freely on the truncated thick flag variety $Y_{\preceq w} \subset Y=G/B^-$.
\end{enumerate}
\end{Lem}

\begin{proof}
First, we verify that $S_w$ is a cofinite ideal. Since $\Delta_+\subset \sum_{i\in I}
\mathbb Z_{\geqslant 0}\alpha_i$ and $I$ is finite, the set of all $\sum_i k_i\alpha_i$ with
$\sum_i k_i\leqslant H_{\mathrm{max}}$ is finite; hence
$\{\alpha\in \Delta_+:\mathrm{ht}(\alpha)\leqslant H_{\mathrm{max}}\}$ is finite and
$\Delta_+\setminus S_w$ is finite. For the ideal property, if $\beta \in S_w$ and
$\gamma \in \Delta_+$ with $\beta+\gamma\in \Delta_+$, then
\[
\mathrm{ht}(\beta + \gamma)=\mathrm{ht}(\beta)+\mathrm{ht}(\gamma)>H_{\mathrm{max}},
\]
so $\beta+\gamma\in S_w$. Thus $S_w$ is a cofinite ideal.

We now prove (1). Fix $v\preceq w$. By construction
$\Phi_v\subset \Phi_{\preceq w}$, hence $S_w\cap \Phi_v=\varnothing$. Equivalently,
$v^{-1}S_w\subset\Delta_+$, where positive imaginary roots remain positive under the
Weyl group action. Kashiwara's root-space description of the completed unipotent
radical then gives
\[
K_w\subset B\cap vBv^{-1}.
\]
Thus $K_w$ fixes the $T$-fixed point $vB\in X$. Since $K_w$ is normal in $B$, it fixes
the entire $B$-orbit $X_v=B\cdot(vB)$. Therefore $K_w$ acts trivially on
$\bigcup_{v\preceq w} X_v = X_{\preceq w}$.

Finally we prove (2). It suffices to check that $K_w$ has trivial stabilizer on each
stratum $Y_v$ for $v\preceq w$. Let $y_v=vB^-\in Y_v$. Then
\[
\mathrm{Stab}_B(y_v)=B\cap vB^-v^{-1}.
\]
The description of the $B$-orbit structure in \cite[4.5.7--4.5.9]{K} and of the intersection
$\widehat U^+\cap v\widehat U^-v^{-1}$ in \cite[(4.5.2)]{K} imply that the unipotent radical of
$B\cap vB^-v^{-1}$ has completed Lie algebra supported precisely on the inversion
roots $\Phi_v$. Since $S_w\cap\Phi_v=\varnothing$, the root-space
factorization of the completed unipotent radical gives
\[
K_w \cap (B\cap vB^-v^{-1})=\{1\},
\]
the torus part being irrelevant because $K_w$ is pro-unipotent. For a point
$b\cdot y_v\in Y_v$, the stabilizer is the $b$-conjugate of $\mathrm{Stab}_B(y_v)$; since $K_w$ is
normal in $B$, the same intersection computation applies. Hence $K_w$ acts freely on
$Y_v$, and therefore on $Y_{\preceq w}$.
\end{proof}

Using this subgroup, we then define
\[
Z_{\preceq w}=K_w\backslash Y_{\preceq w}
\]
as the quotient by the free $K_w$-action.

\begin{Prop}\label{prop:Zfinite}
The quotient $Z_{\preceq w}$ is a smooth scheme of finite type. Moreover, $\mathcal G_w=B/K_w$ acts on $Z_{\preceq w}$, and
this action has finitely many orbits $Z_v$ with $v\preceq w$ so that
\[
Z_{\preceq w}=\bigsqcup_{v\preceq w}Z_v,\qquad Z_v=K_w\backslash Y_v .
\]
\end{Prop}

\begin{proof}
Let \(x_0=B^-/B^-\in Y\).  For \(v\preceq w\), set
\[
\Omega_v=vUB^-/B^-=vUx_0\subset Y .
\]
Each \(\Omega_v\) is an open affine chart of
\(Y=G/B^-\), and \(Y_v=Bvx_0\subset \Omega_v\); hence the finitely many opens
\(\Omega_v\cap Y_{\preceq w}\), for \(v\preceq w\), cover \(Y_{\preceq w}\), as explained in \cite[\S5.8]{K}.  Moreover, since
\(S_w\cap\Phi_v=\varnothing\), we have \(v^{-1}S_w\subset\Delta_+\), and therefore
\(v^{-1}K_wv\subset U\).  Thus each \(\Omega_v\cap Y_{\preceq w}\) is \(K_w\)-stable.
Under the identification
\[
\Omega_v\simeq U,\qquad vu x_0\longmapsto u,
\]
the action of \(K_w\) is the left translation action of the closed pro-unipotent
subgroup
\[
K_{w,v}=v^{-1}K_wv\subset U .
\]

The subgroup \(K_{w,v}\) has finite codimension in \(U\).  Indeed,
\(v^{-1}S_w\subset\Delta_+\) has finite complement, so for some integer \(N_v\) the
set
\[
S_{N_v}=\{\alpha\in\Delta_+\mid \operatorname{ht}(\alpha)>N_v\}
\]
is contained in \(v^{-1}S_w\).  Hence \(U_{S_{N_v}}\subset K_{w,v}\).  Kashiwara's
finite-dimensional quotient construction for \(U\) gives a finite-dimensional
unipotent algebraic group
\[
U^{(N_v)}=U/U_{S_{N_v}}
\]
\cite[Lemma 4.4.2]{K}.  If \(\overline K_{w,v}\) denotes the image
of \(K_{w,v}\) in \(U^{(N_v)}\), then
\[
K_{w,v}\backslash U \simeq
\overline K_{w,v}\backslash U^{(N_v)} .
\]
The latter quotient is a smooth finite-dimensional scheme, since it is the quotient
of a finite-dimensional unipotent algebraic group by a closed unipotent subgroup.

Since \(\Omega_v\cap Y_{\preceq w}\) is \(K_w\)-stable, its image in \(U^{(N_v)}\) is a
\(\overline K_{w,v}\)-stable open subscheme.  Therefore
\[
K_w\backslash(\Omega_v\cap Y_{\preceq w})
\]
is an open subscheme of
\(\overline K_{w,v}\backslash U^{(N_v)}\).  On overlaps these local quotients agree,
because they are obtained by quotienting the same \(K_w\)-stable open subschemes of
\(Y_{\preceq w}\).  Thus the local quotients glue, exactly as in the criterion given in \cite[\S5.7]{K}, to the quotient \(Z_{\preceq w}\).

Since the interval \(\{v\in W\mid v\preceq w\}\) is finite, \(Z_{\preceq w}\) is covered by
finitely many open subschemes of finite-dimensional smooth quotients
\(\overline K_{w,v}\backslash U^{(N_v)}\).  Hence \(Z_{\preceq w}\) is a smooth scheme of finite type.  Finally, because $K_w$ is normal in $B$, the quotient
group
\[
\mathcal G_w=B/K_w
\]
acts on \(Z_{\preceq w}\).  Its orbits are precisely the images of the \(B\)-orbits
\(Y_v\), \(v\preceq w\); in other words,
\[
Z_{\preceq w}=\bigsqcup_{v\preceq w} Z_v,\qquad
Z_v=K_w\backslash Y_v .
\]
\end{proof}

\subsubsection{Constructing the colimit of stacks}

We are now in a position to describe the full double quotient stack $[B \backslash G /
B^-]$ as a colimit of smooth stacks of finite type.

Let $\mathcal{Y}$ denote the stack quotient $[B \backslash Y]$, where $B$ acts on
$Y = G/B^-$ by left multiplication. We define a system of finite-type approximations
$\mathcal{Y}_{\preceq w}$ indexed by the Weyl group $W$. For each $w \in W$, let
\[
\mathcal{Y}_{\preceq w} = [\mathcal{G}_w \backslash Z_{\preceq w}]
= [(B/K_w) \backslash (K_w \backslash Y_{\preceq w})].
\]
Note that this is naturally isomorphic to the stack quotient $[B \backslash Y_{\preceq w}]$.

Consider two elements $w, w' \in W$ such that $w \preceq w'$. Then $Y_{\preceq w} \subset Y_{\preceq
w'}$ is a $B$-stable open immersion, hence it induces an open immersion of quotient stacks
\[
[B \backslash Y_{\preceq w}] \hookrightarrow [B \backslash Y_{\preceq w'}].
\]

\begin{Thm}\label{thm:Y-colimit}
The stack $\mathcal{Y} = [B \backslash G / B^-]$ is a colimit of smooth stacks of finite
type. Specifically,
\[
\mathcal{Y} \cong \varinjlim_{w \in W} \mathcal{Y}_{\preceq w},
\]
where the transition maps $\mathcal{Y}_{\preceq w} \hookrightarrow \mathcal{Y}_{\preceq w'}$ are open
immersions for $w \preceq w'$.
\end{Thm}

\begin{proof}
For each $w$, the scheme $Z_{\preceq w}$ is smooth of finite type and the group
$\mathcal{G}_w=B/K_w$ is a smooth finite-dimensional solvable algebraic group. It follows
that the quotient stack $\mathcal{Y}_{\preceq w} = [\mathcal{G}_w \backslash Z_{\preceq w}]$ is a smooth stack of
finite type.

The quotient stack $\mathcal{Y}$ is covered by the open substacks $\mathcal{Y}_{\preceq w}$, and
the transition maps for $w\preceq w'$ are open immersions. The universal property of the
colimit for a cover by open substacks gives the claimed identification
$\mathcal{Y} \cong \varinjlim_{w\in W} \mathcal{Y}_{\preceq w}$.
\end{proof}
\subsection{Geometric construction of the bimodule}
Since $\mathcal{Y}_{\preceq w}=[\mathcal G_w\backslash Z_{\preceq w}]$
with $Z_{\preceq w}$ a smooth scheme of finite type and $\mathcal G_w$ a
finite-dimensional algebraic group, we use the usual cotangent stack of this
quotient as the Hamiltonian reduction
\begin{equation}\label{eq:truncated_cotangent}
T^*\mathcal{Y}_{\preceq w}\cong [\mu_w^{-1}(0)/\mathcal G_w],
\end{equation}
where $\mu_w$ is the moment map for the
$\mathcal G_w$-action on $Z_{\preceq w}$. Thus $T^*\mathcal{Y}_{\preceq w}$ is a finite-type
algebraic quotient { derived } stack, typically not a scheme; the notation $T^*\mathcal Y$
below is shorthand for the colimit of this compatible system of finite-type cotangent stacks.

By Theorem~\ref{thm:Y-colimit}, we can define $\mathcal{B}$ as the $\mathbb{G}_m$-equivariant $K$-group
\begin{equation}\label{eqn:bdef}
    \mathcal{B} = K_0^{\mathbb{G}_m}(T^*\mathcal{Y}) = \varprojlim_{w} K_0^{\mathbb{G}_m}(T^*\mathcal{Y}_{\preceq w}).
\end{equation}

Here, the action of $\mathbb{G}_m$ scales the fibers of $T^*\mathcal{Y}$ with weight $v^2$ (and thus $\mathcal{B}$ carries the structure of an $\Ring$-module) and the inverse limit is taken with respect to the restriction maps to the open substacks $\mathcal{Y}_{\preceq w}$. In this subsection, we equip $\mathcal{B}$ with the structure of a topological $(\Hecke_W^a, \Hecke_W^a)$-bimodule.

We will be particularly interested in certain standard classes in the $K$-group $\B$, constructed as the associated gradeds of mixed Hodge modules. Before we give the definition, we first fix some conventions for thinking about mixed Hodge modules on smooth stacks like $\mc{Y}$.

Recall that if $Z$ is any smooth scheme of finite type over $\mathbb{C}$, we have the category $\mhm(Z)$ of mixed Hodge modules on $Z$. This comes equipped with a forgetful functor $\mhm(Z) \to \mathrm{Perv}(Z, \mathbb{Q})$ and a Hodge-graded functor $\gr^H : \mhm(Z) \to \mathrm{Coh}^{\mathbb{G}_m'}(T^*Z)$, where we write $\mathbb{G}_m'$ for the multiplicative group with character group $u^{\mathbb{Z}}$, which acts by scaling the fibers of $T^*Z$ with weight $u^{-1}$. To fix conventions, we will take the Hodge-graded functor to be defined by taking the associated graded of the Hodge filtration on the associated left $\mathcal{D}$-modules: in particular, the Hodge-graded of the natural mixed Hodge structure on the constant perverse sheaf $\underline{\mathbb{Q}}_Z[\dim Z]$ is $\gr^H\underline{\mathbb{Q}}_Z[\dim Z] = \mathcal{O}_Z$.

The category $\mhm(Z)$ comes equipped with a Tate twist functor $(1) : \mhm(Z) \to \mhm(Z)$ satisfying $\gr^H(M(1)) = (\gr^H M) \otimes u$. It will be convenient for us to introduce a square-root of the Tate twist: we define, formally,
\[ \mhm(Z)' = \{M \oplus N(\tfrac{1}{2}) \mid M, N \in \mhm(Z)\} \cong \mhm(Z) \times \mhm(Z)\]
with half-Tate twist defined by
\[ (M \oplus N(\tfrac{1}{2}))\left(\tfrac{1}{2}\right) = N(1) \oplus M(\tfrac{1}{2}).\]
We extend $\gr^H$ to a functor
\[ \gr^H : \mhm(Z)' \to \mathrm{Coh}^{\mb{G}_m}(T^*Z),\]
by setting $\gr^H(M \oplus N(\tfrac{1}{2})) = \gr^H M \oplus (\gr^H N) \otimes v^{-1}$, where in the target $\mathbb{G}_m$ is now the multiplicative group with character group $v^{\mb{Z}}$, $v = u^{-\frac{1}{2}}$, now acting by scaling the fibers with weight $v^2$.

If $f : Z \to Z'$ is a smooth morphism, then we have the functor of intermediate pullback
\[ f^\circ = f^*[\dim Z - \dim Z'] : \mhm(Z')' \to \mhm(Z)',\]
compatible with the same functor for perverse sheaves. At the level of $\gr^H$, this satisfies
\begin{equation} \label{eqn:hodge pullback}
 \gr^H f^\circ M = q_*p^*\gr^H M,
\end{equation}
where $p$ and $q$ are the morphisms
\[ T^*Z' \xleftarrow{p} T^*Z' \times_{Z'} Z \xrightarrow{q} T^*Z.\]
The assignment $Z \mapsto \mhm(Z)'$ satisfies smooth descent with respect to the pullback $f^\circ$, so it extends formally to smooth stacks. Explicitly, if $\mc{Z} = [Z/H]$ is a smooth quotient stack, a mixed Hodge module on $\mc{Z}$ is the same thing as an $H$-equivariant mixed Hodge module on $Z$: i.e., a mixed Hodge module $M$ on $Z$ equipped with an isomorphism $\mathrm{pr}_2^\circ M \cong a^\circ M$ satisfying the usual cocycle condition, where $a, \mathrm{pr}_2 : H \times Z \to Z$ are the action and projection maps. In this presentation, to be consistent with the case where $\mc{Z}$ is a scheme, we write the forgetful functor $\mhm(\mc{Z})' \to \mhm(Z)'$ as intermediate pullback $r^\circ$ along the quotient map $r: Z \to \mc{Z}$.

The functor $\gr^H$ also extends to the case of smooth stacks. In the case of a finite type quotient stack $\mc{Z} = [Z/H]$, with quotient map $r : Z \to \mc{Z}$, for $M \in \mhm(\mc{Z})'$ with associated equivariant object $r^\circ M \in \mhm(Z)'$, applying \eqref{eqn:hodge pullback} to the isomorphism $a^\circ r^\circ M \cong \mathrm{pr}_2^\circ M$ gives that
\[ \gr^H r^\circ M \in \mathrm{Coh}^{\mb{G}_m}(\mu^{-1}(0)) \subset \mathrm{Coh}^{\mb{G}_m}(T^*Z),\]
where $\mu : T^*Z \to (\operatorname{Lie} H)^*$ is the moment map. Note that the abelian category of coherent sheaves (rather than the corresponding derived category) is insensitive to the derived structure, so we may as well regard $\mu^{-1}(0)$ as a derived scheme here. The sheaf $\gr^H r^\circ M$ will also be $H$-equivariant as a coherent sheaf, so descends to an object $\gr^H M$ on $T^*\mc{Z} = [\mu^{-1}(0)/H]$. For a union of open finite type quotient stacks (such as $\mc{Y}$ above), we can define $\gr^H M$ on each finite type piece as above and glue the results.

For example, on any smooth stack, we have the object $\underline{\mb{Q}}_\mc{Z}[\dim \mc{Z}]$, defined so that $f^\circ \underline{\mb{Q}}_{\mc{Z}}[\dim \mc{Z}] = \underline{\mb{Q}}_Z[\dim Z]$ for any smooth morphism $f : Z \to \mc{Z}$ with $Z$ a scheme. In the case of a quotient stack $\mc{Z} = [Z/H]$, this corresponds to the mixed Hodge module $\underline{\mb{Q}}_Z[\dim Z]$ on $Z$ with its tautological $H$-action. By construction, we have $\gr^H \underline{\mb{Q}}_\mc{Z}[\dim \mc{Z}] = \mc{O}_\mc{Z}$, the structure sheaf of the zero section in $T^*\mc{Z}$.

We are now ready to define the standard classes in the bimodule $\B$.

\begin{defi}\label{def:standard-bimodule-classes}
For $w \in W$, consider the $w$-stratum $\mc{Y}_w = [B \backslash BwB^- / B^-] \subset \mc{Y}$. This is smooth finite type substack of dimension $-\ell(w) - \rank G$. We denote by $j_{!, w} \in \mhm(\mc{Y})'$ the shriek extension of the local system $\underline{\mb{Q}}_{\mc{Y}_w}(-\tfrac{\ell(w)}{2})[-\ell(w) - \rank G]$. The Hodge-graded $\gr^H j_{!, w}$ is a $\mathbb{G}_m$-equivariant quasi-coherent sheaf on $T^*\mathcal{Y}$ that is coherent on every open substack $T^*\mathcal{Y}_{\preceq w}$. In particular, this element defines a class in
$\B$, which we denote by $\bone_!^w$.
\end{defi}


\subsection{Bimodule structure on $\mathcal{B}$}

\subsubsection{Lattice actions}\label{SSS_lattice_actions}
The $\Ring$-module $\mathcal{B}$ carries left and right actions of the weight lattice $P$ defined by tensoring with line bundles. Let $\pi_L: \mathcal{Y} \to [*/B]$ and $\pi_R: \mathcal{Y} \to [*/B^-]$ be the canonical projections. We use the following convention for comparing the two copies of the weight lattice.  The character groups $X^*(B)$ and $X^*(B^-)$ are each identified with $P$ through the quotient to $T$; we will use this identification for $X^*(B)$ and the composition with $-\mathrm{id}_P$ for $X^*(B^-)$. Informally, the right character $\lambda$ corresponds to the left character $-\lambda$.

For $\lambda \in P$, we define the line bundles $\mathcal{L}_\lambda^L = \pi_L^* \mathcal{L}_\lambda$ and $\mathcal{L}_\lambda^R = \pi_R^* \mathcal{L}'_\lambda$, where $\mathcal{L}_\lambda$ (resp. $\mathcal{L}'_\lambda$) is the line bundle on $[*/B]$ (resp. $[*/B^-]$) associated to  $\lambda\in P$.

We define the operators $X_\lambda^L$ and $X_\lambda^R$ on $\mathcal{B}$ by:
\[
X_\lambda^L \cdot [\mathcal{F}] = [\mathcal{F} \otimes \mathcal{L}_{\lambda}^L], \quad \quad [\mathcal{F}] \cdot X_\lambda^R = [\mathcal{F} \otimes \mathcal{L}_{\lambda}^R].
\]
Since tensoring with a vector bundle preserves the support of a coherent sheaf, these operators induce well-defined continuous maps on the limit in (\ref{eqn:bdef}). { In particular, on the quotient $K_0^{\mathbb{G}_m}(T^*\mathcal{Y}_{\preceq 1})$ of $\mathcal{B}$, the operators $X_\lambda^L,X_\lambda^R$ are well-defined and our sign convention ensures that they are inverse to each other.}

\subsubsection{Coxeter actions}\label{sec:coxeter-actions}
The actions of the simple reflections $s \in W$ are defined via convolution correspondences. Let $P_s \supset B$ and $P_s^- \supset B^-$ be the minimal standard parabolic subgroups associated to $s$. We consider the proper fibrations
\[
q_s^L: \mathcal{Y} \to \mathcal{Z}_s^L = [P_s \backslash G / B^-], \quad \quad q_s^R: \mathcal{Y} \to \mathcal{Z}_s^R = [B \backslash G / P_s^-].
\]
We define the left and right operators $T_s$ by
\[
T_s^L \cdot [\mathcal{F}] = -v^{-1}(q_s^L)^* (q_s^L)_* [\mathcal{F}] + v[\mathcal{F}], \quad \quad [\mathcal{F}] \cdot T_s^R = -v^{-1}(q_s^R)^* (q_s^R)_* [\mathcal{F}] + v[\mathcal{F}],
\]
where here we abuse notation by identifying $q_s^*$ and $q_{s*}$ with the corresponding maps given by cotangent correspondences for $T^*\mathcal{Y}$. For example, for $q_s^L$, if we denote the induced cotangent correspondence by
\[ T^*\mc{Y} \xleftarrow{f} T^*\mc{Z}_s^L \times_{\mc{Z}_s^L} \mc{Y} \xrightarrow{g} T^*\mc{Z}_s^L \]
then $(q_s^L)^*$ is shorthand for $f_*\circ g^*$ and $(q_s^L)_*$ is shorthand for its right adjoint $Rg_*\circ f^!$. For the left action, given the finite-type approximation $\mathcal{Y}_{\preceq x}$, replace it if necessary by the larger open truncation $\mathcal{Y}_{\preceq w}$, where $w$ is the longer element of $\{x,sx\}$. Then $\ell(sw)<\ell(w)$, and the union of strata indexed by $y\preceq w$ is preserved by the left $P_s$-saturation, so the correspondence defining $T_s^L$ preserves $K_0^{\mathbb{G}_m}(T^*\mathcal{Y}_{\preceq w})$. For the right action, the same argument uses the longer element of $\{x,xs\}$.

\begin{Prop}
For $A \in \{L, R\}$, the operators $\{X_\lambda^A, T_s^A\}$ defined above satisfy the defining relations of the affine Hecke algebra $\Hecke_W^a$. Specifically:
\begin{enumerate}
    \item For any simple reflection $s$, the quadratic relation $(T_s^A - v)(T_s^A + v^{-1}) = 0$ holds.
    \item The braid relations are satisfied by the $T_s^A$.
    \item The Bernstein relations hold:
    \[
    X_\lambda^A T_s^A - T_s^A X_{s\lambda}^A = \hbar \frac{X_\lambda^A - X_{s\lambda}^A}{1 - X_{\alpha_s}^A}.
    \]
\end{enumerate}
\end{Prop}
\begin{proof}
{ For (1), let $\pi: E \to X$ be a $\mathbb{P}^1$-bundle. In $K_0^{\mathbb{G}_m}(T^*E)$, the operator $T = v-v^{-1}\pi^* \pi_*$ satisfies $(T-v)(T+v^{-1})=0$, cf. \cite[Theorem 7.2.5]{CG}, where again the $\mathbb{G}_m$-scaling action is by $v^2$. (1) follows.} 

For (2), let $s$ and $t$ be distinct simple reflections such that the order of $st$ in $W$ is finite, and write
$J=\{s,t\}$. By \cite[Proposition~3.13]{Kac_infdim}, the rank--two
Cartan matrix indexed by $J$ is of finite type, so the corresponding
Levi subgroup $L_J$ is finite dimensional.

The projection
\begin{equation}\label{eq:rank-two-projection-left}
q_J^L:\mathcal Y=[B\backslash G/B^-]\longrightarrow
[P_J\backslash G/B^-]
\end{equation}
has fiber
$P_J/B\simeq L_J/B_J$, where $B_J=B\cap L_J$. The cotangent
correspondences defining $T_s^L$ and $T_t^L$ are the relative versions
of the standard Steinberg correspondences for $L_J/B_J$. Under the
convolution realization of the Hecke algebra in
\cite[Theorem~7.6.10]{CG}, with the generators specified in \cite[(7.6.1)]{CG},
and after setting $q=v^2$, these give (after division by $v$)
standard generators of loc.\ cit.\ on one side and their negative inverses on the
other. Both choices satisfy the same braid relations, and hence the
braid relation satisfied by $s$ and $t$ inside $W$.

More precisely, this braid relation is an equality in the
$L_J\times\mathbb G_m$-equivariant convolution algebra. Forming
associated bundles with the $P_J$-torsor
\[
G/B^-\longrightarrow[P_J\backslash G/B^-]
\]
transports this equality to the
relative cotangent correspondences associated with
\eqref{eq:rank-two-projection-left}. This proves the braid
relation for the left action. The right action is handled in the same
way using $P_J^-$ and $B_J^-=B^-\cap L_J$.

For (3), fix a simple reflection $s$ and $\lambda\in P$, and put
$k=\langle\lambda,\alpha_s^\vee\rangle$. With the conventions
of \S\ref{SSS_lattice_actions}, the restrictions to the fibers of the
two projections are
\[
\left.\mathcal L_\lambda^L\right|_{P_s/B}
    \cong \mathcal O_{\mathbb P^1}(-k),
\qquad
\left.\mathcal L_\lambda^R\right|_{P_s^-/B^-}
    \cong \mathcal O_{\mathbb P^1}(-k).
\]
In the notation of
\cite[(7.6.36)]{CG}, these two restrictions correspond respectively to $L_{-\lambda}$ and $L_{\lambda}$.
Thus, when $k=0$, the character extends to $P_s$, respectively to
$P_s^-$, and the Bernstein relation follows from the projection formula.

For arbitrary $k$, we use the rank--one calculation of
\cite[Lemma~7.1.10]{CG}. Under the geometric realization in
\cite[(7.6.1)]{CG}, the element $e^\mu$ in the notation of loc.\ cit.\ corresponds
to $[\mathcal O_{-\mu}]$. Thus our lattice generator corresponds to
$e^\lambda$ on the left and to $e^{-\lambda}$ on the right. Moreover,
after setting $q=v^2$ and comparing the kernel in \cite[(7.6.1)]{CG}
with our normalization, the Hecke
generator corresponds to $T_s^R$ for the right action and to
$\hbar-T_s^L=-(T_s^L)^{-1}$ for the left action (after division by $v$).

This means taking $\mu=-\lambda$ in
\cite[Lemma~7.1.10]{CG} gives the right-action relation
\[
X_\lambda^R T_s^R-T_s^R X_{s\lambda}^R
  =\hbar\frac{X_\lambda^R-X_{s\lambda}^R}
  {1-X_{\alpha_s}^R}.
\]
For the left action, taking $\mu=\lambda$ gives
\[
\hbar(X_\lambda^L-X_{s\lambda}^L)
  -(X_\lambda^L T_s^L-T_s^L X_{s\lambda}^L)
  =\hbar\frac{X_\lambda^L-X_{s\lambda}^L}
  {1-X_{-\alpha_s}^L}.
\]
Since $1-1/(1-z^{-1})=1/(1-z)$, rearranging yields the left-action version of the
relation in (3) as well. Hence
\[
X_\lambda^A T_s^A-T_s^A X_{s\lambda}^A
  =\hbar\frac{X_\lambda^A-X_{s\lambda}^A}
  {1-X_{\alpha_s}^A},\qquad A\in\{L,R\}.
\]
\end{proof}

\begin{Lem}\label{Lem:bone_properties}
The following claims hold.
\begin{enumerate}
\item For any $w \in W$, we have $T_w\bone_!^1=\bone_!^1 T_w=\bone_!^w$.
\item The elements $\bone_!^w$ form a topological basis of $\mathcal{B}$ as a left topological $\Ring P$-module.
\end{enumerate}
\end{Lem}
\begin{proof}
The statement in (1) follows from the compatibility of the Hodge associated graded with smooth pullbacks and proper pushforwards (see, for example, \cite[\S 28,30]{SchnellMHM}), together with the following rank-one computation. Let $s$ be a simple reflection and observe that $\mathcal{Y}_{\preceq s}$ has a smooth atlas identified with $\mathbb{P}^1$, stratified as $\mathbb{P}^1 \cong \mathbb{A}^1 \sqcup \{\infty\}$. Pulling back to this atlas, $\bone_!^1$ becomes the class of the Hodge associated graded of the shriek extension $j_{\mathbb{A}^1!}$ of $\underline{\mb{Q}}_{\mb{A}^1}[1]$. Letting $\pi : \mathbb{P}^1 \to \{*\}$, we have $\pi^*\pi_*j_{\mathbb{A}^1!} \cong \underline{\mathbb{Q}}_{\mathbb{P}^1}(-1)[-1]$, and therefore, recalling that the Tate twist $(-1)$ becomes $v^2$ after taking $\gr^H$,
\[
(-v^{-1}\pi^*\pi_* + v)[\gr^H j_{\mathbb{A}^1!}]
= -v[\gr^H \underline{\mathbb{Q}}_{\mathbb{P}^1}[1]] + v[\gr^H j_{\mathbb{A}^1!}]
= v[\gr^H j_{\infty!}]
= \bone_!^s.
\]
Here we have used the general fact that, for $\pi : Z \to Z'$ smooth and proper,
\[ [\gr^H \pi_* M] = \pi_*[\gr^H M [\dim \pi]] \quad \text{and} \quad [\gr^H \pi^*N] = \pi^*[\gr^H N [-\dim \pi]] \]
for $M \in \mhm(Z)'$, $N \in \mhm(Z')'$. By the construction of the operators $T_s^L,T_s^R$, the expression $(-v^{-1}\pi^*\pi_* + v)[\gr^H j_{\mathbb{A}^1!}]$ coincides with both $T_s\bone^1_!$ and $\bone^1_! T_s$ finishing the proof when $w = s$. The general case follows by induction on $\ell(w)$, where for the induction step we use the calculation above regarding $\mb{P}^1$ as a smooth atlas of $\mc{Y}_w \cup \mc{Y}_{sw}$.

To prove (2), fix a finite truncation $\mathcal{Y}_{\preceq z}$. Recall, (\ref{eq:truncated_cotangent}), that $T^*\mathcal{Y}_{\preceq z}=[\mu^{-1}_z(0)/G_z]$, where $\mu_z:T^*Z_{\preceq z}\rightarrow \operatorname{Lie}(\mathcal{G}_z)^*$ is the moment map. Since $K_0$ does not depend on the derived/ non-reduced structure, we have $K_0^{\mathbb{G}_m}(T^*\mathcal{Y}_{\preceq z})=K_0^{\mathbb{G}_m\times \mathcal{G}_z}(\mu_z^{-1}(0))$, where we consider $\mu_z^{-1}(0)$ as a variety. Since $\mathcal{G}_z$ has finitely many orbits on $Z_{\preceq z}$, $\mu_z^{-1}(0)$ is the union of their conormals.
{ The orbits are homogeneous spaces for unipotent groups. Hence they and their conormal bundles are affine spaces.}
In particular, \cite[Lemma 5.5.1]{CG} applies for the filtration of $\mu_z^{-1}(0)$ by the preimages of orbits {($\mu_z^{-1}(0)$ is viewed as a fibration over a point)}. We note that each stabilizer for the action of $\mathcal{G}_z\times \mathbb{G}_m$ on $Z_{\preceq z}$ is a semidirect product of $T\times \mathbb{G}_m$ and a unipotent group. Moreover,  the class of $\bone^w_!$ is supported on the orbits labeled by elements $y$ with {$w\preceq y$ (in the Bruhat order)}, and the contribution of the orbit labeled $w$ is (up to a twist with a character of $T\times \mathbb{G}_m$) 
the class of the conormal bundle to this orbit.
Now our claim follows from  \cite[Lemma 5.5.1]{CG}.
\end{proof}

Thus we have the following.
\begin{Cor}\label{cor:levelzero-a-t-relation}
    As a { topological} $\Ring P$-module, $\mathcal{B}$ is {(topologically)} freely generated by $\{\bone_!^w\}_{w \in W}$.
We have
$$\mathcal{B}=\prod_{w} \Ring P\cdot \bone_!^w.$$
\end{Cor}

For $\beta\in\mathcal B$, we write $\operatorname{res}_w(\beta)\in\Ring P$ for
the coefficient of $\bone_!^w$ in this expansion, so that
\begin{equation}\label{eq:res_defn}
\beta=\sum_{x\in W}\operatorname{res}_x(\beta)\bone^x_!.
\end{equation}
We write an element of $\mathcal{B}$ as $(b_w)$ with $w\in W, b_w\in \Ring P\cdot \bone_!^w$.
{ We note that $\mathcal{B}$ can be viewed as} a complete topological { $\Ring$-module} equipped with the topology of the direct product.

For a finite Bruhat interval $\mathcal{J}\subset W$, we consider the neighborhood of zero $\mathcal{B}_{\mathcal{J}}=\{(b_w)| b_w=0, \forall w\in \mathcal{J}\}$ and the quotient $\mathcal{B}^{\mathcal{J}}=\mathcal{B}/\mathcal{B}_{\mathcal{J}}$ so that $\mathcal{B}=\varprojlim_{\mathcal{J}} \mathcal{B}/\mathcal{B}_{\mathcal{J}}$.

In what follows we will also need another topological basis of $\mathcal{B}$, the elements $\bone_*^w, w\in W,$ obtained similarly to $\bone_!^w$ but using the star pushforward instead of the shriek pushforward. We have $\bone_*^w=T_{w^{-1}}^{-1}\bone_*^1=\bone_*^1 T_{w^{-1}}^{-1}$.

\begin{Rem}\label{Rem:star_shriek_expansion}
We want to understand the expansion of the element $\bone_*^1$ in terms of the topological basis $\bone_!^w$ of the right $\Ring P$-module $\mathcal{B}$. Since the classes of $j_{!,w}$ form a topological basis in the Grothendieck group of the category of mixed Hodge modules on $B\backslash G/B^-$ {viewed as the module over $K_0$ of Hodge structures over a point}, we have elements $a_w\in \Ring$ with $\bone_*^1=\sum_{w\in W}a_w \bone_!^w$.
\end{Rem}

\begin{Rem}\label{Rem:finite_type}
{ When $W$ is finite, $\mathcal{B}$ is a twisted (by the automorphism corresponding to $-w_0$) regular $\mathcal{H}^a_W$-bimodule, where the twist is by the involution of $G$ giving $-w_0$ on the Cartan subalgebra, and $\bone^!_1={T^{-1}_{w_0}}$. Outside of finite type, we do not know any algebraic (i.e., without the Hodge theory) construction of $\mathcal{B}$ and its standard basis $\{\bone^!_w\}_{w \in W}$. We believe it would be very interesting to find one.}
\end{Rem}

\subsection{Twisted standard and costandard generators}\label{sec:twisted-generators}

We now introduce a collection of elements of $\mathcal{B}$ which correspond to the twisted
standard and costandard generators appearing in the Hodge-theoretic construction of
\cite[\S 3.2--3.3]{DMB}. We outline a set of relations that these prospective elements should satisfy in accordance with their geometric origins. We will make use of these elements and their relations only for heuristic purposes in \S\ref{sec:heuristic} in the final section of the paper, which will serve as a justification for our rigorous definition of the bar operation on the level-zero DAHA. Thus the purpose of this section is ultimately to explain why Definition~\ref{def:levelzero-bar} is inspired by and related to \cite{DMB}, and thus rooted in geometry.

For each real parameter \(\lambda\in P_{\mathbb R}\) and each
\(x\in W\), we introduce elements \(\Delta_x(\lambda), \nabla_x(\lambda) \in \Ha\) and  elements
\(\bone_!^x(\lambda)\in \B\), and state expected relations describing the interaction of these elements.

For a real parameter \(\lambda\), set
\[
m_i(\lambda)=\lceil\langle\lambda,\alpha_i^\vee\rangle\rceil,
\qquad i\in I.
\]
The dependence of the following elements only on these integral parts is the analogue of the
deformation dependence in \cite[Theorem~3.13(1)]{DMB}. Define
\begin{equation}\label{eq:affine-simple-factor}
\nabla_{s_i}(\lambda)
=
T_i-\hbar\,\frac{1-X_{m_i(\lambda)\alpha_i}}
{1-X_{\alpha_i}},
\qquad i\in I.
\end{equation}
For \(i\in I\), the Bernstein relation says
\[
\nabla_{s_i}(\lambda)=X_{-s_i\mu}T_iX_{\mu}\]
for $\langle\mu,\alpha_i^\vee\rangle=m_i(\lambda).$

If \(x=s_{i_1}\cdots s_{i_r}\) is a reduced expression in
\(W\), set
\begin{equation}\label{eq:affine-twisted-costandard-general}
\nabla_x(\lambda)
=
\nabla_{s_{i_1}}(s_{i_2}\cdots s_{i_r}\lambda)
\nabla_{s_{i_2}}(s_{i_3}\cdots s_{i_r}\lambda)
\cdots
\nabla_{s_{i_{r-1}}}(s_{i_r}\lambda)
\nabla_{s_{i_r}}(\lambda).
\end{equation}
A rank-two computation shows that this element is independent of the
chosen reduced expression. In particular, if
\(\ell(x_1x_2)=\ell(x_1)+\ell(x_2)\), then
\begin{equation}\label{eq:twisted-generator-multiplicativity}
\nabla_{x_1x_2}(\lambda)=\nabla_{x_1}(x_2\lambda)\nabla_{x_2}(\lambda).
\end{equation}
This is the combinatorial counterpart of the length-additive convolution
identity for the free-monodromic costandard objects in
\cite[Theorem~3.12(1)]{DMB}. Thus \(\nabla_x(0)=T_x\); this is the normalization compatible with \(T_w\bone_!^1=\bone_!^w\), which is opposite to the normalization in \cite{DMB} by which \( T_x = \Delta_x(0)\).




These elements {should} have a geometric interpretation in terms of the setup of \S\ref{sec:geometric}. Namely, let
\[
j_x:[B\backslash BxB/B]\hookrightarrow [B\backslash G/B]
\]
be the inclusion. For a real parameter \(\lambda\),
let \(\mathcal L_{x,\lambda}\) be the rank-one monodromic local
system on \([B\backslash BxB/B]\) whose right \(T\)-monodromy is \(-\lambda\)
and whose left \(T\)-monodromy is \(x\lambda\). With the same Tate-shift normalization as in
Definition~\ref{def:standard-bimodule-classes}, the element \(\nabla_x(\lambda)\) should be the
class
\[
\left[
\gr^H j_{x,*}\bigl(\mathcal L_{x,\lambda}(\ell(x)/2)[\ell(x)]\bigr)
\right]
\in K_0^{\mathbb G_m}\bigl(T^*([B\backslash G/B])\bigr).
\]
This is the associated-graded Kac--Moody analogue of { the
specialization to the closed point of the free-monodromic costandard object} \(\tilde{\nabla}_w^{(\lambda)}\) appearing in
\cite[Notation~3.1]{DMB}. We note that at this point the basics of associated graded of Hodge modules on ind-schemes have not been worked out, so the identification of \(\nabla_x(\lambda)\) with the associated graded class is not rigorous. Nevertheless, one can obtain formulas (\ref{eq:twisted-bimodule-action})
and (\ref{eq:twisted-bimodule-right-action}) below (that are inspired by Hodge theory) from the elementary definition (\ref{eq:affine-twisted-costandard-general}) similarly to the proof of Lemma \ref{Lem:bone_properties}.

We now explain analogous elements for the bimodule
\(\mathcal B\). For \(x\in W\), let
\[
j_x:\mathcal Y_x=[B\backslash Y_x]\hookrightarrow \mathcal Y
\]
be the locally closed stratum. For a real parameter \(\lambda\), let
\(\mathcal L^{\mathcal Y}_{x,\lambda}\) denote the rank-one free-monodromic local system on
\(\mathcal Y_x\) with right \(T\)-monodromy \(\lambda\) and
left monodromy \(-x\lambda\). We write
\[
\bone_!^x(\lambda)=
\left[
\gr^H j_{x,!}\bigl(\mathcal L^{\mathcal Y}_{x,\lambda}(-\ell(x)/2)[-\ell(x)]\bigr)
\right]
\in \mathcal B .
\]
This is the corresponding bimodule analogue of the free-monodromic standard
objects of \cite[Notation~3.1]{DMB}.
We normalize \(\bone_!^x(0)=\bone_!^x\).

The convolution formalism used for the
actions in \S\ref{sec:coxeter-actions}, parallel to
\cite[Theorem~3.12(1)]{DMB}, gives
\begin{align}
\label{eq:twisted-bimodule-action}
\nabla_{x_1}(-x_2\lambda)\bone_!^{x_2}(\lambda)& =\bone_!^{x_1x_2}(\lambda),\\
\label{eq:twisted-bimodule-right-action}
\bone_!^{x_1}(x_2\lambda)\nabla_{x_2}(\lambda)& =\bone_!^{x_1x_2}(\lambda)
\end{align}
for all $x_1, x_2$ with $\ell(x_1x_2) = \ell(x_1) + \ell(x_2)$. 

We also record the effect of the lattice generators on these deformed classes;
compare the translation relation in
\cite[Theorem~3.12(4)]{DMB}. With the preceding parameter convention and the
left/right character comparison fixed in \S\ref{SSS_lattice_actions}, the
corresponding bimodule relation is expected to take the form
\begin{equation}\label{eq:twisted-bimodule-parameter-translation}
X_{x\mu}\bone_!^x(\lambda)X_{\mu}
=
\bone_!^x(\lambda+\mu).
\end{equation}
Here \(\mu\in P\). These signs are also forced by
\(\nabla_x(\lambda+\mu)=X_{-x\mu}\nabla_x(\lambda)X_\mu\) and the two
convolution relations above.
In particular, at $\lambda=0$ this gives the usual formula for
the right action of the lattice part,
\begin{equation}\label{eq:twisted-bimodule-right-lattice-action}
\bone_!^xX_\mu
=
X_{-x\mu}\bone_!^x(\mu).
\end{equation}

For completeness we also keep track of the corresponding standard
elements. We denote the Hecke-side standard family by
\(\Delta_x(\lambda):=\nabla_{x^{-1}}(x\lambda)^{-1}\); in the conventions of this paper
\[
\Delta_x(0)=T_{x^{-1}}^{-1}.
\]
The analogues of the multiplicativity, translation, and
deformation-dependence relations for both \(\Delta\) and \(\nabla\) are proved
in the finite-dimensional setting in \cite[Theorems~3.12--3.13]{DMB}.

\begin{Rem}\label{Rem:alcove_dependence}
{ We would like to comment on the dependence of objects $\nabla_x(\lambda)$ and $\bone^x_!(\lambda)$ on the parameter $\lambda$. Let us start with $\nabla_x(\lambda)$, where it is easier. We consider the hyperplane arrangement $\beta^\vee=m$ in $P_{\mathbb{R}}$, where $m$ runs over
all integers, while $\beta$ runs over the real roots such that $y\beta<0$ for some $y\preceq x$.
These hyperplanes partition $P_{\mathbb{R}}$ into the union of alcoves, and for $\lambda$ in the interior of one of these alcoves, $\nabla_x(\lambda)$ only depends on the alcove. This can be seen from (\ref{eq:affine-simple-factor}) and (\ref{eq:affine-twisted-costandard-general}).

The situation with $\bone^x_!(\lambda)$ for general $\lambda$ is more complicated. Let us explain what happens for $x=1$. Pick a finite poset ideal $\mathcal{J}\subset W$. We say that a positive real root $\beta$ is {\it relevant} for $\mathcal{J}$ if $w\beta<0$ for some $w\in \mathcal{J}$. By an {\it open $\mathcal{J}$-alcove} in $P_{\mathbb{R}}$ we mean a connected component of the complement to the union of hyperplanes $\beta^\vee=m$ for relevant roots $\beta$. The Hodge-theoretic construction from \cite{DMB} shows that the projection of $\bone_!^1(\lambda)$ to $\mathcal{B}^{\mathcal{J}}$ only depends on the $\mathcal{J}$-alcove containing $\lambda$ provided $\lambda$ lies in its interior.

Now we concentrate on the case when $G$ is untwisted affine and $\lambda\in P_{0,\mathbb{R}}$, where $P_{0,\mathbb{R}}$ is the $\mathbb{R}$-span of the roots. Since the value of $\beta^\vee$ on $P_{0,\mathbb{R}}$ only depends on the projection of $\beta^\vee$ to the finite Cartan, the hyperplanes $\beta^\vee=m$ partition
$P_{0,\mathbb{R}}$ into the usual affine alcoves.
For such an alcove $A$ we use the notation $\Delta_x(A),\nabla_x(A), \bone^x_!(A)$ for the elements corresponding to $\lambda$ that lies in the relative interior of $A$.}
 \end{Rem}

In the rest of this section we will restrict to the situation when $G$ is untwisted affine, and $\lambda\in P_{0,\mathbb{R}}$ lies in the relative interior of an alcove, say $A$.
 For \(\mu\in P\cap P_{0,\mathbb R}\), equation
 (\ref{eq:twisted-bimodule-parameter-translation}) becomes
 \begin{equation}\label{eq:twisted-bimodule-alcove-translation}
X_{x\mu}\bone_!^x(A)X_{\mu}
=
\bone_!^x(A+\mu).
\end{equation}

The specialization of \eqref{eq:twisted-bimodule-action} to
\(x_2=1\) gives
\begin{equation}\label{eq:twisted-bimodule-left-identity-transport}
\nabla_w(-A)\bone_!^1(A)=\bone_!^w(A).
\end{equation}
Using (\ref{eq:twisted-bimodule-right-action}) we obtain
\begin{equation}\label{eq:twisted-bimodule-right-standard-transport}
\bone_!^1(wA)\nabla_w(A)=\bone_!^w(A),
\end{equation}
thus we obtain
\begin{equation}\label{eq:twisted-bimodule-mixed-transport}
\nabla_w(-A)\bone_!^1(A)=\bone_!^1(wA)\nabla_w(A).
\end{equation}
A specialization of \eqref{eq:twisted-bimodule-mixed-transport} serves as a key step in the level-zero computations we make in \S\ref{sec:heuristic}.

\section{The combinatorics of products in $\Ha$ and $\mathcal{B}$}\label{sec:products}

\subsection{Characterization of $\bone_!^1X_\lambda$}

From the geometric definition of the bimodule structure on $\mathcal{B}$, we obtain the following.
\begin{Lem}\label{Lem:star}
     For each $w\in W$, the element $T_w\bone_!^1 X_\lambda=\bone_!^w X_\lambda$ has an expansion of the form
     \begin{align}
         T_w\bone_!^1X_\lambda & = X_{-w\lambda}\bone_!^w+\sum_{u\succ w}F_u^\lambda \bone_!^u,
     \end{align}
     for $F_u^{\lambda} \in \Ring P$.
\end{Lem}

\begin{proof}
By construction, $\bone_!^w$ is supported on {
$\mathcal{Y}_{\succeq w}=\bigsqcup_{z\succeq w}\mathcal{Y}_z$}, as is $\bone_!^w X_\lambda$. It follows that in the
expansion of $\bone_!^w X_\lambda$ in the $\{\bone_!^u\}$ basis only indices $u\succeq w$ can occur:
\[
\bone_!^w X_\lambda \;=\; F_w^\lambda\,\bone_!^w \;+\; \sum_{u\succ w} F_u^\lambda\,\bone_!^u
\qquad (F_u^\lambda\in \Ring P).
\]
It remains to compute $\operatorname{res}_w(\bone_!^w X_\lambda)\in \Ring P$.
To do so, restrict $\mathcal L_\lambda^R=\pi_R^*\mathcal L'_\lambda$ to $\mathcal{Y}_{\preceq w}$.
The map $\pi_R:\mathcal Y\to[*/B^-]$ restricts on stabilizers as
\[
B\cap wB^-w^{-1}\hookrightarrow wB^-w^{-1}\xrightarrow{\;\mathrm{Ad}(w^{-1})\;} B^-.
\]
The character $\lambda$ of $B^-$ then pulls back along $\mathrm{Ad}(w^{-1})$ to the character $w\lambda$ of $T$,
and thus (by our sign convention for comparing characters of $B^-$ and $B$ { from the beginning of Section \ref{SSS_lattice_actions}}), we obtain the leading term $X_{-w\lambda}$ in $\Ring P$.
\end{proof}

For the next lemma we will need some standard definitions. 

\begin{defi}\label{defi:Conv}
 Define the partial pre-order $\leqslant$ on $W$-orbits in $P$ by $W\mu\leqslant W\lambda$ if $\mu\in \operatorname{Conv}(W\lambda)$.
\end{defi}

\begin{defi} \label{defi:Tits}
Let $P_+$ denote the subset of dominant weights in $P$.
    An element $\mu$ lies in the \emph{Tits cone} $\mathcal{T}$ if it is a $W$-conjugate of an element in $P_+$. This dominant element is determined uniquely by $\mu$ and is denoted by $\mu_+$. Note that $\mu\in \mathcal{T}$ is equivalent to the following condition: the number of positive real roots $\beta$ such that $\langle\mu,\beta^\vee\rangle<0$ is finite.

    An element $\mu$ in the Tits cone is called \emph{strictly Tits} if the set of positive real roots $\beta$ for which $\langle \mu, \beta^\vee\rangle \leqslant 0$ is finite. The locus of strictly Tits elements will be denoted by $\mathcal{T}^>$.
\end{defi}

{ We emphasize that the pre-order $\leqslant$ on $P$ becomes a partial order when restricted to either $\mathcal{T}$ or $-\mathcal{T}$.}

%


\begin{Lem}\label{Lem:convexity}
If $X_{\mu}\bone_?^u$ for some $\mu\in P$ and $u\in W$ occurs in the expansion of $\bone_\bullet^1 X_\lambda$, where $?,\bullet\in \{*,!\}$ {and $\lambda\in \pm \mathcal{T}$}, then $\mu\in \operatorname{Conv}(-W\lambda)$. A completely analogous result holds for $\bone_?^u X_\mu$ and $X_\lambda \bone_\bullet^1$.
\end{Lem}

\begin{proof}
We will handle the case when $?=\bullet=!$ and then explain necessary modifications in the other three cases.

Let $\Sigma$ be the set of pairs $(u, \mu)$ such that the term $X_\mu \bone_!^u$ appears with a nonzero coefficient, { to be denoted by $c_{u,\nu}$}, and $\mu \notin \operatorname{Conv}(-W\lambda)$, and suppose for contradiction that $\Sigma$ is nonempty. Choose $u$ to be a minimal element in the projection of $\Sigma$ to $W$ (with respect to the Bruhat order). Among the pairs $(u, \mu) \in \Sigma$ with this minimal $u$, choose $\mu$ so that $W\mu$ is maximal with respect to the order $\leqslant$ on $W$-orbits from Definition \ref{defi:Conv}. Note that $u \neq 1$ by Lemma \ref{Lem:star} (with $w=1$).

We apply the operator $T_{u^{-1}}$ to the expansion of $\bone_!^1 X_\lambda$:
\begin{align}
    T_{u^{-1}} (\bone_!^1 X_\lambda) & = \sum_{(w, \nu) \in \Sigma} c_{w,\nu} T_{u^{-1}} X_\nu \bone_!^w.\label{eq:lhsrhs}
\end{align}
By Lemma \ref{Lem:star},  $\operatorname{res}_1(T_{u^{-1}}\bone_!^1 X_\lambda)=0$ (where $\operatorname{res}_1$ is defined by (\ref{eq:res_defn})).
We will now analyze the contribution to the $\bone_!^1$ term from each summand $T_{u^{-1}} X_\nu \bone_!^w$ on the right-hand side of (\ref{eq:lhsrhs}).
\begin{enumerate}
\item Assume, first, that $w\not\preceq u$. { From Lemma \ref{Lem:Bernstein_consequence}, it follows that}   $T_{u^{-1}}X_{\nu}T_w$ is of the form $\sum_{u'}F_{u'}T_{u'}$ with $F_{u'}\in \Ring P$ and $F_1=0$.  Hence $w\not\preceq u$ indeed implies $F_1=0$.
    \item When $w = u$, note that
    \[
    T_{u^{-1}} X_\mu \bone_!^u = T_{u^{-1}} X_\mu T_u \bone_!^1.
    \]
    Using Lemma \ref{Lem:Bernstein_consequence}, we see that $T_{u^{-1}} X_\nu T_u=X_{u^{-1}\nu}+\ldots$, where ``$\ldots$'' signifies the sum of two kind of terms. First, there are terms that only involve weights in $\operatorname{Conv}(W\nu)\setminus W\nu$, and, second, similarly to (1), there are terms of the form $X_{\nu'}T_{u'}$ with $u'\neq 1$.
    Thus, $$\operatorname{res}_1(\sum_{\nu| (u,\nu)\in \Sigma}c_{u,\nu}T_{u^{-1}}X_\nu\bone^u_!)= \sum c_{u,\nu}X_{u^{-1}\nu}+\ldots,$$
    where the summation is taken over all $\nu\in W\mu$ and ``$\ldots$'' only includes weights in $\operatorname{Conv}(W\mu)\setminus W\mu$.
    Since $\mu \notin \operatorname{Conv}(-W\lambda)$, it follows that $u^{-1}\mu \notin \operatorname{Conv}(-W\lambda)$.
    \item If $w \prec u$, then by the minimality of $u$, any $(w, \nu) \in \Sigma$ with $w \prec u$ must have $\nu \in \operatorname{Conv}(-W\lambda)$. Thus $\operatorname{res}_1(T_{u^{-1}}\sum_{\nu| (w,\nu)\in \Sigma}c_{w,\nu}X_\nu \bone^w_!)$ only has weights in  $\operatorname{Conv}(-W\lambda)$.
\end{enumerate}

Thus, the right-hand side of (\ref{eq:lhsrhs}) contains a nonzero multiple of$X_{u^{-1}\mu}\bone_!^1$ with $u^{-1}\mu \notin \operatorname{Conv}(-W\lambda)$ from (2) above, while the left hand side contains no terms involving $\bone^!_1$. This is a contradiction, and thus we must conclude that $\Sigma$ is empty.

To handle the other three cases of $(\bullet,?)$ we argue as follows. If $?=*$, we expand the elements of $\mathcal{H}^a_W$ in the basis $X_\nu T_{w^{-1}}^{-1}$, while for $?=!$, we use the basis $X_\nu T_w$. And if $\bullet=?$, we multiply $\bone^1_\bullet X_\lambda$ by $T_{u}^{-1}$ from the left, while for $\bullet=!$, we still multiply by $T_{u^{-1}}$.
\end{proof}

The following result shows that Lemma \ref{Lem:star} uniquely characterizes an element once we fix the coefficient of $\bone_!^1$ in $\Ring P$.

\begin{Lem}\label{Lem:uniqueness_conjugation}
Suppose $x=\sum_{w\in W}F_w \bone_!^w$ with $F_w\in \Ring P$ satisfies the following:
 $T_u x=\sum_{w\neq 1}F_w^u \bone_!^w$ with $F_w^u\in \Ring P$ for all $u\in W$.
Then $x=0$. The same conclusion holds if we replace $T_{u}$ with $T_u^{-1}$ or $!$ with $*$ in (2), and if we reverse the order of the products.
\end{Lem}
\begin{proof}
Take a minimal $u$ with $F_u\neq 0$, {then show that $\operatorname{Res}_1{T_{u^{-1}}}x\neq 0$ arguing} as in the proof of Lemma \ref{Lem:convexity}.
\end{proof}

\begin{Rem}\label{Rem:uniqueness}
Let us spell out the uniqueness property for $\bone_!^1 X_\lambda$ explicitly: this is the unique element $\sum_{w\in W}F_w\bone_!^w\in\mathcal{B}$ with the following properties:
\begin{itemize}
\item[(!1)] $F_1=X_{-\lambda}$,
\item[(!2)]  $T_u\sum_{w\in W}F_w\bone_!^w=\sum_{w\neq 1}F^u_w\bone_!^w$ for all $u\in W\setminus \{1\}$.
\end{itemize}
Similarly, $\bone_*^1 X_\lambda$
is the unique element $\sum_{w\in W}F'_w\bone_!^w\in\mathcal{B}$ with the following properties:
\begin{itemize}
\item[($*1$)] $F'_1=X_{-\lambda}$,
\item[($*2$)]  $T_{u^{-1}}^{-1}\sum_{w\in W}F'_w\bone_!^w=\sum_{w\neq 1}F^u_w \bone^w_*$ for all $u\in W\setminus \{1\}$.
\end{itemize}
We also note that similar characterizations in terms of the multiplications from the right work for $X_\lambda \bone_!^1$ and $X_\lambda\bone_*^1$.
\end{Rem}

For $\gamma=\sum_i n_i\alpha_i\in Q$, set
$\operatorname{ht}(\gamma)=\sum_i n_i$. Suppose $\mu,\mu'\in P$ are $W$-conjugates of dominant elements, $\mu_+,\mu'_+$, and $\mu'\in \mu+Q$. Then we write
$d(\mu,\mu')=\operatorname{ht}(\mu_+-\mu'_+)$. We note that for $\mu'\in \operatorname{Conv}(W\mu)$ we have $d(\mu,\mu')\geqslant 0$ with the equality if and only if $\mu'\in W\mu$.

\begin{Rem}\label{Rem:uniqueness_approx}
We will need an ``approximate version'' of Remark \ref{Rem:uniqueness} proved by similar methods. Assume that $\lambda$ is dominant and $d\in \Z_{\geqslant 0}$. We want to investigate expressions of the form
$$X_{-\lambda} \bone^1_!=\sum_{w\in W}\bone^w_*(F_w+G_w),$$
where $F_w,G_w\in \Ring \mathcal{T}$, and all monomials $X_\nu$ that appear in $G_w$ satisfy $\nu\in \operatorname{Conv}(W\lambda)$ and $d(\lambda,\nu)>d$. Then $\sum_{w\in W}\bone^w_*F_w$ satisfies the following conditions:
\begin{itemize}
\item[(1!)] $F_1-X_\lambda$ is an $\Ring$-linear combination of monomials $X_\nu$ with $W\nu\leqslant W\lambda$ and $d(\lambda,\nu)>d$,
\item[(2!)] For all $u\neq 1$, for the expansion $\sum_{w}\bone^w_* F_wT_u=\sum_{w\in W}\bone^w_* F^u_w$, the coefficient $F^u_1$ is the $\Ring$-linear combination of monomials $X_\nu$ with $W\nu\leqslant W\lambda$ and $d(\lambda,\nu)>d$.
\end{itemize}
\end{Rem}

\subsection{Products in $\Hecke^a_W$}\label{sec:productsinh}
Our goal in this section is to analyze the expansion of products of the form $T_w^{-1} X_\mu T_{u^{-1}}$ in terms of the basis elements $T_{(w')^{-1}}^{-1}X_{\mu'}$ of $\Hecke^a_W$ in the case when $\mu$ is strictly Tits, see Definition \ref{defi:Tits}.

The next result is straightforward.

\begin{Lem}\label{Lem:strictly_Tits_properties}
The following claims are true:
\begin{enumerate}
\item Let $\lambda_1\in \mathcal{T}^>, \lambda_2\in \mathcal{T}$ and $\alpha_1,\alpha_2$ be positive rational numbers such that $\alpha_1\lambda_1+\alpha_2\lambda_2\in P$. Then $\alpha_1\lambda_1+\alpha_2\lambda_2\in \mathcal{T}^>$.
\item $\mathcal{T}^>$ is $W$-stable.
\end{enumerate}
\end{Lem}

For example, when $G$ is untwisted affine, we have $\mathcal{T}^>=\mathcal{T}\setminus \mathbb{Z}\delta$. The case where $\mu\in \Z\delta$ is not particularly interesting as the corresponding element $X_\mu$ is central.

We introduce some notation. Lemma \ref{Lem:Bernstein_consequence} shows that if $T_{(w')^{-1}}^{-1}X_{\mu'}$ appears in the expansion of $T_{w^{-1}}^{-1}X_\mu T_{u^{-1}}$, then $\mu'\in \operatorname{Conv}(W\mu)$, in particular, $\mu'\in \mathcal{T}$ if $\mu\in \mathcal{T}$.

The following are two technical results that we aim to prove.

\begin{Prop}\label{Prop:weights_occuring1}
Let $\mu$ be strictly Tits.
For each $d>0$, there is a finite set $\Sigma(\mu,d)$ such that for all monomials $T_{(w')^{-1}}^{-1}X_{\mu'}$ with $d(\mu,\mu')\leqslant d$ that occur with nonzero coefficient in the expansion of $T_{w^{-1}}^{-1}X_\mu T_{u^{-1}}$ we have $\mu'\in u \Sigma(\mu,d)$.
\end{Prop}

\begin{Prop}\label{Prop:weights_occuring2}
Let $\mu$ be strictly Tits. For each $w\in W, d>0$,
there is a finite set $\Sigma'(w,\mu,d)$ such that for all $\mu'$ with $d(\mu, \mu') \leqslant d$ {such that $T_1X_{\mu'}$ occurs} with nonzero coefficient in the expansion of $T_{w^{-1}}^{-1} X_\mu T_{u^{-1}}$ we have $\mu'\in \Sigma'(w,\mu,d)$.
\end{Prop}

The important feature is that neither of the finite sets $\Sigma'(w,\mu,d),\Sigma(\mu,d)$ in these propositions depends on $u$.

We start by providing formulas for $T_{w^{-1}}^{-1} X_\lambda T_s$ for a simple reflection $s$.
\begin{Lem}\label{Lem:simplecases} For any $w \in W$, $\lambda \in P$, and simple reflection $s$, set $k=\langle \lambda,\alpha_s^\vee\rangle$ and $\alpha=\alpha_s$. Then
    \begin{equation}\label{eq:basis_expansion}
T_{w^{-1}}^{-1}X_\lambda T_s=T_{(ws)^{-1}}^{-1}X_{s\lambda}+\hbar T_{w^{-1}}^{-1}\begin{cases} -\sum_{i=1}^{k-1}X_{s\lambda+i\alpha}& \text{if }k>0, \ell(ws)>\ell(w),\\
-\sum_{i=0}^{k-1}X_{s\lambda+i\alpha}& \text{if }k\geqslant 0, \ell(ws)<\ell(w),\\
\sum_{i=k}^0 X_{s\lambda+i\alpha}&\text{if }{k\leqslant 0}, \ell(ws)>\ell(w),\\
\sum_{i=k}^{-1} X_{s\lambda+i\alpha}&\text{if }{k< 0}, \ell(ws)<\ell(w).
\end{cases}
\end{equation}
\end{Lem}
\begin{proof}
By (\ref{eq:Bernstein_alternative}),

\begin{equation}\label{eq:Bernstein_new}
\begin{alignedat}{2}
X_\lambda T_s
&=T_sX_{s\lambda}
-\hbar\sum_{i=0}^{k-1}X_{s\lambda+i\alpha}
=T_s^{-1}X_{s\lambda}-\hbar\sum_{i=1}^{k-1} X_{s\lambda+i\alpha},
&\qquad& \text{if } k>0,\\
&=T_sX_{s\lambda}=T_s^{-1} X_{s\lambda}+\hbar X_{s\lambda},
&& \text{if } k=0,\\
&=T_sX_{s\lambda}+\hbar\sum_{i=k}^{-1}X_{s\lambda+i\alpha}
=T_s^{-1}X_{s\lambda}+\hbar\sum_{i=k}^{0}X_{s\lambda+i\alpha},
&& \text{if } k<0.
\end{alignedat}
\end{equation}

Combining this with $T_{w^{-1}}^{-1}T_s^{-1}=T_{(ws)^{-1}}^{-1}$ if $\ell(ws)=\ell(w)+1$ and $T_{w^{-1}}^{-1}T_s^{-1}=-\hbar T_{w^{-1}}^{-1}+T_{(ws)^{-1}}^{-1}$ otherwise, we arrive at the formula (\ref{eq:basis_expansion}).
\end{proof}

Now take $u\in W$ with reduced expression $s_\ell s_{\ell-1}\ldots s_1$. Let $\alpha_j$ denote the simple root corresponding to $s_j$. Set $\beta_j=s_1s_2\ldots s_{j-1}(\alpha_j)$. The roots $\beta_1,\ldots,\beta_\ell$ are pairwise distinct positive roots and are exactly the positive roots $\beta$ satisfying $u\beta<0$.

We now describe the term of the form $T_{(w')^{-1}}^{-1}F_{w'}$ with $F_{w'}\in \Ring P$ in the expansion of $T_{w^{-1}}^{-1} X_\mu T_{u^{-1}}$. This will require the following combinatorial setup, which we will use to index contributions to this term.

\begin{defi}
    Denote by $\underline{j} = (j_1, j_2, \dots, j_s)$ a strictly increasing sequence of elements in $\{1, \dots, \ell\}$. We can then write
    \begin{align}
        u = u_ss_{j_s}u_{s-1} \dots u_1s_{j_1}u_0,
    \end{align}
    defining elements $u_0, \dots, u_s$ and $w_0, \dots, w_s$ by
    \begin{align*}
        u_0 & = s_{j_1 - 1} \dots s_1 & w_0 & = wu_0^{-1}\\
        u_1 & = s_{j_2 - 1} \dots s_{j_1 + 1} & w_1 & = wu_0^{-1}u_1^{-1}\\[-1ex]
        \vdots & & \vdots & \\
        u_{s} & = s_\ell \dots s_{j_s + 1} & w_s & = wu_0^{-1} \dots u_s^{-1}.
    \end{align*}
    We then define the element $w(u, \underline{j}) = w_s$. We let $J(u, w')$ denote the set of sequences $\underline{j}$ for which \begin{equation}\label{eq:Weyl_element_decomposition}
w'=w(u, \underline{j}).
\end{equation}
\end{defi}

The monomial contributions to $F_{w'}$ we will describe are indexed by a choice of $\underline{j} \in J(u, w')$ along with another piece of combinatorial data which we now introduce.

\begin{defi}
    Given $\underline{j} \in J(u, w')$, let $I(u, \underline{j})$ be the set of tuples  $(i_1,\ldots,i_r)\in \Z^r$ which can be inductively obtained according to the following procedure.

    First, let $\lambda_0 = u_0\mu$. Now assume for induction that $\lambda_{r-1}$ is already defined. Then pick $i_r$ in the interval:
    \begin{equation}\label{eq:intervals}
\begin{split}
&
{[1,\langle \lambda_{r-1}, \alpha_{j_r}^\vee\rangle-1]}, \text{ if } \langle \lambda_{r-1}, \alpha_{j_r}^\vee\rangle> 0,\, \ell(w_{r-1}s_{j_r})>\ell(w_{r-1}),\\
& {[0,\langle \lambda_{r-1}, \alpha_{j_r}^\vee\rangle-1]}, \text{ if } \langle \lambda_{r-1}, \alpha_{j_r}^\vee\rangle\geqslant  0, \,\ell(w_{r-1}s_{j_r})<\ell(w_{r-1}),\\
&{[\langle\lambda_{r-1},\alpha_{j_r}^\vee\rangle,0]}, \text{ if }
\langle \lambda_{r-1}, \alpha_{j_r}^\vee\rangle{\leqslant 0},\, \ell(w_{r-1}s_{j_r})>\ell(w_{r-1}),\\
&{[\langle\lambda_{r-1},\alpha_{j_r}^\vee\rangle,-1]}, \text{ if }
\langle \lambda_{r-1}, \alpha_{j_r}^\vee\rangle{< 0},\, \ell(w_{r-1}s_{j_r})<\ell(w_{r-1}).
\end{split}
\end{equation}
Then set \begin{equation}\label{eq:lambda_r_formula}
\lambda_r=u_r(s_{j_r}\lambda_{r-1}{+}i_r\alpha_{j_r}).
\end{equation}
We let $I(u, w')$ denote the set of pairs $(\underline{i}, \underline{j})$ with $\underline{j} \in J(u, w')$ and $\underline{i} \in I(u, \underline{j})$.
\end{defi}
With these definitions in hand, we notice that by induction on Lemma \ref{Lem:simplecases}, we obtain the following.

\begin{Lem}\label{Lem:contributions}
    Let $T_{(w')^{-1}}^{-1}F_{w'}$ with $F_{w'} \in \Ring P$ be the corresponding term in the expansion of $T_{w^{-1}}^{-1} X_\mu {T_{u^{-1}}}$. Then $F_{w'}$ is itself the sum of terms of the form
    \begin{align}
        \pm \hbar^s X_{\lambda_s}
    \end{align}
    for $(\underline{i}, \underline{j}) \in I(u, w')$.
\end{Lem}
\begin{proof}
    One can use induction on the length of $u$ by Lemma \ref{Lem:simplecases} to expand the expression $T_{w^{-1}}^{-1} X_\mu T_{u^{-1}}$. This induction (whose base case and inductive step both follow from comparing equations (\ref{eq:basis_expansion}) and (\ref{eq:intervals})) reveals that the contribution of the tuples $(j_1, \dots, j_s)$ and $(i_1, \dots, i_s)$ in the resulting expansion is $\pm \hbar^s X_{\lambda_s}$, and so we get the result by the definition of $I(u, w')$.
\end{proof}

Further, we note the following formula for $\lambda_s$. Notice that $$u_ss_{j_s}\ldots u_r\alpha_{j_r}=u_ss_{j_s}\ldots u_rs_{j_r}\ldots s_1 s_1\ldots s_{j_r}\alpha_{j_r}=-u\beta_{j_r},$$
so it follows that
\begin{equation}\label{eq:lambda_s_formula1}
\lambda_s=u(\mu-\sum_{r=1}^s i_r\beta_{j_r})
\end{equation}

Our next task is to examine all possibilities to have nonzero $i_r$ in this formula under the restriction that $d(\mu,\lambda_s)\leqslant d$.

\begin{Lem}\label{Lem:root_finiteness} {Let $d>0$.}
Let $\mu$ be a strictly Tits element.
 The set of pairs $(i,\beta)$, where $i\in \Z_{>0}$ and $\beta$ is a real root such that $\mu-i\beta\in \operatorname{Conv}(W\mu)$ and $d(\mu,\mu-i\beta)\in \{1,\ldots,d\}$ is finite.
\end{Lem}
\begin{proof}
Suppose $\operatorname{ht}(\mu_+-(\mu-i\beta)_+)\in \{1,\ldots,d\}$. Then we have that $\mu-i\beta\in \operatorname{Conv}(W\mu)\setminus W\mu$. Indeed, otherwise $\mu$ lies in the convex hull of $W(\mu-i\beta)$ and $\operatorname{ht}(\mu_+-(\mu-i\beta)_+)\leqslant 0$.

We can assume without loss of generality that $\mu$ is dominant. Then $\beta$ must be positive. 

{We claim that it is enough to prove the finiteness of the number of pairs $(i,\beta)$ as in the statement of the lemma with  $i=1$. Indeed, we have $\mu-(i+1)\beta\in \operatorname{Conv}(W(\mu-i\beta))$ for $i<\frac{1}{2}\langle\mu,\beta^\vee\rangle$. We have $\mu-i\beta\in \operatorname{Conv}(W\mu)\setminus W\mu$ if and only if $i\in [1,\langle\mu,\beta^\vee\rangle-1]$ and $\mu-i\beta$ and $s_\beta\mu+i\beta$ lie in the same $W$-orbit. So if $(i,\beta)$ satisfies the conditions of the lemma, then $(1,\beta)$ does finishing the reduction.}

Consider a $W$-invariant form $(\cdot,\cdot)$ on $P$ that is positive on the real roots. Then there are only finitely many options for $(\beta,\beta)$. Note that there are only finitely many options for $W(\mu-\beta)$, hence for $(\mu-\beta,\mu-\beta)$, hence for $(\mu,\beta)$. Let $I_1$ be the set of all labels $i$ of simple roots $\alpha_i$ such that $(\mu,\alpha_i)>0$ and $I_2$ be the complement, i.e.\ the set of $i$ such that $(\mu,\alpha_i)=0$. Then $\beta=\sum_{i\in I}m_i \alpha_i$ and $(m_i)_{i\in I_1}$ is in a finite set. Write
\[
\beta_1=\sum_{i\in I_1}m_i\alpha_i,\qquad
\beta_2=\sum_{i\in I_2}m_i\alpha_i.
\]
It remains to show that, for each possible $\beta_1$, only finitely many $\beta_2$ can occur. Since $\mu$ is dominant and strictly Tits, the real root subsystem generated by $I_2$ is finite: otherwise all its positive real roots $\gamma$ would satisfy $\langle\mu,\gamma^\vee\rangle=0$, contradicting $\mu\in \mathcal{T}^>$. Hence the restriction of $(\cdot,\cdot)$ to the span of $\{\alpha_i\mid i\in I_2\}$ is positive definite. For fixed $\beta_1$, the identity
\[
(\beta_2,\beta_2)+2(\beta_1,\beta_2)=(\beta,\beta)-(\beta_1,\beta_1)
\]
has only finitely many solutions $\beta_2$ in the lattice spanned by $\{\alpha_i\mid i\in I_2\}$, because the left-hand side is a positive-definite quadratic function plus a linear term and the right-hand side takes only finitely many values. Thus only finitely many $\beta$ can occur.
\end{proof}

\begin{proof}[Proof of Proposition \ref{Prop:weights_occuring1}]
Fix $\rho^\vee\in \mathfrak{t}_{\mathbb{Q}}$ such that $\langle\rho^\vee,-\rangle$ is $1$ on all simple roots, so that $\operatorname{ht}(\mu)=\langle\rho^\vee,\mu\rangle$ for all $\mu$ in the root lattice, and $d(\mu,\lambda)=\langle\rho^\vee,\mu_+-\lambda_+\rangle$.

Because of Lemma \ref{Lem:contributions} and equation (\ref{eq:lambda_s_formula1}), our task is to show that the set
\begin{equation}\label{eq:finite_set}
\{\lambda=\mu-\sum_{r=1}^s i_r\beta_{j_r}| 0\leqslant d(\mu,\lambda)\leqslant d\}
\end{equation}
is finite.
In the proof it is enough to consider the cases when $i_r\neq 0$.

We will show that
\begin{itemize}
\item[($\heartsuit$)]
for each $m$ we have either
\begin{equation}\label{eq:heartsuit1}
\langle\rho^\vee,(\mu-\sum_{r=1}^{m}i_r\beta_{j_r})_+\rangle>
\langle\rho^\vee,(\mu-\sum_{r=1}^{m+1}i_r\beta_{j_r})_+\rangle\end{equation}
or
\begin{equation}\label{eq:heartsuit2}
\mu-\sum_{r=1}^{m+1}i_r\beta_{j_r}\in W(\mu-\sum_{r=1}^m i_r\beta_{j_r}) \text{ and }\mu-\sum_{r=1}^{m+1}i_r\beta_{j_r}>\mu-\sum_{r=1}^m i_r\beta_{j_r}.
\end{equation}
\end{itemize}
To show ($\heartsuit$) note that
$$\langle\rho^\vee, (\mu-\sum_{r=1}^{m}i_r\beta_{j_r})_+\rangle\geqslant
\langle\rho^\vee, (\mu-\sum_{r=1}^{m+1}i_r\beta_{j_r})_+\rangle.$$
We need to understand what happens if we have the equality. In this case we see that
 $\mu-\sum_{r=1}^{m}i_r\beta_{j_r}$ and $\mu-\sum_{r=1}^{m+1}i_r\beta_{j_r}$ are in the same $W$-orbit. Furthermore, (\ref{eq:intervals}) implies that
\[
i_{m+1}=\langle\lambda_m, \alpha_{j_{m+1}}^\vee\rangle<0.
\]
Hence $\langle\mu-\sum_{r=1}^m i_r \beta_{j_r},\beta_{j_{m+1}}^\vee\rangle<0$ and since the root $\beta_{j_{m+1}}$ is positive, (\ref{eq:heartsuit2}) holds. This proves ($\heartsuit$).

Now we use ($\heartsuit$) to prove that (\ref{eq:finite_set}) is finite. In order for $\lambda$ to satisfy $0\leqslant d(\mu,\lambda)\leqslant d$, situation (\ref{eq:heartsuit1}) can occur at most $d$ times. For each $m$ where it occurs, by Lemma \ref{Lem:root_finiteness}, there are only finitely many options for $\beta_{j_{m+1}}$. So, we just need to show that the occurrences of (\ref{eq:heartsuit2}) between two consecutive occurrences of (\ref{eq:heartsuit1}) (or after the last occurrence of (\ref{eq:heartsuit1})) yield a finite set of weights   $\mu-\sum_{r=1}^{m+1}i_r\beta_{j_r}$. This is a consequence of the following general fact: 

for each dominant $\nu$ and each $w\in W$, the set $\{\nu'\in W\nu| \nu'\geqslant w\nu\}$ is finite. 

Indeed, we can present $\nu'$ as $s_{k_1}\ldots s_{k_q}\nu$ with $\langle\alpha_{k_i},s_{k_{i+1}}\ldots s_q\nu \rangle>0$ for all $i$. Then we must have $q\leqslant \operatorname{ht}(\nu-w\nu)$, which yields the finiteness claim.

This finishes the proof of the finiteness of (\ref{eq:finite_set}) and hence of the proposition.

\end{proof}

In order to prove Proposition \ref{Prop:weights_occuring2} we will need to rewrite formulas (\ref{eq:lambda_r_formula}) and
(\ref{eq:lambda_s_formula1}). Namely, set
$$i'_r=\langle\lambda_{r-1},\alpha^\vee_{j_r}\rangle{-}i_r.$$
Then (\ref{eq:lambda_r_formula}) is equivalent to
$$\lambda_r=u_r(\lambda_{r-1}-i'_r\alpha_{j_r}).$$
Assume $w'=1$, equivalently, by (\ref{eq:Weyl_element_decomposition}),
\begin{equation}
w=u_s u_{s-1}\ldots u_0.
\end{equation}
We can now rewrite (\ref{eq:lambda_s_formula1}) as
\begin{equation}\label{eq:lambda_s_formula2}
\lambda_s=w\mu-\sum_{r=1}^s i'_r u_su_{s-1}\ldots u_r\alpha_{j_r}.
\end{equation}

\begin{proof}[Proof of Proposition \ref{Prop:weights_occuring2}]
We claim that for each $r=1,\ldots,{s}$ exactly one of the following possibilities holds:
\begin{enumerate}[(i)]
\item $i'_r=0$,\label{cond:first}
\item $d({\lambda_{r-1},\lambda_{r}})>0$,\label{cond:second}
\item $W\lambda_r={W\lambda_{r-1}}, i'_r\neq 0$, and $\operatorname{sgn}(i'_r)u_s\ldots u_r \alpha_{j_r}$ is a negative root. \label{cond:third}
\end{enumerate}
Indeed, condition (\ref{cond:second}) occurs exactly when $i'_r\neq 0$ and $i'_r\neq \langle\lambda_{r-1},\alpha_{j_r}^\vee\rangle$. So it remains to show that all cases where $i'_r=\langle\lambda_{r-1},\alpha_{j_r}^\vee\rangle\neq 0$ (equivalently, $i_r=0$) satisfy (\ref{cond:third}). By (\ref{eq:intervals}), this can happen in two cases.

The first case is when $i'_r=\langle \lambda_{r-1}, \alpha^\vee_{j_r}\rangle< 0$. Here $\ell(w_{r-1}s_{j_r})>\ell(w_{r-1})$, equivalently (because $1=w_s=w_{r-1}u_r^{-1}\ldots u_s^{-1}$), $\ell(u_s\ldots u_r s_{j_r})>\ell(u_s\ldots u_r)$, equivalently, $u_s\ldots u_r \alpha_{j_r}>0$. So, the conclusion of  (\ref{cond:third}) holds in this case.

The second case is when $i'_r=\langle \lambda_{r-1}, \alpha^\vee_{j_r}\rangle>0$ and $\ell(w_{r-1}s_{j_r})<\ell(w_{r-1})$. We see that  (\ref{cond:third}) holds similarly to the previous paragraph.

So (\ref{cond:first})--(\ref{cond:third}) exhaust all options. We claim that the set of possible right-hand sides of (\ref{eq:lambda_s_formula2}) is finite. The proof repeats the analogous part in the proof of Proposition \ref{Prop:weights_occuring1}: deducing the finiteness of (\ref{eq:finite_set}) from ($\heartsuit$).
\end{proof}

\begin{Rem}\label{Rem:weights_occuring1}
We note that the conclusion of Proposition \ref{Prop:weights_occuring1} remains valid if instead of the basis $T_{(w')^{-1}}^{-1}X_{\mu'}$ we consider the basis $T_{w'}X_{\mu'}$. On the other hand, the conclusion of Proposition \ref{Prop:weights_occuring2} is not true in that basis.
\end{Rem}
\section{Completions of $\mathcal{H}_W^a$}\label{sec:completions}

In this section, we define two completions of certain (non-unital) subalgebras of $\Ha$ relevant for the bar operation. We begin with the \emph{small completion}, a topological algebra whose most important property is as follows: for every $\lambda\in \mathcal{T}^>$ (the strict Tits cone), there is an element $\omega(X_\lambda)$ in the small completion satisfying $X_{-\lambda} \bone^1_!=\bone^1_* \omega(X_\lambda)$.

We then proceed to the strictly Tits part of $\mathcal{H}^a_W$ and its \emph{large completion}. The latter is a topological algebra on which the bar operation makes sense and gives a ring automorphism.

\subsection{Small completion}
In this subsection, we define the small completion and establish its properties.

\begin{defi}
    Let $\Upsilon$ be the ring automorphism of $\Ha$ defined by $$\Upsilon(v)=v^{-1}, \Upsilon(T_w)=T^{-1}_{w^{-1}},\Upsilon(X_\lambda)=X_{-\lambda}.$$
\end{defi}

\begin{defi}
    Let $\Hecke_W^+$ be the $\Ring$-subalgebra of $\Ha$ generated by $T_w$ and $X_\lambda$ with $w \in W$, $\lambda\in \mathcal{T}$. Let $\Hecke_W^-=\Upsilon(\Hecke_W^+)$; this is the $\Ring$-subalgebra of $\Ha$ generated by $T_w$ and $X_{-\lambda}$ for $w \in W$, $\lambda\in \mathcal{T}$.
\end{defi}

The small completion ${\breve{\Hecke}}_W^+$ will be a topological $\Ring$-algebra which is a completion of $\Hecke_W^+$. We will also define a modified version
$\breve{\B}^+$ of $\mathcal{B}$ which is a $\breve{\Hecke}_W^+$-bimodule. Recall that we have a partial order on the $W$-orbits in $\mathcal{T}$: $W\mu\leqslant W\mu'$ if $W\mu\subset \operatorname{Conv}(W\mu')$,
Definition \ref{defi:Conv}. Also recall $\rho^\vee\in \mathfrak{t}_{\mathbb{Q}}$ such that $\langle\rho^\vee,-\rangle$ is $1$ on all simple roots.




\begin{defi}\label{defi:small_completion}
For $\vec{\mu}=(\mu_1,\ldots,\mu_k)\in P_+^k$ (where $P_+$ is the set of dominant weights)
consider the $\Ring$-submodule $\Hecke^+_{\leqslant W\vec{\mu}}$
consisting of all elements of the form $\sum_{w,\mu} a_{w,\mu}T_w X_\mu$ with the following property:
\begin{itemize}
\item[(I)] if $a_{w,\mu}\neq 0$, then $W\mu\leqslant W\mu_i$ for some $i=1,\ldots,k$.
\end{itemize}
We equip $\Hecke^+_{\leqslant W\vec{\mu}}$ with an $\Ring$-module topology  where the basis of neighborhood of $0$ is given by
$$U_d:=\{\sum_{w,\mu}a_{w,\mu}T_w X_\mu\in \Hecke^+_{\leqslant W\vec{\mu}}| a_{w,\mu}\neq 0\Rightarrow \langle\rho^\vee,\mu_+\rangle\leqslant d\}, d\in \Z.$$
We write $\breve{\Hecke}^+_{\leqslant W\vec{\mu}}$ for the completion of $\Hecke^+_{\leqslant W\vec{\mu}}$ with respect to this topology. By definition, it consists of all infinite sums
$\sum_{w,\mu}a_{w,\mu}T_w X_\mu$ satisfying (I) and the following condition:
\begin{itemize}
\item[(II)] For each $d\in \Z$ there are only finitely many pairs $(w,\mu)$ such that $a_{w,\mu}\neq 0$ and $\langle\rho^\vee, \mu_+\rangle\geqslant d$.
\end{itemize}
Note that all $\breve{\Hecke}^+_{\leqslant W\vec{\mu}}$ are subsets in the set of all infinite sums $\sum_{w,\mu}a_{w,\mu}T_wX_\mu$. Let $\breve{\Hecke}^+_W$ denote the union of all such subsets. This is the {\it small completion} of $\Hecke^+_W$.

\end{defi}

\begin{Lem}\label{Lem:small_completion_algebra}
The following claims hold:
\begin{enumerate}
\item For $\vec{\mu}=(\mu_1,\ldots,\mu_k), \vec{\mu'}=(\mu_1',\ldots,\mu'_\ell)$, two collections of dominant weights, set $$\vec{\mu}*\vec{\mu}'=(\mu_i+\mu'_j)_{i=1,\ldots,k; j=1,\ldots,\ell}.$$
Then the product $\Hecke^+_W\times \Hecke^+_W\rightarrow \Hecke^+_W$ restricts to
\begin{equation}\label{eq:resticted_product}
\Hecke^+_{\leqslant W\vec{\mu}}\times \Hecke^+_{\leqslant W\vec{\mu}'}\rightarrow \Hecke^+_{\leqslant W(\vec{\mu}*\vec{\mu}')}.
\end{equation}
\item (\ref{eq:resticted_product}) is continuous with respect to the topology introduced in
Definition \ref{defi:small_completion}.
\item $\breve{\Hecke}^+_W$ has the unique topological algebra structure extending the algebra structure on $\Hecke^+_W$ by continuity. This algebra structure is associative and unital.
\end{enumerate}
\end{Lem}
\begin{proof}
(1) and (2) follow from Lemma \ref{Lem:Bernstein_consequence}.   (3) immediately follows from (2).
\end{proof}

Since $\Ring \mathcal{T}$ is a subset of $\Hecke^+_W$, we can define $\Ring\mathcal{T}_{\leqslant W\vec{\mu}}$ as
$\Ring\mathcal{T}\cap \Hecke^+_{\leqslant W\vec{\mu}}$. Let $\Ring\breve{\mathcal{T}}$ denote the completion.

\begin{defi}\label{defi:completed_bimodule}
For $\vec{\mu}$ as in Definition \ref{defi:small_completion},
define $\B^+_{\leqslant W\vec{\mu}}$ as the subset of $\B$ consisting of all sums $\sum_{w,\mu}a_{w,\mu}\bone^!_w X_\mu$ satisfying (I) from Definition \ref{defi:small_completion}, so that $\B^+_{\leqslant W\vec{\mu}}$ is identified with $\prod_{W} \Ring \mathcal{T}_{\leqslant W\vec{\mu}}$ and we equip it with the topology of direct product of topological spaces. We define  $\breve{\B}^+_{\leqslant W\vec{\mu}}$ as the completion of this topological $\Ring$-module. In other words, an element of $\breve{\B}^+_{\leqslant W\vec{\mu}}$ is an infinite sum $\sum_{w,\mu} a_{w,\mu}\bone^!_w X_\mu$ satisfying (I) and (II) for each individual $w$. We write $\B^+$ (resp., $\breve{\B}^+$) for the union of all $\B^+_{\leqslant W\vec{\mu}}$ (resp., $\breve{\B}^+_{\leqslant W\vec{\mu}}$).


\end{defi}

For example, $\bone_*^1$ lies in $\B^+_{\leqslant 0}$ thanks to Remark \ref{Rem:star_shriek_expansion}.

\begin{Prop}\label{Prop:small_completion_bimodule}
The following claims hold:
\begin{enumerate}
\item For each $\vec{\mu},\vec{\mu}',\vec{\mu}''$, the twisted action map
$$\Hecke^a_W\times\B\times \Hecke^a_W\rightarrow \B,(a_1,b,a_2)\mapsto \Upsilon(a_1)ba_2,$$ restricts to
\begin{equation}\label{eq:restricted_bimodule_map}
\Hecke^+_{\leqslant W\vec{\mu}}\times\B^+_{\leqslant W\vec{\mu}'}\times \Hecke^+_{\leqslant W\vec{\mu}''}\rightarrow \B^+_{\leqslant W(\vec{\mu}*\vec{\mu}'*\vec{\mu}'')}
\end{equation}
\item (\ref{eq:restricted_bimodule_map}) is continuous.
\item $\breve{\B}^+$ is a topological $\breve{\Hecke}^+$-bimodule via $(a_1,b,a_2)\mapsto \Upsilon(a_1)ba_2$.
\end{enumerate}
\end{Prop}
\begin{proof}
Similar to that of Lemma \ref{Lem:small_completion_algebra} but in (1) and (2) one also uses Lemma \ref{Lem:convexity}.
\end{proof}

Our main result related to these constructions is the next proposition. We define $\Hecke^>_W$ as the $\Ring$-linear span of the elements of the form $T_w X_\lambda$ with $\lambda\in \mathcal{T}^>$, this is a non-unital subalgebra. Let $\breve{\Hecke}^>_W$ denote the closure of $\Hecke^>_W$ in $\breve{\Hecke}^+_W$. Note that $\Hecke_W^>\subset\Hecke_W^+$ and
$\breve{\Hecke}_W^>\subset\breve{\Hecke}_W^+$ are two-sided ideals thanks to Lemma \ref{Lem:strictly_Tits_properties}.

\begin{Prop}\label{Prop:action_description}
For any element  $a\in \Hecke_{\leqslant W\vec{\mu}}^>$, there is a unique element $\omega(a)\in \breve{\Hecke}_{\leqslant W\vec{\mu}}^>$ such that
$\Upsilon(a)\bone_!^1=\bone_*^1\omega(a)$.
\end{Prop}
The same conclusion holds for all $a\in \breve{\Hecke}_{\leqslant W\vec{\mu}}^>$, but we will not need this claim.
\begin{proof}
The map $\breve{\Hecke}_W^+\rightarrow \breve{\B}^+$ given by $b\mapsto \bone_*^1 b$ is injective because every element of $\breve{\Hecke}^+_{\leqslant W\vec{\mu}}$ can be uniquely written as an infinite sum $\sum_{w\in W}T^{-1}_w F_w$ for $F_w\in \Ring\breve{\mathcal{T}}_{\leqslant W\vec{\mu}}$.

Note that to establish the existence of $\omega(a)$, it is enough to establish the existence of $\omega(X_\lambda)\in \breve{\Hecke}_{\leqslant W\lambda}^>$
for $\lambda\in P_+\cap \mathcal{T}^>$.
Indeed, the elements of the form $T_wX_\lambda T_u$ for such $\lambda$ with $\lambda\leqslant \mu_i$ for some $i$ span $\Hecke^>_{\leqslant W\vec{\mu}}$. Since $T_w\bone_!^1=\bone_!^1 T_w, T^{-1}_w\bone_*^1=\bone_*^1T^{-1}_w$ for all $w\in W$, we must have ${\omega(T_wX_{\lambda} T_u)=T^{-1}_{w^{-1}}\omega(X_{\lambda})T^{-1}_{u^{-1}}}$.

By  Lemma \ref{Lem:convexity}, if $\bone_*^wX_\mu$ appears in the expansion of $X_{-\lambda} \bone_!^1$, then $W\mu\leqslant W\lambda$. We need to prove that for each $d\geqslant 0$, the number of monomials $\bone_*^w X_\mu$ with nonzero coefficients in $X_{-\lambda}\bone^1_!$ such that $d(\lambda,\mu)=d$ is finite; this will ensure that $\omega(X_{\lambda})$ actually lies in the small completion, and will thus complete the proof. We prove this claim by induction on $d$.

For the base case, suppose $d=0$. We claim that the corresponding summand of $X_{-\lambda}\bone_!^1$ is $\bone_*^1 T_{w_{0,\lambda}}^{-2} X_{\lambda}$, where we write $w_{0,\lambda}$ for the maximal length element in the standard parabolic subgroup ${W_0:=\operatorname{Stab}_W(\lambda)}$. In other words, we want to show that $X_{-\lambda}\bone_!^1$ and $\bone_*^1 T_{w_{0,\lambda}}^{-2} X_{\lambda}$, agree up to lower degree terms, i.e.\ the sum of monomials $\bone^*_wX_{\mu}$ with $W\mu<W\lambda$. By  Remark \ref{Rem:uniqueness_approx}, this is equivalent to the following two conditions:
\begin{itemize}
\item[(1!)] The coefficient of $\bone^1_*$ in the expansion of  $\bone_*^1 T_{w_{0,\lambda}}^{-2} X_{\lambda}$ in the right $\Ring \mathcal{T}$-module basis $\bone^w_*$ is $X_{\lambda}$,
\item[(2!)] and the coefficient of $\bone^1_*$ in the expansion of $\bone_*^1 T_{w_{0,\lambda}}^{-2} X_{\lambda}T_u$ is in $\Ring\mathcal{T}_{<W\lambda}$ for all $u\neq 1$.
\end{itemize}

Note that the coefficient of $1$ in the expansion of $T_{w_{0,\lambda}}^{-2}$ in the basis $T_{w^{-1}}^{-1}$ is $1$ (this is because the coefficient of $T_{w_{0,\lambda}}$ in the expansion of $T_{w_{0,\lambda}}^{-1}$ in the basis $T_w$ is $1$). So (1!) holds. Now we need to check (2!).

We write $u$ as {$u_1u_2$, where $u_1\in W_0$} and $u_2$ is the shortest element in the coset {$W_0u$}. Then
$$T_{w_{0,\lambda}}^{-2} X_{\lambda}T_u=T_{w_{0,\lambda}}^{-1}T_{(w_{0,\lambda}u_1^{-1})^{-1}}^{-1}X_{\lambda}T_{u_2}.$$
Let $u_2=s_1\ldots s_\ell$ be a reduced expression. Since $u_2$ is shortest in
$u_2\operatorname{Stab}_W(\lambda)$, we see that
$\lambda>s_\ell \lambda>s_{\ell-1}s_\ell \lambda>\ldots>s_1\ldots s_\ell \lambda$. Using
(\ref{eq:Bernstein_new}) $\ell$ times, we see that $X_{-\lambda}T_{u_2}$ and $T_{u_2^{-1}}^{-1}X_{-u_2\lambda}$ agree up to terms of the form $T_?X_\nu$ with $W\nu<W\lambda$ (to be referred to as {\it lower layer terms} for time being).
So
$$T_{w_{0,\lambda}}^{-2} X_{-\lambda}T_u=T_{w_{0,\lambda}}^{-1}T_{(w_{0,\lambda}u_1^{-1})^{-1}}^{-1}T_{u_2^{-1}}^{-1}X_{-u_2\lambda}$$
up to lower layer terms.
The coefficient of $1$ in the expansion of $$T_{w_{0,\lambda}}^{-1}T_{(w_{0,\lambda}u_1^{-1})^{-1}}^{-1}T_{u_2^{-1}}^{-1}$$ in the basis $T_w^{-1}$ can only be nonzero if $u_1=u_2=1$.
This follows because the expansion of $T_{w_{0,\lambda}}^{-1}T^{-1}_{w_{0,\lambda}u_1^{-1}}$ in the basis $T_{w}^{-1}$ only contains $w\in {W_0}$, and the coefficient of $1$ is nonzero if and only if $u_1=1$.
This finishes the base case of the induction.

We now proceed to the induction step. Suppose that $${X_{-\lambda}\bone_!^1}=\bone_*^1(\sum_{i=1}^r f_i T_{w_i^{-1}}^{-1}X_{\mu_i}+\sum_{\substack{w, \mu, f:\ d(\lambda,\mu)\geqslant d}}fT_{w^{-1}}^{-1}X_\mu),$$
where $f_i, f\in \Ring$, $W\mu_i\leqslant W\lambda$, and $d(\lambda,\mu_i)<d$. We want to show that the number of nonzero terms $fT_{w^{-1}}^{-1}X_\mu$ with $d(\lambda,\mu)=d$ is finite. Note that from $W\mu_i\leqslant W\lambda$ and Lemma \ref{Lem:strictly_Tits_properties} it follows that $\mu_i\in \mathcal{T}^>$.

First, we show that
\begin{itemize}
\item[($\star$)]
there is a positive integer $\ell$ such that for all $u\in W$ with $\ell(u)>\ell$ the expansion of $T_{w_i^{-1}}^{-1}X_{\mu_i}T_u$ in the basis $T_{w^{-1}}^{-1}X_\mu$ has no terms with $w=1$ and $d(\lambda,\mu)\leqslant d$.
\end{itemize}

Suppose that the term $X_{\mu'}$ occurs with nonzero coefficient. By Propositions \ref{Prop:weights_occuring1}, \ref{Prop:weights_occuring2}, there are finite subsets $\Sigma_1,\Sigma_2\subset \mathcal{T}^>$, depending on $w_i,\mu_i,d$ (but not on $u$) such that
$\mu'\in u^{-1}\Sigma_1\cap \Sigma_2$. For fixed finite subsets $\Sigma_1,\Sigma_2\subset \mathcal{T}^>$, only finitely many $u\in W$ can satisfy $u^{-1}\Sigma_1\cap \Sigma_2\neq\varnothing$: for each pair $(\sigma_1,\sigma_2)\in \Sigma_1\times \Sigma_2$, the set of $u$ with $u\sigma_1=\sigma_2$ is either empty or a coset of the finite stabilizer of $\sigma_1$. This shows ($\star$).

Set $\beta={X_{-\lambda}\bone_!^1}-\bone_*^1\sum_{i=1}^r f_i T_{w_i^{-1}}^{-1}X_{\mu_i}$ so that for all terms $\bone_*^w X_\mu$ appearing in the expansion of $\beta$ we have $W\mu\leqslant W\lambda$ and $d(\lambda,\mu)\geqslant d$. What we need to show is that the number of terms with $d(\lambda,\mu)=d$ is finite. Arguing as in the previous paragraph, we see that for $\ell(u)>\ell$, the coefficient of $\bone_*^1X_{\mu'}$ in $\beta T_u$ is zero for all $\mu'$ satisfying $d(\lambda,\mu')=d$. Equivalently, by Lemma \ref{Lem:convexity}, $\operatorname{res}_1(\beta T_u)$ is the sum of lower layer terms, i.e., only includes $X_{-\nu}$ with $d(\lambda,\nu)>d$, where, $\operatorname{res}_1$ stands for the coefficient of $\bone^!_1$ in the expansion in the left $\Ring P$-module basis $\bone^w_!$.

We claim that
\begin{itemize}
\item[($\star\star$)]
there are (uniquely determined) elements $F_w=\sum_{\mu'}f_{w,\mu'}X_{-\mu'}$, where $\ell(w)\leqslant \ell$, the sum is taken over all $\mu'$ with $d(\lambda,\mu')=d$ and $f_{w,\mu'}\in \Ring$
such that, up to terms involving only lower layer terms, i.e., $X_\nu$ with $W\nu\leqslant W\lambda$ and $d(\lambda,\nu)>d$, we have
$$\beta\equiv\sum_{\ell(w)\leqslant \ell}F_w\bone_!^1T_{w^{-1}}^{-1}.$$
\end{itemize}
By the induction base proved above, we know that each $F_w\bone_!^w$ is a finite linear combination of $\bone^u_*X_{\mu''}$ modulo lower layer summands, so ($\star\star$) will finish the proof.

To prove ($\star\star$) we observe that the  $\operatorname{res}_1(F_w \bone_!^1T_{w^{-1}}^{-1}T_u)$ equals
\begin{itemize}
\item
$0$ if $\ell(u)\geqslant \ell(w)$ and $u\neq w$,
\item $F_w$ if $u=w$.
\end{itemize}
In particular, for any choice of $F_w$, 
\begin{equation}\label{eq:difference}
\operatorname{res}_1([\beta-\sum_{\ell(w)\leqslant \ell}F_w \bone_!^1 T_{w^{-1}}^{-1}]T_u)
\end{equation} is the sum of lower layer terms if $\ell(u)>\ell$. We can then solve for $F_w$ recursively in the decreasing order of $\ell(w)$ to achieve that (\ref{eq:difference}) is a lower layer term. Namely, suppose we found $F_w$ for all $w$ with $\ell(w)>\ell(w')$ and want to find $F_{w'}$. Then 
 $F_{w'}$ equals $$\operatorname{res}_1([\beta-\sum_{w|\ell(w')<\ell(w)\leqslant \ell}F_w\bone_!^1T^{-1}_{w^{-1}}]T_{w'}).$$
 modulo the lower layer terms.

Thanks to the approximate analog of Lemma \ref{Lem:uniqueness_conjugation} (cf. Remark \ref{Rem:uniqueness_approx} vs Remark \ref{Rem:uniqueness}), this choice of $F_w$ with $\ell(w)\leqslant \ell$ satisfies ($\star\star$) and finishes the proof.
\end{proof}

\subsection{Large completion}\label{SS_large_completion}
The goal of this section is to investigate a larger completion of $\Hecke_W^>$.

\begin{defi}\label{defi:large_completion}
Let $\vec{\mu}=(\mu_1,\ldots,\mu_k)\in (P_+\cap \mathcal{T}^>)^k$. Let
$\widehat{\Hecke}^{>}_{\leqslant W\vec{\mu}}$ denote the set of all infinite sums $\sum a_{w,\mu}  T_w X_\mu$, where $a_{w,\mu}\in \Ring$  ($w$ runs over $W$ and $\mu$ over $\mathcal{T}^>$) subject to (I) { in Definition \ref{defi:small_completion}} and
\begin{itemize}
\item[(II')] For each $d\in \Z$ there are only finitely many elements $\mu$ such that $a_{w,\mu}\neq 0$ for some $w$ and $\langle\rho^\vee, \mu_+\rangle\geqslant d$.
\end{itemize}
We write $\widehat{\Hecke}^>_W$ for the union of $\widehat{\Hecke}^>_{\leqslant W\vec{\mu}}$ (viewed as subsets in the set of all infinite sums $\sum_{w,\mu}a_{w,\mu}T_w X_\mu$).
\end{defi}
In particular, $\breve{\Hecke}_W^{>}\subset \widehat{\Hecke}^>_W$.

\begin{Rem}\label{rem:hat_topology}
We note that $\widehat{\Hecke}^>_{\leqslant W\vec{\mu}}$ comes with an $\Ring$-module topology as follows. 
{We write $\widehat{\Hecke}_W$ for the set of all formal sums $\sum_{w\in W}a_w T_w$ equipped with the direct product topology.
Consider the module $
\widehat{\mathcal{H}}^d_{\leqslant W\vec{\mu}}=\bigoplus_\mu \widehat{\Hecke}_W X_\mu$, where the summation is taken over all $\mu$ { $\in \mathcal{T}^>$} such that $W\mu\leqslant W\mu_i$ for some $i$ and $\langle\rho^\vee,\mu_+\rangle=d$. Note that { $\widehat{\Hecke}^d_{\leqslant W\vec{\mu}} = 0$ for $d \gg 0$ and} $\widehat{\Hecke}^{>}_{\leqslant W\vec{\mu}}$ is the product $\prod_d \widehat{\Hecke}^d_{\leqslant W\vec{\mu}}$. We equip each $\widehat{\Hecke}^d_{\leqslant W\vec{\mu}}$ with the topology of the direct sum of the topological modules $\widehat{\Hecke}_W X_\mu$. We note that by the definition of the direct sum topology, each converging sequence in $\widehat{\Hecke}^d_{\leqslant W\vec{\mu}}$ has finite support in $\mathcal{T}^>$.
We equip $\widehat{\Hecke}^>_{\leqslant W\vec{\mu}}$ with the product topology. As the direct sums and products of (sequentially) complete topological spaces are  complete, $\widehat{\Hecke}^>_{\leqslant W\vec{\mu}}$ is complete as well. }
\end{Rem}

The following lemma explains crucial properties of $\widehat{\Hecke}_W^{>}$. 

\begin{Lem}\label{Lem:completion_product}
The following claims are true:
\begin{enumerate}
\item The product map $\Hecke_W\times \Hecke^>_{\leqslant W\vec{\mu}}\rightarrow \Hecke^>_{\leqslant W\vec{\mu}}$ uniquely extends by continuity to $$\widehat{\Hecke}_W\times \breve{\Hecke}^>_{\leqslant W\vec{\mu}}\rightarrow
\widehat{\Hecke}^>_{\leqslant W\vec{\mu}}$$
\item The product map $\Hecke^>_{\leqslant W\vec{\mu}}\times \Hecke^>_{\leqslant W\vec{\mu}'}\rightarrow \Hecke^>_{\leqslant W(\vec{\mu}*\vec{\mu}')}$ uniquely extends by continuity to
\begin{itemize}
\item[(a)] $\widehat{\Hecke}^>_{\leqslant W\vec{\mu}}\times \widehat{\Hecke}^>_{\leqslant W\vec{\mu}'}\rightarrow \widehat{\Hecke}^>_{\leqslant W(\vec{\mu}*\vec{\mu}')}$,
\item[(b)]  $\breve{\Hecke}^>_{\leqslant W\vec{\mu}}\times \widehat{\Hecke}^>_{\leqslant W\vec{\mu}'}\rightarrow \breve{\Hecke}^>_{\leqslant W(\vec{\mu}*\vec{\mu}')}$.
\end{itemize}
\end{enumerate}
\end{Lem}
\begin{proof}
(1) follows by comparing (II) from Definition \ref{defi:small_completion} with (II') from Definition \ref{defi:large_completion}.

{ Since all $\Ring$-modules are complete, it suffices to show that the product maps $\Hecke^>_{\leqslant W\vec{\mu}}\times \Hecke^>_{\leqslant W\vec{\mu}'}\rightarrow \Hecke^>_{\leqslant W(\vec{\mu}*\vec{\mu}')}$ are continuous with respect to the relevant topologies.}

{

We prove (2a). We need to show that if
\begin{equation} \label{eq:completion product 0}
\sum_{w_1 \in W, \nu_1 \in \mathcal{T}^>} a_{w_1, \nu_1} T_{w_1}X_{\nu_1} \in \widehat{\Hecke}^>_{\leqslant W\vec{\mu}}\quad \text{and} \quad \sum_{w_2 \in W, \nu_2 \in \mathcal{T}^>} b_{w_2, \nu_2} T_{w_2}X_{\nu_2} \in \widehat{\Hecke}^>_{\leqslant W\vec{\mu}'}
\end{equation}
are convergent sums, then the product
\begin{equation} \label{eq:completion product 1}
\sum_{w_1, w_2, \nu_1, \nu_2} a_{w_1, \nu_1} b_{w_2, \nu_2} T_{w_1}X_{\nu_1} T_{w_2} X_{\nu_2}
\end{equation}
is a convergent sum in $\widehat{\Hecke}^>_{\leqslant W(\vec{\mu}*\vec{\mu}')}$.
%
Write
\[ X_{\nu_1} T_{w_2} = \sum_{x \leq w_2, \lambda \in \mathcal{T}^>} c_{w_2, \nu_1, x, \lambda} T_x X_{w_2^{-1}\lambda}\]
for some constants $c_{w_2, \nu_1, x, \lambda} \in \Ring$, so that \eqref{eq:completion product 1} becomes
\begin{equation} \label{eq:completion product 2}
 \sum_{\nu_1, \nu_2, \lambda \in \mathcal{T}^{>}} \sum_{w_1, w_2 \in W, x \leq w_2} a_{w_1, \nu_1} b_{w_2, \nu_2}c_{w_2, \nu_1, x, \lambda} T_{w_1}T_x X_{w_2^{-1}\lambda + \nu_2}.
\end{equation}
Fix $d \in \mathbb{Z}$. We will show that there are only finitely many $w_2, \lambda, \nu_1, \nu_2$ contributing a term to \eqref{eq:completion product 2} such that $\langle \rho^\vee, (w_2^{-1}\lambda + \nu_2)_+ \rangle \geq d$, and hence that \eqref{eq:completion product 2} is a convergent sum in $\widehat{\Hecke}^>_{\leqslant W(\vec{\mu}*\vec{\mu}')}$.

First, observe that
\begin{equation} \label{eq:completion product 3}
\langle \rho^\vee, (w_2^{-1} \lambda + \nu_2)_+ \rangle \leq \langle \rho^\vee, \lambda_+\rangle + \langle \rho^\vee, \nu_{2+} \rangle \leq \langle \rho^\vee, \lambda_+\rangle + \langle \rho^\vee, \nu_{2+} \rangle.
\end{equation}
For contributing $\lambda$, we have $W\lambda \leqslant W\nu_1$, so \eqref{eq:completion product 3} gives
\[ \langle \rho^\vee, (w_2^{-1} \lambda + \nu_2)_+ \rangle \leq \langle \rho^\vee, \nu_{1+}\rangle + \langle \rho^\vee, \nu_{2+}\rangle.\]
Since each term $\langle \rho^\vee, \nu_{i+}\rangle$ is uniformly bounded above, we deduce from the convergence of \eqref{eq:completion product 0} that there are only finitely many $\nu_1$, $\nu_2$ contributing terms with $\langle \rho^\vee, (w_2^{-1}\lambda + \nu_2)_+ \rangle \geq d$.

Next, for fixed $\nu_1$, $\nu_2$ and $e \in \mathbb{Z}$, we have by Proposition \ref{Prop:weights_occuring1} that $c_{w_2, \nu_1, x, \lambda} = 0$ unless $\lambda \in \Sigma(\nu_1, e)$ (a finite set) or $d(\nu_1, \lambda) > e$. If $d(\nu_1, \lambda) > e$ then by \eqref{eq:completion product 3} again
\[
\langle \rho^\vee, (w_2^{-1} \lambda + \nu_2)_+ \rangle < \langle \rho^\vee, \nu_{1+}\rangle + \langle \rho^\vee, \nu_{2+}\rangle - e.
\]
So choosing $e$ large enough, we can ensure that $\langle \rho^\vee, (w_2^{-1} \lambda + \nu_2)_+ \rangle < d$. If $\lambda \in \Sigma(\nu_1, d)$ then writing $(w_2^{-1} \lambda + \nu_2)_+ = y(w_2^{-1}\lambda + \nu_2)$, $\lambda = y'\lambda_+$ and $\nu_2 = y''\nu_{2+}$, we have
\begin{align*}
\langle \rho^\vee, (w_2^{-1} \lambda + \nu_2)_+ \rangle &= \langle \rho^\vee, yw_2^{-1}y'\lambda_+ + yy''\nu_{2+}\rangle \\
&\leq \langle \rho^\vee, \lambda_+\rangle + \langle \rho^\vee, \nu_{2+}\rangle - \ell(yw_2^{-1}y') + L(\lambda) - \ell(y y'')  + L(\nu_2) \\
&\leq \langle \rho^\vee, \lambda_+\rangle + \langle \rho^\vee, \nu_{2+}\rangle - \ell((y'')^{-1}w_2^{-1}y') + L(\lambda) + L(\nu_2) \\
&\leq \langle \rho^\vee, \nu_{1+}\rangle + \langle \rho^\vee, \nu_{2+}\rangle + L(\lambda) + L(\nu_2) + \ell(y'') + \ell(y') - \ell(w_2), 
\end{align*}
where $L(\mu)$ denotes the number of positive real coroots singular with respect to $\mu$, and we have used that $W\lambda \leq W\nu_1$. Since everything on the right hand side but $w_2$ is fixed, the right hand side is $< d$ for all but finitely many $w_2$. This completes the proof of (2a). 
%
%

For (2b), we assume moreover that the first sum in \eqref{eq:completion product 0} converges in $\breve{\Hecke}^>_{\leqslant W\vec{\mu}}$. Arguing as above, for fixed $d$ we have only finitely many $w_2, \nu_1, \nu_2, \lambda$ contributing to \eqref{eq:completion product 2} with $\langle \rho^\vee, (w_2^{-1} \lambda + \nu_2)_+ \rangle \geq d$. But our assumption implies that there are also only finitely many $w_1$ contributing terms with fixed $\nu_1$, so \eqref{eq:completion product 2} has only finitely many terms with $\langle \rho^\vee, (w_2^{-1} \lambda + \nu_2)_+ \rangle \geq d$ and hence converges in $\breve{\Hecke}^>_{\leqslant W(\vec{\mu}*\vec{\mu}')}$ as claimed.
}

\end{proof}



\section{The bar operation}\label{sec:barinvolution}

Recall the elements $a_w\in \Ring$ from Remark \ref{Rem:star_shriek_expansion} so that
\begin{equation}\label{eq:star_shriek_expansion}
\bone_*^1=\sum_{w\in W}a_w \bone_!^w.
\end{equation}
\begin{defi}
    For each $b\in \Hecke_W^{>}$, define
$$\bar{b}=(\sum_{w\in W}a_w T_w)\omega(b),$$
which is a well-defined element of $\widehat{\Hecke}_W^{>}$ by Lemma \ref{Lem:completion_product}(1).
\end{defi}

Note that we have the  well-defined map $$\sigma:\widehat{\mathcal{H}}_W^>\rightarrow \breve{\mathcal{B}}^+, \sum_{w,\mu}a_{w,\mu}T_w X_\mu\rightarrow \sum_{w,\mu}a_{w,\mu}\bone^w_! X_\mu.$$   It is easy to see that we have
\begin{align}
    \Upsilon(b)\bone_!^1 & = \sigma(\bar{b}), \forall b\in \Hecke^>_W.
\end{align}
Of course, if $\beta\in \breve{\mathcal{H}}^>$, then $\sigma(\beta)=\bone^!_1 \beta$.

We also have the map $\sigma_0:\widehat{\Hecke}_W\rightarrow \breve{\mathcal{B}}^+$ defined similarly. Note that for all $\beta\in \widehat{\mathcal{H}}_W, c\in \breve{\Hecke}^>_W$, we have
\begin{equation}\label{eq:bimod_assoc_sigma}
\sigma_0(\beta)c=\sigma(\beta c).
\end{equation}

Thanks to Remark \ref{Rem:finite_type}, $b\mapsto\bar{b}$ is the usual bar involution in the case when $W$ is finite.

The following is our main result on the properties of $\bar{\bullet}$.

\begin{Thm}\label{Thm:bar_Tits_properties}
We have the following:
\begin{enumerate}
\item There is a unique continuous extension of $\bar{\bullet}$ to a map $ \widehat{\Hecke}_W^{>}
\rightarrow \widehat{\Hecke}_W^{>}$ (again denoted by $\bar{\bullet}$).
\item 
This extension is an automorphism.
\item $\bar{\bullet}$ respects the product on
$\widehat{\Hecke}_W^{>}$.\end{enumerate}
\end{Thm}

\begin{proof}
For (1), since $\Hecke_W^{>}$ is dense in $\widehat{\Hecke}_W^{>}$, uniqueness follows from continuity. So, we only need to establish the existence of an extension. In order to do this, we need to write a formula for $\bar{b}$ when $b=(\sum_{u\in W}b_uT_u)X_\lambda$ for $\lambda\in \mathcal{T}^>$. 

Now note that $\bone_*^1 T_{u^{-1}}^{-1}=\bone_*^u$. It follows that there are elements $a_w^u\in \Ring$ with $\bone_*^1 T_{u^{-1}}^{-1}= \bone_!^1\sum_{w} a_w^u T_w$ such that $w\succeq u$ whenever $a_{w}^u\neq 0$. Then we set
\begin{equation}\label{eq:bar_cont_extension}
\bar{b}=(\sum_{u,w}\bar{b}_u a_w^u T_w) \omega(X_{\lambda}).
\end{equation}
Since $\sum_{u,w}\bar{b}_u a_w^u T_w=\sum_{w\in W}(\sum_{u\preceq w}\bar{b}_u a_w^u)T_w$,  we use (1) of Lemma \ref{Lem:completion_product} to see that the right-hand side of (\ref{eq:bar_cont_extension}) is a well-defined element of $\widehat{\Hecke}_{\leqslant W\lambda}^{>}$. From here and the finiteness condition (II') in Definition \ref{defi:large_completion} we see that for every $$\sum_{i} (\sum_{w\in W}b^i_w T_{{w}^{-1}})X_{\lambda^i}\in \Hecke^>_{\leqslant W\vec{\mu}}\subset\widehat{\Hecke}_W^{>},$$ we get that
$$\sum_i(\sum_{u,w}\bar{b}^i_u a_w^u T_w) \omega(X_{\lambda^i})$$
is a well-defined element of $\Hecke^>_{\leqslant W\vec{\mu}}$. { Also the map $b\mapsto \overline{b}: \Hecke^{>}_{\leqslant W\vec{\mu}}\rightarrow \widehat{\Hecke}^{>}_{\leqslant W\vec{\mu}}$ is continuous with respect to the topology from Remark \ref{rem:hat_topology}}. This shows (1).

We now prove (2). It is enough to show that $b\mapsto \bar{b}$ is an automorphism of
$\widehat{\mathcal{H}}^>_{\leqslant W\vec{\mu}}$ for each $\vec{\mu}$. We equip $\widehat{\mathcal{H}}^>_{\leqslant W\vec{\mu}}$
with the {filtration} $\widehat{\mathcal{H}}^>_{\leqslant W\vec{\mu}}(\leqslant d)$ for $d\in \Z$ consisting of all elements $\sum_{w,\mu}b_{w,\mu}T_w X_\mu$ with 
$b_{w,\mu}\neq 0\Rightarrow \langle\rho^\vee,\mu_+\rangle\leqslant d$. By the construction in the proof of (1), in particular, (\ref{eq:bar_cont_extension}), $\bar{\bullet}$ preserves each $\widehat{\mathcal{H}}^>_{\leqslant W\vec{\mu}}(\leqslant d)$. Also, it is easy to see that 
$\widehat{\mathcal{H}}^>_{\leqslant W\vec{\mu}}$ is complete and separated with respect to the topology induced by the filtration above. So it is enough to prove the following claim: for each $\lambda\in P_+\cap \mathcal{T}^>$, the endomorphism of the top layer part
$$\widehat{\mathcal{H}}^>_{W\lambda}:=
\widehat{\mathcal{H}}^>_{\leqslant W\lambda}/ 
\widehat{\mathcal{H}}^>_{\leqslant W\lambda}(\leqslant \langle\rho^\vee,\lambda\rangle-1)$$ induced by $b\mapsto \bar{b}$
is invertible. 
Thanks to Lemma \ref{Lem:simplecases}, we can present $\widehat{\mathcal{H}}^>_{ W\lambda}$ 
as the set of all infinite sums $\sum_{u,x}b_{u,x} T_u X_\lambda T_x$ with $w\in W$ and $x$ shortest in $\operatorname{Stab}_W(\lambda)x$ subject to the following finiteness condition:
\begin{itemize}
\item  the set of all $x$ such that $b_{u,x}\neq 0$ for some $u$ is finite. 
\end{itemize}

By the case of $d=0$ in the proof of Proposition \ref{Prop:action_description}, if $\lambda$ is dominant, then 
$\omega(X_\lambda)={T_{w_{0,\lambda}}^{-2}} X_\lambda$ in $\widehat{\mathcal{H}}^>_{W\lambda}$. So thanks to (\ref{eq:bar_cont_extension}),
the endomorphism $b\mapsto \bar{b}$ of $\widehat{\mathcal{H}}^>_{W\lambda}$ is given by 

\begin{equation}\label{eq:bar_top_layer}
\sum_{u,x}b_{u,x}T_u X_\lambda T_x\mapsto \sum_{u,w,x}\bar{b}_{u,x} a_w^u T_w{T_{w_{0,\lambda}}^{-2}} X_\lambda T^{-1}_{x^{-1}}.
\end{equation}
We can filter $\widehat{\mathcal{H}}^>_{W\lambda}$ by the poset $(W/\operatorname{Stab}_W(\lambda),\leqslant)$ according to the $x$-component. Since the right multiplication by ${T_{w_{0,\lambda}}^{-2}}$ is an  automorphism of $\widehat{\mathcal{H}}_W$, we reduce to showing that the endomorphism 
\begin{equation}\label{eq:widehat_H_W_endomorphism}
\sum_{u\in W}b_u T_u\mapsto \sum_{u,w\in W}\bar{b}_u a^u_w T_w 
\end{equation}
of $\widehat{\mathcal{H}}_W$ is invertible. Recall that the elements $a^u_w$ are defined by $\bone^u_*=\sum_{w\in W}a_w^u \bone^w_!$, so that $a_w^w=1$ and $a_{w}^u\neq 0\Rightarrow w\succeq u$. The claim that (\ref{eq:widehat_H_W_endomorphism}) is an automorphism easily follows because it is strictly uni-upper-triangular with respect to the Bruhat order. 

Let us show (3): { By (1), $\bar{\bullet}$ is continuous, and by (2a) of Lemma \ref{Lem:completion_product}, the product on $\widehat{\Hecke}_W^>$ is continuous.} So, it suffices to show that $\bar{\bullet}:\Hecke^>_W\rightarrow \widehat{\Hecke}^>_W$ respects the product.  Set $\alpha=\sum_{w\in W}a_w T_{w}$ so that $\bone_*^1 = \bone_!^1\alpha$.

By definition, $\bar{b}=\alpha \omega(b)$ for any $b\in \Hecke_W^{>}$.
Consequently, the required equality $\overline{b_1b_2}=\bar{b}_1\bar{b}_2$ is equivalent to $\alpha \omega(b_1b_2) = [\alpha \omega(b_1)][ \alpha \omega(b_2)]$. This will follow if we show
\begin{equation}\label{eq:omega_relation}
\omega(b_1b_2) = \omega(b_1)(\alpha \omega(b_2)),
\end{equation}
note that the right hand side is a well-defined element of $\breve{\Hecke}^>_W$ by (2b) of Lemma \ref{Lem:completion_product}.

To prove (\ref{eq:omega_relation}), we utilize the $\breve{\Hecke}^+_W$-bimodule structure on $\breve{\B}^+$, Proposition \ref{Prop:small_completion_bimodule}. First, we claim that the map $\rho: \breve{\Hecke}_W^+ \to \breve{\B}^+$ defined by $\rho(h) = \bone_*^1 h$ is injective. Any $h \in \breve{\Hecke}_W^+$ can be written uniquely as $\sum_{w} T^{-1}_{w^{-1}} F_w$ with $F_w \in \Ring\breve{\mathcal{T}}$, and then
\[
\rho(h)=\sum_w \bone_*^w F_w,
\]

Since $\bone_*^w = \bone_!^w + \sum_{v\succ w} a_v^w \bone_!^v$, we see that if $w$ is a maximal element in the support of $h$ (in the Bruhat order), then the coefficient of $\bone_!^w$ in $\rho(h)$ involves the leading term of $F_w$ and cannot vanish unless $h=0$. This shows the injectivity claim.

Using the injectivity of $\rho$, we compute $\Upsilon(b_1 b_2)\bone_!^1$ in two ways. By definition,
$$\Upsilon(b_1 b_2)\bone_!^1 = \bone_*^1 \omega(b_1 b_2).$$
On the other hand, using the $\breve{\Hecke}^+_W$-bimodule structure on $\breve{\B}^+$, we have
$$\Upsilon(b_1) (\Upsilon(b_2) \bone_!^1) = \Upsilon(b_1) (\bone_*^1 \omega(b_2))$$
We claim that
\begin{equation}\label{eq:bimodule_equality}
\Upsilon(b_1) \sigma(\beta c_2)=\bone_*^1[\omega(b_1)(\beta c_2)], \forall \beta\in \widehat{\mathcal{H}}_W, b_1\in \Hecke^>_W, c_2\in \breve{\Hecke}^>_W.\end{equation}

By Lemma \ref{Lem:completion_product},  the product map
$$\breve{\Hecke}^>_W\times \widehat{\Hecke}_W\times \breve{\Hecke}^>_{W}\rightarrow \breve{\Hecke}^>_W, (a,b,c)\mapsto a(bc),$$
is well-defined and continuous. And by Proposition \ref{Prop:small_completion_bimodule}, the action map $\breve{\Hecke}_W^>\times \breve{\B}^+\times \breve{\Hecke}_W^>\rightarrow \breve{\B}^+$ is continuous. 
So it is enough to check (\ref{eq:bimodule_equality}) when $\beta\in \Hecke_W$ where it just follows from the associativity of the $\breve{\Hecke}_W$-bimodule product and the identity $\Upsilon(b_1)\bone^1_!=\bone^*_1\omega(b_1)$.

Finally, we deduce (\ref{eq:omega_relation}) from
(\ref{eq:bimodule_equality}). Take $\beta=\alpha, c_2=\omega(b_2)$. In the r.h.s. of (\ref{eq:bimodule_equality}) we have $\bone^1_* [\omega(b_1)(\alpha\omega(b_2))]$. Thanks to the injectivity of the map $\rho$, it is enough to show that in the l.h.s. we have $\bone^1_*\omega(b_1b_2)=\Upsilon(b_1)\Upsilon(b_2)\bone^1_!$. So we reduce to $\sigma(\alpha\omega(b_2))=\Upsilon(b_2)\bone^1_!$. By (\ref{eq:bimod_assoc_sigma}) $\sigma(\alpha\omega(b_2))=\sigma_0(\alpha)\omega(b_2)$. The right hand side equals $\bone^1_*\omega(b_2)=\Upsilon(b_2)\bone^!_1$. This finishes the proof. 
\end{proof}

\section{The level zero setting}\label{sec:level0}

We now turn to the level-zero affine case, where we make a similar construction of the bar operation in
a more explicit combinatorial form. The algebra in question is an appropriate localization of Cherednik's double affine Hecke algebra. After recalling two useful presentations of this algebra, we use the formal elements and relations introduced in \S\ref{sec:twisted-generators}
to motivate the construction of a certain ``semi-infinite
correction factor'' \(a\) (an infinite sum that coincides with an expansion of a suitable element of the localization to be introduced in Proposition~\ref{prop:a-limit-formula}), compute an explicit formula for this element, and use it to define the bar operation. Finally,
we show that the resulting bar operation becomes an involution after the
specialization \(q=1\).

\subsection{Cherednik's DAHA}

In this subsection, we recall Cherednik's presentation of the double affine
Hecke algebra \cite{Cherednik}. We then recall Cherednik's duality
anti-isomorphism $\varphi$, which we view as an instance of Dolbeault Langlands duality between the DAHA associated to a group and its Langlands dual, and we use it to pass to a second presentation, which will be more convenient
for the formulas we introduce. We now introduce the groups to which the algebras below will be naturally associated.

Let $G$ be a connected almost simple group with maximal torus $T$, and let
\[
(\Lambda,\Phi,\Lambda^\vee,\Phi^\vee)
=\bigl(X^*(T),\Phi,X_*(T),\Phi^\vee\bigr)
\]
be its root datum. The Langlands dual
group $G^\vee$ has dual root datum
\[
(\Lambda^\vee,\Phi^\vee,\Lambda,\Phi).
\]
Let \(Q_{\mathrm{fin}}\subset \Lambda\) be the root lattice generated by
\(\Phi\), and let \(Q_{\mathrm{fin}}^\vee\subset \Lambda^\vee\) be the coroot
lattice generated by \(\Phi^\vee\).
Write $\Phi^+$
for the positive roots and $W_{\mathrm f}$ for the (finite) Weyl group. We label the finite simple roots by
$\alpha_1,\dots,\alpha_n$, with corresponding simple coroots
\(\alpha_1^\vee,\ldots,\alpha_n^\vee\), and set
\(\Lambda_{\mathbb R}:=\Lambda\otimes_{\mathbb Z}\mathbb R\) and
\(\Lambda^\vee_{\mathbb R}:=\Lambda^\vee\otimes_{\mathbb Z}\mathbb R\).

We now pass to the associated untwisted affine setting and set
\[
\Lambda_{\mathrm{aff}}:=\Lambda\oplus \mathbb Z\delta,
\qquad
\Lambda_{\mathrm{aff}}^\vee:=\Lambda^\vee\oplus \mathbb Z\delta.
\]

Let $\theta$ and $\theta^\vee$ be the highest roots of $\Phi$ and
$\Phi^\vee$, respectively. We write
\[
\alpha_0=\delta-\theta,
\qquad
\alpha_0^\vee=\delta-\theta^\vee,
\]
so for
\(\lambda \in \Lambda\) we have
\[
\langle \lambda,\alpha_0^\vee\rangle=-\langle \lambda,\theta^\vee\rangle.
\]
We set
\[
W_{\mathrm{aff}}^\vee:=W_{\mathrm f}\ltimes Q_{\mathrm{fin}},\qquad
W_{\mathrm{ext}}^\vee:=W_{\mathrm f}\ltimes \Lambda.
\]

Let \(t_\lambda\in W_{\mathrm{ext}}^\vee\) denote translation
by \(\lambda\).

The actions of \(W_{\mathrm{ext}}^\vee\) on $\Lambda^\vee_{\mathrm{aff}}\otimes_{\mathbb{Z}}\mathbb{R}=\Lambda^\vee_{\mathbb{R}}\oplus \mathbb{R}\delta$ and its dual $\Lambda_{\mathbb{R}}\oplus \mathbb{R}d$ (where $d$ sends $\Lambda^\vee_{\mathbb{R}}$ to $0$ and $\delta$ to $1$) are  given by
\begin{equation}\label{eq:dual-affine-actions}
\begin{split}
&w(b+m\delta)=wb+m\delta,
\quad t_\lambda(b+m\delta)=b+\bigl(m-\langle\lambda,b\rangle\bigr)\delta,\\
&w(x+zd)=wx+zd, \quad t_\lambda(x+zd)=x+z\lambda+zd. 
\end{split}
\end{equation}
where \(x,\lambda\in\Lambda_{\mathbb{R}}\),
\(b\in\Lambda^\vee_{\mathbb R}\) and $m,z\in \mathbb{R}$.
The affine hyperplane arrangement on
\(\Lambda_{\mathbb R}
\)
is given by
\begin{equation}\label{eq:level-zero-affine-arrangement}
\langle x,\beta^\vee\rangle\in\mathbb Z
\qquad(\beta\in\Phi).
\end{equation}
Let \(A^+\) be the alcove in the finite dominant chamber with \(0\) in its
closure, and set \(A^-=-A^+\). Explicitly,
\begin{align}\label{eq:positive-negative-alcoves}
A^+
&=\left\{x\ \middle|\
\langle x,\alpha_i^\vee\rangle>0\ (1\leqslant i\leqslant n),\quad
\langle x,\theta^\vee\rangle<1\right\},\notag\\
A^-
&=\left\{x\ \middle|\
\langle x,\alpha_i^\vee\rangle<0\ (1\leqslant i\leqslant n),\quad
\langle x,\theta^\vee\rangle>-1\right\}.
\end{align}

We denote the affine Dynkin diagram by $\Gamma$.
For \(0\leqslant i,j\leqslant n\), let \(m_{ij}\) denote the order of \(s_is_j\) in
\(W_{\mathrm{aff}}^\vee\).

The extended affine Weyl group \(W_{\mathrm{ext}}^\vee\) contains a subgroup
\(\Pi \cong \Lambda/Q_{\mathrm{fin}}\) of elements of length zero relative to \(A^+\); their action
on \(\Lambda_{\mathrm{aff}}^\vee\) preserves the set of affine simple coroots
(and fixes \(\delta\)), and thus gives a distinguished subgroup of diagram
automorphisms of the affine Dynkin diagram.

\begin{defi}
    Let $O$ be the set of indices in the orbit of the zero vertex in the affine Dynkin diagram under automorphisms in $\Pi$ ($O= \{0\}$ for $E_8$, $F_4$, $G_2$). Let $O' = O \setminus \{0\}$.
\end{defi}

For \(\pi\in\Pi\), write \(\pi(i)=j\) when
\(\pi(\alpha_i^\vee)=\alpha_j^\vee\). For each $r \in O'$, there exists a unique element $\pi_r \in \Pi$ such that $\pi_r(0)=r$. When an index $r\in O'$ is used as a lattice element, it denotes the corresponding minuscule fundamental weight in \(\Lambda\). For a non-simply-connected group, only those minuscule fundamental weights that belong to \(\Lambda\) occur in \(O'\). For the simply connected form, in type $A_n$ one has $O'=\{1,\ldots,n\}$, and $\pi_r$ is the cyclic rotation carrying vertex $0$ to vertex $r$; in type $D_n$ one has $O'=\{1,n-1,n\}$, corresponding to the minuscule weights $\omega_1,\omega_{n-1},\omega_n$.

For each \(r\in O'\), let \(w_0^{(r)}\) be the longest element of the
stabilizer of \(r\) in \(W_{\mathrm f}\), and set
\[
u_r=w_0w_0^{(r)}.
\]
Writing \(p_{\mathrm f}:W_{\mathrm{ext}}^\vee\to W_{\mathrm f}\) for the
finite projection, we have
\begin{equation}\label{eq:positive-length-zero}
\pi_r=t_ru_r^{-1},
\qquad
p_{\mathrm f}(\pi_r)=u_r^{-1},
\end{equation}
and \(\pi_r\) preserves \(A^+\). We define an involution
\(r\mapsto r^*\) on \(O'\) by \(\pi_{r^*}=\pi_r^{-1}\); it satisfies
\(u_{r^*}=u_r^{-1}\). We extend it to \(O\) by setting \(0^*=0\). By \eqref{eq:dual-affine-actions},
\begin{equation}\label{eq:pi-affine-action}
\pi_r(b+m\delta)
=u_r^{-1}(b)
+\left(m-\left\langle r,u_r^{-1}(b)\right\rangle\right)\delta.
\end{equation}

Since \(\pi_r^{-1}\) preserves \(A^+\), its conjugate under \(x\mapsto -x\),
namely \(u_rt_r\), preserves \(A^-\). Therefore
\begin{equation}\label{eq:negative-alcove-shift}
A^-+r=u_r^{-1}A^-.
\end{equation}

Let \(q\) be an indeterminate and set
\[
\Ring_q=\Ring[q^{\pm1}].
\]

\begin{defi}\label{def:daha}
    The double affine Hecke algebra $\mathbb{H}_{G^\vee}$ is presented over $\Ring_q$ by the affine Hecke generators $\{T_i\}_{i=0}^n$, the pairwise commuting elements $\{X_b \mid b \in \Lambda^\vee\}$, and the group elements $\pi_r$ for $r \in O'$. The generators $\{X_b\}$ satisfy the group algebra relations $X_0 = 1$ and $X_b X_c = X_{b+c}$.
    For affine indices, we use the abbreviation
    \[
    X_{b+m\delta}=q^mX_b
    \qquad (b\in\Lambda^\vee,\ m\in\mathbb Z).
    \]
    Thus \(X_\delta=q\) and
    \(X_{\alpha_0^\vee}=qX_{-\theta^\vee}\).
    We additionally impose the following relations:
    \begin{enumerate}
        \item\label{rel:daha-quadratic} $(T_i - v)(T_i + v^{-1}) = 0$ for $0 \leqslant i \leqslant n$;
        \item\label{rel:daha-braid} $T_iT_jT_i \dots = T_jT_iT_j \cdots$ with $m_{ij}$ factors on each side;
        \item\label{rel:daha-pi-product} The elements \(\pi_r\) satisfy multiplication relations in \(\Pi\): if \(\pi_r\pi_s=\pi_t\) in \(\Pi\), with \(r,s,t\in O\), then the same relation holds in \(\mathbb H_{G^\vee}\) with $\pi_0 = 1$;
        \item\label{rel:daha-diagram} $\pi_r T_i \pi_r^{-1} = T_j$ if $\pi_r(i)=j$, for $r \in O'$;
        \item\label{rel:daha-bernstein} for \(0\leqslant i\leqslant n\) and \(f\in \Ring_q[X_{\Lambda^\vee}]\),
        \[
        T_if-f^{s_i}T_i
        =
        \hbar\,\frac{f-f^{s_i}}{1-X_{\alpha_i^\vee}},
        \]
        with the affine simple coroot interpreted using \(X_{\alpha_0^\vee}=qX_{-\theta^\vee}\);
        \item\label{rel:daha-pi-x} $\pi_r X_b \pi_r^{-1} = X_{\pi_r(b)}$, for $r \in O'$.
    \end{enumerate}
\end{defi}

The algebra $\mathbb{H}_{G^\vee}$ also contains the usual commuting dual lattice
elements $Y_\lambda$ for \(\lambda\in\Lambda\).  If \(\lambda\) is dominant, write \(t_\lambda=\pi_r u\) with
\(r\in O\) and \(u\in W_{\mathrm{aff}}^\vee\), and set
\[
Y_\lambda=\pi_r T_u.
\]
When \(\lambda\in Q_{\mathrm{fin}}\) is dominant, this is simply \(Y_\lambda=T_{t_\lambda}\), since
\(t_\lambda\in W_{\mathrm{aff}}^\vee\) can be written as a word in the simple
reflections \(s_0,\ldots,s_n\).
For arbitrary \(\lambda\in \Lambda\), choose dominant \(\lambda_+,\lambda_-\in \Lambda\) with
\(\lambda=\lambda_+-\lambda_-\), and define
\[
Y_\lambda=Y_{\lambda_+}Y_{\lambda_-}^{-1}.
\]
This is independent of the choice of
\(\lambda_+\) and \(\lambda_-\), and \(Y_\lambda Y_\mu=Y_{\lambda+\mu}\). In the rank $1$ case, for
example, \(t_\alpha=s_0s_1\), so \(Y_{\alpha}=T_0T_1\). With this
normalization, the Bernstein relation between the \(Y\)-variables and the
finite Hecke generators is
\begin{equation}\label{eq:source-y-bernstein}
Y_\lambda T_i-T_iY_{s_i\lambda}
=
\hbar\,\frac{Y_\lambda-Y_{s_i\lambda}}{1-Y_{-\alpha_i}}
\qquad
(\lambda\in\Lambda,\ 1\leqslant i\leqslant n).
\end{equation}

Thus the algebra \(\mathbb H_{G^\vee}\) has \(X\)-lattice \(\Lambda^\vee\)
and \(Y\)-lattice \(\Lambda\). Applying the same definition to the dual root
datum gives \(\mathbb H_G\), whose \(X\)-lattice is \(\Lambda\) and whose
\(Y\)-lattice is \(\Lambda^\vee\). We can then compare the algebras $\mathbb{H}_G$ and $\mathbb{H}_{G^\vee}$ as follows.
In the formulas below, \(X_r\) and \(X_{r^*}\) are therefore \(X\)-variables
in the algebra \(\mathbb H_G\), indexed by the corresponding minuscule
weights in \(\Lambda\).

\begin{Lem}[{\cite[\S 3.2.2]{Cherednik}}]
    There exists an $\Ring_q$-linear anti-isomorphism $\varphi : \mathbb{H}_{G^\vee} \to \mathbb{H}_{G}$ satisfying
    \begin{align*}
        \varphi(T_i) & = T_i \qquad (1\leqslant i\leqslant n),\\
        \varphi(X_b) & = Y_b^{-1},\\
        \varphi(\pi_r) & = T_{u_{r^*}}^{-1}X_r^{-1}
        =X_{r^*}T_{u_r}.
    \end{align*}
\end{Lem}

This is the dual-root-datum form of Cherednik's anti-involution
\(\phi=\varepsilon\star\) from \cite[(3.3.20)]{Cherednik}; it is related to (but distinct from) Cherednik's automorphism
\(\sigma\) in \cite[(3.2.15)]{Cherednik}.

For affine indices in either \(Y\)-lattice, we use the abbreviation
\begin{equation}\label{eq:Y-affine-abbreviation}
Y_{\lambda+m\delta}=q^{-m}Y_\lambda.
\end{equation}
This convention is compatible with Cherednik duality: by
\(\Ring_q\)-linearity,
\[
\varphi(X_{b+m\delta})=q^mY_{-b}
=Y_{-b-m\delta}=Y_{-(b+m\delta)}.
\]
In particular,
\(Y_{\alpha_0^\vee}=q^{-1}Y_{-\theta^\vee}\).

We write
\begin{equation}\label{eq:gamma_defn}
\gamma_r=\varphi(\pi_{r^*})=T_{u_r}^{-1}X_{r^*}^{-1}.
\end{equation}
With this indexing, applying Cherednik's duality map to the defining relations of $\mathbb{H}_{G^\vee}$
yields the following presentation of $\mathbb{H}_{G}$ in terms of the \(Y\)-variables.

\begin{Prop}\label{prop:tygamma}
    $\mathbb{H}_{G}$ is presented over $\Ring_q$ by the generators
    \(\varphi(T_0),T_1,\ldots,T_n\), pairwise commuting elements
    $\{Y_b \mid b \in \Lambda^\vee\}$ satisfying
    \(Y_0=1\) and \(Y_bY_c=Y_{b+c}\),
    and elements $\gamma_r$ for $r \in O'$ with the following relations,
    where \(\varphi(T_i)=T_i\) for \(1\leqslant i\leqslant n\):
    \begin{enumerate}
        \item\label{rel:tygamma-quadratic} $(\varphi(T_i) - v)(\varphi(T_i) + v^{-1}) = 0, 0 \leqslant i \leqslant n$;
        \item\label{rel:tygamma-braid} $\varphi(T_i)\varphi(T_j)\varphi(T_i) \dots = \varphi(T_j)\varphi(T_i)\varphi(T_j) \cdots,$ with $m_{ij}$ factors on each side;
        \item\label{rel:tygamma-gamma-product} The elements \(\gamma_r\) satisfy multiplication relations in \(\Pi\): if \(\pi_r\pi_s=\pi_t\) in \(\Pi\), with \(r,s,t\in O\), then \(\gamma_r\gamma_s=\gamma_t\), where \(\gamma_0 = 1\);
        \item\label{rel:tygamma-diagram} $\gamma_r \varphi(T_i) \gamma_r^{-1} = \varphi(T_j)$ if $\pi_r(i)=j$;
        \item\label{rel:tygamma-bernstein} for \(0\leqslant i\leqslant n\) and \(g\in \Ring_q[Y_{\Lambda^\vee}]\),
        \[
        g\varphi(T_i)-\varphi(T_i) g^{s_i}
        =
        \hbar\,\frac{g-g^{s_i}}{1-Y_{-\alpha_i^\vee}},
        \]
        with the affine simple coroot interpreted using
        \(Y_{\alpha_0^\vee}=q^{-1}Y_{-\theta^\vee}\). For \(i=0\), the
        denominator is
        \(1-qY_{\theta^\vee}=1-Y_{\theta^\vee-\delta}=1-Y_{-\alpha_0^\vee}\);
        \item\label{rel:tygamma-y-conjugation} $\gamma_r Y_b \gamma_r^{-1} = Y_{\pi_r(b)}$, $r \in O'$.
    \end{enumerate}
\end{Prop}

\begin{proof}
    This is the presentation obtained from Definition~\ref{def:daha} by applying
    Cherednik's anti-isomorphism $\varphi$ and reindexing the length-zero
    generators by \(r\mapsto r^*\). The Hecke and braid relations are transported
    to the generators \(\varphi(T_i)\). The group law for the $\gamma_r$ follows from
    relation~(\ref{rel:daha-pi-product}) in Definition \ref{def:daha}, using that \(\Pi\cong \Lambda/Q_{\mathrm{fin}}\) is abelian
    and \(\pi_{r^*}\pi_{s^*}=\pi_{t^*}\) whenever \(\pi_r\pi_s=\pi_t\).
    The diagram and \(Y\)-conjugation relations (i.e.,(4) and (6)) are reindexed in the same way.
    For instance, if \(\pi_r(i)=j\), then applying \(\varphi\) to
    \(\pi_{r^*}T_j\pi_{r^*}^{-1}=T_i\) gives
    \(\gamma_r \varphi(T_i)\gamma_r^{-1}=\varphi(T_j)\), and the \(Y\)-conjugation relation is
    similar. It remains only to spell out the Bernstein relation in the dual
    variables. Let \(f\in\Ring_q[X_{\Lambda^\vee}]\), and put
    \(g=\varphi(f)\in\Ring_q[Y_{\Lambda^\vee}]\), so that
    \(\varphi(f^{s_i})=g^{s_i}\). Applying the anti-isomorphism \(\varphi\) to
    relation~(\ref{rel:daha-bernstein}) in Definition \ref{def:daha} gives
    \[
    g\varphi(T_i)-\varphi(T_i)g^{s_i}
    =
    \hbar\,\frac{g-g^{s_i}}{1-Y_{-\alpha_i^\vee}},
    \]
    because \(\varphi(X_{\alpha_i^\vee})=Y_{-\alpha_i^\vee}\). This is
    relation~(\ref{rel:tygamma-bernstein}).
\end{proof}

From this point onward we assume that \(G\) is simply connected and
\(O'\neq\varnothing\). Under these assumptions, the latter condition
excludes precisely types \(E_8\), \(F_4\), and \(G_2\). Thus, if
\(P_{\mathrm{fin}}\) and \(Q_{\mathrm{fin}}\) are the weight and root lattices
of \(\Phi\), and \(P_{\mathrm{fin}}^\vee\) and \(Q_{\mathrm{fin}}^\vee\) are the
coweight and coroot lattices, then
\[
\Lambda=P_{\mathrm{fin}},
\qquad
\Lambda^\vee=Q_{\mathrm{fin}}^\vee.
\]
We write \(\mathbb H=\mathbb H_G\). In this algebra the \(X\)-variables are
indexed by \(P_{\mathrm{fin}}\), while the \(Y\)-variables are indexed by
\(Q_{\mathrm{fin}}^\vee\). We call the presentation obtained by applying
Definition~\ref{def:daha} to the dual root datum the \(X\)-presentation, and
the presentation in Proposition~\ref{prop:tygamma} the \(Y\)-presentation.
Since \(\Lambda^\vee/Q_{\mathrm{fin}}^\vee\) is trivial, the \(X\)-presentation
is generated by the \(T_i\)'s and the \(X_\lambda\)'s.

\begin{defi}\label{def:upsilon}
    Let $\Upsilon:\mathbb H\to\mathbb H$ be the semilinear algebra involution
    given on the \(X\)-presentation of \(\mathbb H_G\) by
    \[
    \Upsilon(T_i)=T_i^{-1},
    \qquad
    \Upsilon(X_\lambda)=X_{-\lambda}
    \]
    for $0\leqslant i\leqslant n$ and \(\lambda\in\Lambda\), with $\Upsilon(v) = v^{-1}, \Upsilon(q) = q^{-1}$.
    Since
    \[
    \gamma_r=\varphi(\pi_{r^*})=T_{u_r}^{-1}X_{r^*}^{-1},
    \]
    the same involution acts on the $\gamma$-generators by
    \[
    \Upsilon(\gamma_r)
    =
    T_{u_{r^*}}X_{r^*}.
    \]
\end{defi}

We also fix the localization that will be used for the rest of this section.
Let
\[
F=\Ring_q[Y_{\Lambda^\vee}],
\qquad
\Ring_q[Y_{\Lambda^\vee}]_{\mathrm{loc}}=F_{\mathrm{loc}}
=
F\left[
\frac{1}{1-Y_{\beta^\vee}},
\frac{1}{1-v^2Y_{\beta^\vee}}
\right]_{\beta\in\Phi}.
\]
The extended affine Hecke algebra $\mathcal H_{\mathrm{ex}}$ is the subalgebra of $\mathbb H$ generated by
$F$ and the finite Hecke generators $\{T_i\}_{i=1}^n$. Thus we will use the same notation $\Upsilon$ for the restriction of this involution to $\mathcal H_{\mathrm{ex}}$ whenever we are working only with this subalgebra rather than the full double affine Hecke algebra. We set
\[
(\mathcal H_{\mathrm{ex}})_{\mathrm{loc}}
=
F_{\mathrm{loc}}\otimes_F\mathcal H_{\mathrm{ex}}
\]
and use the isomorphism
\[
(\mathcal H_{\mathrm{ex}})_{\mathrm{loc}}\cong F_{\mathrm{loc}}\rtimes W_{\mathrm f},
\]
denoting by $\eta_w$ the algebra element corresponding to any $w\in W_{\mathrm f}$.
We choose this isomorphism by the convention that if $f\in F_{\mathrm{loc}}$ then
$\eta_w f=f^w\eta_w$, where \(f^w\) denotes the image of $f$ under the action of $w$ on $F_{\mathrm{loc}}$. Under this isomorphism the element corresponding to a
simple reflection $s_i$ is
\begin{equation}\label{eq:eta-simple-localization}
\eta_{s_i}
=
v\,\frac{1-Y_{-\alpha_i^\vee}}{1-v^2Y_{-\alpha_i^\vee}}
\left(T_i-\frac{\hbar}{1-Y_{-\alpha_i^\vee}}\right),
\qquad 1\leqslant i\leqslant n.
\end{equation}
Equivalently,
\begin{equation}\label{eq:localized-Ti}
T_i=\frac{\hbar}{1-Y_{-\alpha_i^\vee}}
    +v^{-1}\frac{1-v^2Y_{-\alpha_i^\vee}}{1-Y_{-\alpha_i^\vee}}\,\eta_{s_i}.
\end{equation}
Let \((\Phi^\vee)_{\mathrm{aff}}^{\mathrm{re}}\) denote the set of real affine coroots,
and set
\[
F_{\mathrm{aff,loc}}
=
F\left[
\frac{1}{1-Y_{\widetilde\beta^\vee}},
\frac{1}{1-v^2Y_{\widetilde\beta^\vee}}
\right]_{\widetilde\beta^\vee\in(\Phi^\vee)_{\mathrm{aff}}^{\mathrm{re}}},
\]
where \(Y_{\beta^\vee+m\delta}\) is interpreted using the convention in
\eqref{eq:Y-affine-abbreviation}. We write
\[
\mathbb H_{\mathrm{loc}}
=
F_{\mathrm{aff,loc}}\otimes_F\mathbb H.
\]
The algebra $(\mathcal H_{\mathrm{ex}})_{\mathrm{loc}}$ identifies with the
subalgebra of $\mathbb H_{\mathrm{loc}}$ generated by the finite Hecke algebra
and the $Y$-lattice, via the inclusion \(F_{\mathrm{loc}}\subset
F_{\mathrm{aff,loc}}\).

\begin{Rem}\label{rem:dolbeault-localization}
Geometrically, this localization should be {basically} viewed as fixed-point localization
in equivariant \(K\)-theory of the affine Steinberg variety for \(LG^\vee\);
see \cite[\S 5.11]{CG} and \cite[\S 3.2]{DMB}. The normal weights at the
fixed points are the real affine roots, and the cotangent scaling contributes
the corresponding \(v^2\)-shift. Thus the Euler classes which are inverted have
the form \(1-Y_{\widetilde\beta^\vee}\) and \(1-v^2Y_{\widetilde\beta^\vee}\) for real
affine coroots \(\widetilde\beta^\vee\).
\end{Rem}

\subsection{Heuristic derivation of the bar operation formula}\label{sec:heuristic}
%
We put
\[
P_{\mathrm{fin},\mathbb R}=P_{\mathrm{fin}}\otimes_{\mathbb Z}\mathbb R,
\qquad
Q_{\mathrm{fin},\mathbb R}^\vee
=Q_{\mathrm{fin}}^\vee\otimes_{\mathbb Z}\mathbb R.
\]
In this simply connected setting, the affine Weyl group acting by the
coroot translations used below is
\[
W_{\mathrm{aff}}=W_{\mathrm f}\ltimes Q_{\mathrm{fin}}^\vee.
\]

Let $\mathcal{A}_0$ be the set of alcoves in $\Lambda_{\mathbb R}$. Then the level 0 alcoves in Remark \ref{Rem:alcove_dependence} are of the form $A\times \mathbb{R}\delta$.
We use the alcoves \(A^+,A^-\in\mathcal A_0\) fixed in
\eqref{eq:positive-negative-alcoves}. 

We now choose a point
$\lambda$ strictly dominant in $P_{\mathbb{R}}$, the dual affine Cartan (of dimension $\dim \Lambda_{\mathbb R}+2$), and close to zero and set
$\lambda_\epsilon=\epsilon\lambda$ with $\epsilon>0$. Let $\mu$ be strictly
dominant in the finite coroot lattice, and let $t_\mu$ be the corresponding
positive translation in \(W_{\mathrm{aff}}\). Applying
\eqref{eq:twisted-bimodule-right-action} with $x_1=1$, $x_2=t_{-\mu}$ gives
\begin{equation}\label{eq:limiting_bone}
\bone_!^1(t_{-\mu}\lambda_\epsilon)
=
\bone_!^{t_{-\mu}}(\lambda_\epsilon)\,\nabla_{t_{-\mu}}(\lambda_\epsilon)^{-1}.
\end{equation}
Let $\mathcal{J}\subset W$ be a finite poset ideal, and set
\[
Q_{\mathrm{fin,dom}}^\vee
=
\left\{
\mu\in Q_{\mathrm{fin}}^\vee
\ \middle|\
\langle\alpha_i,\mu\rangle\geqslant 0\text{ for all }1\leqslant i\leqslant n
\right\}.
\]
We write \(\mu\to+\infty\) if
\(\langle\alpha_i,\mu\rangle\to+\infty\) for every \(1\leqslant i\leqslant n\). Note that the lattice part of $W_{\mathrm{aff}}$ acts on real root
\(\beta=\bar\beta+k\delta\), with \(\bar\beta\in\Phi\) and
\(k\in\mathbb Z\),  by
\[
t_\mu\beta
=\bar\beta+\bigl(k-\langle\bar\beta,\mu\rangle\bigr)\delta,
\]
So, there is a strictly dominant element $\mu_0(\mathcal{J})$ such that for all
$\mu\in \mu_0(\mathcal{J})+Q_{\mathrm{fin,dom}}^\vee$ the following holds: for every positive real root $\beta$ relevant for $\mathcal{J}$, the root 
 \(t_\mu\beta\) is positive if and only if the finite root
\(\bar\beta\) is negative. For each
$\mu\in \mu_0(\mathcal{J})+Q_{\mathrm{fin,dom}}^\vee$, there is
$\epsilon(\mu)>0$ such that for all
$\epsilon\in (0,\epsilon(\mu))$, the relative interior of the level $0$
alcove $A^-$ and the element $t_{-\mu}\lambda_\epsilon$ lies in the
interior of the same $\mathcal{J}$-alcove.
According to Remark \ref{Rem:alcove_dependence},
the projections of $\bone_!^1(A^-)$ and
$\bone^1_!(t_{-\mu}\lambda_\epsilon)$ to $\mathcal{B}^{\mathcal{J}}$
coincide. It follows that
$\lim_{\mu\to+\infty}\lim_{\epsilon\to0^+}
\bone^1_!(t_{-\mu}\lambda_\epsilon)$ exists in $\mathcal{B}$ and
\begin{equation}\label{eq:bone_Aminus_formula}
\bone^1_!(A^-)
=
\lim_{\mu\to+\infty}\lim_{\epsilon\to0^+}
\bone^1_!(t_{-\mu}\lambda_\epsilon).
\end{equation}
This limit in the case of $\mathfrak{sl}_2$ is illustrated by Figure \ref{fig:sl2} (where we project $P_{\mathbb{R}}$ to the 2-dimensional plane along $\mathbb{R}\delta$).

Now we analyze the right hand side of (\ref{eq:limiting_bone}) for fixed dominant $\mu$ as $\epsilon\rightarrow 0^+$. By (\ref{eq:twisted-bimodule-action}) with $x_1=t_{-\mu},x_2=1$, $\bone^{t_{-\mu}}_!(\lambda_\epsilon)=\nabla_{t_{-\mu}}(-\lambda_\epsilon)\bone^1_!(\lambda_\epsilon)$. It follows that 
\begin{equation}\label{eq:limiting_bone1}
\bone_!^1(t_{-\mu}\lambda_\epsilon)=\nabla_{t_{-\mu}}(-\lambda_\epsilon)\bone_!^1(\lambda_\epsilon)\nabla_{t_{-\mu}}(\lambda_\epsilon)^{-1}.
\end{equation}
For fixed \(\mu\) and sufficiently small \(\epsilon>0\), the 
pairings of \(-\lambda_\epsilon\) and \(\lambda_\epsilon\) with the relevant positive roots lie respectively
in \((-1,0)\) and \((0,1)\). Hence
\eqref{eq:affine-simple-factor}--\eqref{eq:affine-twisted-costandard-general}
give
\[
\nabla_{t_{-\mu}}(-\lambda_\epsilon)=T_{t_{-\mu}},
\qquad
\nabla_{t_{-\mu}}(\lambda_\epsilon)=T_{t_\mu}^{-1}.
\]
Thus \(\nabla_{t_{-\mu}}(\lambda_\epsilon)^{-1}=T_{t_\mu}\). Finally, for each finite poset ideal $\mathcal{J}$, the projections of $\bone_!^1(\lambda_\epsilon)$ and $\bone^1_!$ to $\mathcal{B}^{\mathcal{J}}$ coincide as long as $\epsilon$ is sufficiently close to $0$. It follows that
\begin{equation}\label{eq:lim_fixed_mu}\lim_{\epsilon\rightarrow 0^+}\bone^1_!(t_{-\mu}\lambda_\epsilon)=T_{t_{-\mu}}\bone_!^1 T_{t_\mu}=T_{t_{-\mu}}T_{t_\mu}\bone_!^1.
\end{equation}

\begin{figure}[ht]
    \centering
    \includegraphics[width=0.75\textwidth]{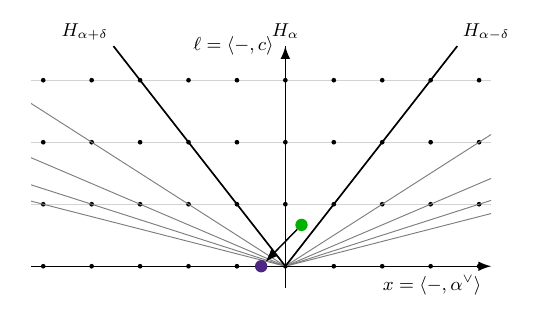}
    \caption{The limiting procedure in the affine rank $1$ case. The slanted
    lines are the real affine root hyperplanes in the positive-level parameter
    space, while the level-zero axis carries the alcoves in \(\mathcal A_0\).
    The arrow indicates the limit of a positive-level parameter (in green) to a point (in purple) in the
    level-zero alcove \(A^-\).\label{fig:sl2}}
\end{figure}
Combining (\ref{eq:bone_Aminus_formula}) with (\ref{eq:lim_fixed_mu}) we finally get
\begin{equation}\label{eqn:alimit}
\bone_!^1(A^-)
=
\lim_{\mu\to+\infty}
\left( T_{t_{-\mu}} T_{t_\mu} \right) \bone_!^1.
\end{equation}
Note that
\(T_{t_{-\mu}} = \Upsilon(Y_{-\mu})\), while
\(T_{t_\mu} = Y_\mu\). Thus the limiting operator is exactly
the element
\[
a=\lim_{\xi\to+\infty}\Upsilon(Y_{-\xi})Y_\xi.
\]
Equivalently, \(a\) is characterized by the property
\begin{equation}\label{eq:a-changes-alcove-class}
a\bone_!^1=\bone_!^1(A^-).
\end{equation}

We now recall the two deformation rules from
\S\ref{sec:twisted-generators} that motivate the rest of the argument. First,
specializing \eqref{eq:twisted-bimodule-alcove-translation} to \(x=1\) gives
the level-zero deformation rule
\begin{equation}\label{eq:levelzero-x-translation}
X_\nu\bone_!^1(A)X_\nu=\bone_!^1(A+\nu),
\qquad A\in\mathcal A_0,\ \nu\in P_{\mathrm{fin}}.
\end{equation}
Equivalently,
\begin{equation}\label{eq:levelzero-x-translation-one-sided}
X_\nu\bone_!^1(A)=\bone_!^1(A+\nu)X_{-\nu}.
\end{equation}
Second, we specialize relation
\eqref{eq:twisted-bimodule-mixed-transport}. Taking \(A=w^{-1}A^-\), we have
\(wA=A^-\). For \(w\in W_{\mathrm f}\), the relevant pairings
on \(w^{-1}A^-\) lie in \((0,1)\), so formulas (\ref{eq:affine-simple-factor}), (\ref{eq:affine-twisted-costandard-general}) give \[\nabla_w(-w^{-1}A^-)=T_w, \nabla_w(w^{-1}A^-)=T_{w^{-1}}^{-1},\] hence
\begin{equation}\label{eq:levelzero-finite-transport}
T_w\bone_!^1(w^{-1}A^-)
=
\bone_!^1(A^-)T_{w^{-1}}^{-1}
\qquad w\in W_{\mathrm f}.
\end{equation}

Now consider the dual generator $\gamma_r$, indexed so that it acts by the
length-zero element $\pi_r$ on the \(Y\)-lattice. By Definition
\ref{def:upsilon},
\[
\gamma_r=T_{u_r}^{-1}X_{r^*}^{-1},
\qquad
\Upsilon(\gamma_r)=T_{u_{r^*}}X_{r^*} .
\]
At \(v=1\), \(\Upsilon(\gamma_r)\) acts as
\(u_{r^*}t_{r^*}\), which preserves \(A^-\). Equation
\eqref{eq:negative-alcove-shift}, applied to \(r^*\), gives
\(A^-+r^*=u_{r^*}^{-1}A^-\). Combining
\eqref{eq:levelzero-x-translation-one-sided} with
\eqref{eq:levelzero-finite-transport}, we get
\[
\begin{aligned}
\Upsilon(\gamma_r)\bone_!^1(A^-)
&=T_{u_{r^*}}X_{r^*}\bone_!^1(A^-)\\
&=T_{u_{r^*}}\bone_!^1(A^-+r^*)X_{-r^*}\\
&=\bone_!^1(A^-)T_{u_r}^{-1}X_{-r^*}\\
&=\bone_!^1(A^-)\gamma_r .
\end{aligned}
\]
Thus we have
\begin{equation}\label{eq:gamma-a-minus-heuristic}
\Upsilon(\gamma_r)\bone_!^1(A^-)=\bone_!^1(A^-)\gamma_r.
\end{equation}
Substituting \eqref{eq:a-changes-alcove-class} into
\eqref{eq:gamma-a-minus-heuristic} gives the relation
\[
\Upsilon(\gamma_r)a\bone_!^1=a\bone_!^1\gamma_r
\]
in the topological module $\mathcal{B}$. Equivalently, after conjugating back from
\(\bone_!^1(A^-)\) to \(\bone_!^1\), this suggests introducing the  formal element
\begin{equation}
\omega_0(\gamma_r) = a^{-1} \Upsilon(\gamma_r) a,
\end{equation}
which one would then expect to satisfy
\begin{equation}
    \omega_0(\gamma_r)\bone_!^1 = \bone_!^1\gamma_r.
\end{equation}
Note that the use of $a^{-1}$ here means as of yet that the map $\omega_0$ is only defined heuristically; the preceding discussion motivates the semilinear map $\omega_0$ in what follows by
\[
\begin{aligned}
\omega_0(v)&=v, & \omega_0(q)&=q^{-1},\\
\omega_0(T_i)&=T_i\quad(1\leqslant i\leqslant n), & \omega_0(Y_b)&=Y_b,\\
\omega_0(\gamma_r)&=a^{-1}\Upsilon(\gamma_r)a.
\end{aligned}
\]
which will soon become precise once we explain an explicit formula for $a$ and its inverse.

The actual bar operation will then use the inverse map, and we will define
\[
\overline{x}=\omega_0^{-1}\circ\Upsilon(x).
\]
With this convention the \(\gamma\)-generators transform as
\(\overline{\gamma_r}=a\gamma_r a^{-1}\), matching the normalization of
\(\bone_!^1\).
In \S\ref{sec:levelzero-a-new}--\S\ref{sec:levelzero-involution}, we make this precise: first we compute $a$ inside the
localized algebra, and then we deduce these formulas purely from relations in $\mathbb H_{\mathrm{loc}}$.

\subsection{A formula for the element $a$}\label{sec:levelzero-a-new}

We now provide an explicit formula for the limit element $a$ which was motivated in the previous subsection. To do so, we first make precise the appropriate completion in which we consider this limit. Set
\[
Q_{\mathrm{fin},+}^\vee:=\sum_{i=1}^n\mathbb Z_{\geqslant 0}\alpha_i^\vee.
\]
We use the dominant completion
\[
\widehat F^{+}
=
\left\{
\sum_{\lambda\in Q_{\mathrm{fin}}^\vee} c_\lambda Y_\lambda\ \middle|\
c_\lambda\in\Ring_q,\
\{\lambda\mid c_\lambda\neq 0\}\subset
\bigcup_{j=1}^m(\lambda_j+Q_{\mathrm{fin},+}^\vee)\text{ for some }\lambda_j\in Q_{\mathrm{fin}}^\vee
\right\}.
\]
{ We note that $\widehat{F}^+$ is the union of complete topological submodules corresponding to fixed $\lambda_1,\ldots,\lambda_m$.}
The submodules
consisting of series supported on coweights \(\lambda\) with
\(\langle\rho,\lambda\rangle\geqslant N\) form a fundamental system of
neighborhoods of zero, where $\rho$ has the usual meaning, i.e., half the sum of all positive roots.

We embed \(F_{\mathrm{loc}}\) into \(\widehat F^{+}\) by expanding every
denominator in the positive root direction:
\[
\frac{1}{1-cY_{\alpha^\vee}}=\sum_{m\geqslant 0}c^mY_{m\alpha^\vee},
\qquad
\frac{1}{1-cY_{-\alpha^\vee}}=-\sum_{m\geqslant 1}c^{-m}Y_{m\alpha^\vee}
\]
for any $\alpha\in\Phi^+$ and $c = q^n$ or $v^2q^n$ for $n \in \mathbb{Z}$.
Thus \((\mathcal H_{\mathrm{ex}})_{\mathrm{loc}}\), writing elements in the form \(\sum_{x\in W_{\mathrm f}}f_x\eta_x\), maps to the completed
module
\[
\widehat{\mathcal H}_{\mathrm{ex}}^{+}
=
\bigoplus_{x\in W_{\mathrm f}}\widehat F^{+}\eta_x.
\]
We write \(T_{w_0}\widehat{\mathcal H}_{\mathrm{ex}}^+\) for its left
translate, equipped with the topology transported by the bijection
\(h\mapsto T_{w_0}h\).

\begin{Prop}\label{prop:a-limit-formula}
Let \(\xi\in Q_{\mathrm{fin,dom}}^\vee\) tend to infinity deep in the
dominant chamber, in the sense that
\(\langle\alpha_i,\xi\rangle\to+\infty\) for every \(1\leqslant i\leqslant n\).
Then the limit
\[
\lim_{\xi\to+\infty}\Upsilon(Y_{-\xi})Y_\xi
\]
exists in \(T_{w_0}\widehat{\mathcal H}_{\mathrm{ex}}^{+}\), and is represented
by the image of the following element of
\((\mathcal H_{\mathrm{ex}})_{\mathrm{loc}}\):
\[
a = v^{-\ell(w_0)}T_{w_0}\eta_{w_0}
\prod_{\alpha\in\Phi^+}
\frac{1-v^2Y_{\alpha^\vee}}{1-Y_{\alpha^\vee}}.
\]
\end{Prop}

\begin{proof}
By the affine Hecke bar formula \eqref{eqn:bar-ylambda}, we can write
\[
\Upsilon(Y_{-\xi})Y_\xi
=
T_{w_0}Y_{-w_0\xi}T_{w_0}^{-1}Y_\xi.
\]
By the definition of the translated completion, it is enough to compute the
limit of \(Y_{-w_0\xi}T_{w_0}^{-1}Y_\xi\) in
\(\widehat{\mathcal H}_{\mathrm{ex}}^{+}\) as
\(\xi\to+\infty\).

Set
\[
R_\alpha=\frac{1-v^2Y_{\alpha^\vee}}{1-Y_{\alpha^\vee}},
\]
and note that for any $1 \leqslant i\leqslant n$, (\ref{eq:eta-simple-localization}) implies
\begin{equation}
    T_i^{-1}
=
-\frac{\hbar}{1-Y_{\alpha_i^\vee}}
+
v^{-1}R_{-\alpha_i}\eta_{s_i}.\label{eqn:7first}
\end{equation}

Choose a reduced expression  $w_0=s_{i_1}\cdots s_{i_N},$ then
\[
T_{w_0}^{-1}=T_{i_N}^{-1}\cdots T_{i_1}^{-1}.
\]
Expanding this product using (\ref{eqn:7first}), the unique term obtained by choosing the
\(\eta\)-summand at every step can be written explicitly as follows. Let
\[
\beta_1=\alpha_{i_N},\quad
\beta_2=s_{i_N}\alpha_{i_{N-1}},\quad \ldots,\quad
\beta_N=s_{i_N}\cdots s_{i_2}\alpha_{i_1}.
\]
Using \(\eta_w f=f^w\eta_w\), the contribution of these summands is
\[
v^{-\ell(w_0)}R_{-\beta_1}R_{-\beta_2}\cdots R_{-\beta_N}\,\eta_{w_0}.
\]
The sequence \(\beta_1,\ldots,\beta_N\) is the ordering of
\(\Phi^+\) determined by the chosen reduced expression. Since the factors
\(R_{-\beta_j}\) lie in the commutative algebra \(F_{\mathrm{loc}}\), this
coefficient is
\[
v^{-\ell(w_0)}\prod_{\alpha\in\Phi^+}R_{-\alpha}.
\]
All other \(\eta_x\) terms which occur have \(x<w_0\). Hence
\begin{equation}
T_{w_0}^{-1}
=
v^{-\ell(w_0)}\prod_{\alpha\in\Phi^+}R_{-\alpha}\,\eta_{w_0}
+
\sum_{x<w_0} f_x\eta_x\label{eqn:7second}
\end{equation}
for some \(f_x\in F_{\mathrm{loc}}\).

Substituting (\ref{eqn:7second}), and using \(\eta_xY_\xi=Y_{x\xi}\eta_x\) gives
\begin{equation}
    Y_{-w_0\xi}T_{w_0}^{-1}Y_\xi
=
v^{-\ell(w_0)}\prod_{\alpha\in\Phi^+}R_{-\alpha}\,\eta_{w_0}
+
\sum_{x<w_0}Y_{x\xi-w_0\xi}f_x\eta_x.\label{eqn:7third}
\end{equation}
For \(x<w_0\), the difference \(x\xi-w_0\xi\) lies in
\(Q_{\mathrm{fin},+}^\vee\). Put \(z=x^{-1}w_0\neq 1\). Since
\(w_0\rho=-\rho\), we have
\[
\langle\rho,x\xi-w_0\xi\rangle
=
\langle\rho-z\rho,\xi\rangle.
\]
The weight \(\rho-z\rho\) is a nonzero sum of positive roots, so the
deep-dominant condition on \(\xi\) implies that this quantity tends to
\(+\infty\). Thus multiplication by
\(Y_{x\xi-w_0\xi}\) shifts the support of \(f_x\) arbitrarily far in the
positive direction, so every term labeled by $x<w_0$ in the sum in (\ref{eqn:7third}) tends to zero in
\(\widehat{\mathcal H}_{\mathrm{ex}}^{+}\).
Therefore
\[
\lim_{\xi\to+\infty}
Y_{-w_0\xi}T_{w_0}^{-1}Y_\xi
=
v^{-\ell(w_0)}\prod_{\alpha\in\Phi^+}R_{-\alpha}\,\eta_{w_0}.
\]
Finally,
\[
\left(\prod_{\alpha\in\Phi^+}R_{-\alpha}\right)\eta_{w_0}
=
\eta_{w_0}\prod_{\alpha\in\Phi^+}R_\alpha.
\]
Multiplying on the left by \(T_{w_0}\), in the translated completion
\(T_{w_0}\widehat{\mathcal H}_{\mathrm{ex}}^{+}\), we obtain
\[
\lim_{\xi\to+\infty}\Upsilon(Y_{-\xi})Y_\xi
=
v^{-\ell(w_0)}T_{w_0}\eta_{w_0}
\prod_{\alpha\in\Phi^+}
\frac{1-v^2Y_{\alpha^\vee}}{1-Y_{\alpha^\vee}},
\]
as claimed.
\end{proof}

\begin{defi}\label{def:levelzero-a}
Set
\[
a:=
v^{-\ell(w_0)}T_{w_0}\eta_{w_0}
\prod_{\alpha\in\Phi^+}
\frac{1-v^2Y_{\alpha^\vee}}{1-Y_{\alpha^\vee}}
\in
(\mathcal H_{\mathrm{ex}})_{\mathrm{loc}}.
\]
\end{defi}

\begin{Ex}\label{ex:rank-one-a}
If $\Phi$ has rank $1$, write its positive root as $\alpha$. Then
$w_0=s_1$, and Definition~\ref{def:levelzero-a} gives
\[
a
=
v^{-1}T_1\eta_{s_1}\frac{1-v^2Y_{\alpha^\vee}}{1-Y_{\alpha^\vee}}.
\]
By the Bernstein relation, this is
\[
a = 1+T_1\frac{\hbar}{1-Y_{\alpha^\vee}}.
\]
Expanding the rational function in the dominant completion and using the
\(X\)-presentation identity $Y_{\alpha^\vee}=T_0T_1$ gives the semi-infinite expression
\begin{align*}
a
&=
1+\hbar\sum_{m\geqslant 0}T_1Y_{\alpha^\vee}^m\\
&=
1+\hbar\bigl(T_1+T_1T_0T_1+T_1T_0T_1T_0T_1+\cdots\bigr)\\
&=
\lim_{n\to\infty}(T_1T_0)^n(T_0T_1)^n
=
\lim_{n\to\infty}\Upsilon(Y_{-n\alpha^\vee})Y_{n\alpha^\vee},
\end{align*}
which is the rank-one version of the limiting picture above.
This spells out the conclusion of Proposition~\ref{prop:a-limit-formula} explicitly in the rank $1$ setting.
\end{Ex}

\begin{Rem}
    It was explained to us by Quan Situ that the formula for \(a\) in Definition \ref{def:levelzero-a} can also be
    obtained from a natural bar involution on Lusztig's periodic Hecke module from \cite{Lgeneric}. This
    involution is introduced in \cite[Theorem~11.2]{LPeriodic}, and writing an explicit formula for this operator on elements arising from the natural inclusion of the affine Hecke algebra gives a derivation which is alternative to the proof in Proposition \ref{prop:a-limit-formula}.
\end{Rem}

\subsection{Explicit formula in the localized algebra}\label{sec:localized-formula}
We use the same symbol $\Upsilon$ when referring to the natural extension to $\mathbb H_{\mathrm{loc}}$.

The next proposition isolates the cases in which the present construction of $\bar{\bullet}$ on the full level-zero algebra can be carried out (ultimately, this is the condition that allows us to eliminate
\(\varphi(T_0)\) from Cherednik's dual presentation; see Remark \ref{rem:not0}.)

\begin{Prop}\label{prop:levelzero-generators}
Assume $O'\neq\varnothing$, and choose \(r_0\in O'\). Set
\[
\tau=\gamma_{r_0}^{-1}T_{r_0}\gamma_{r_0}.
\]
Then $\mathbb H$ is presented over $\Ring_q$ by the generators
\[
\{T_i\}_{i=1}^n,\qquad \{Y_b\}_{b\in Q_{\mathrm{fin}}^\vee},\qquad \{\gamma_r\}_{r\in O'},
\]
with the relations of Proposition~\ref{prop:tygamma} after replacing every
occurrence of \(\varphi(T_0)\) by \(\tau\).
The algebra \(\mathbb H_{\mathrm{loc}}\) is obtained from the same
presentation by localizing the subalgebra generated by the elements $Y_b$.
\end{Prop}
\begin{proof}
Relation~(\ref{rel:tygamma-diagram})
for \(r_0\) gives
\[
\gamma_{r_0}\varphi(T_0)\gamma_{r_0}^{-1}=T_{r_0},
\]
since \(\pi_{r_0}(0)=r_0\), and hence
\[
\varphi(T_0)=\gamma_{r_0}^{-1}T_{r_0}\gamma_{r_0}=\tau.
\]
Thus the substituted presentation in Proposition \ref{prop:tygamma} maps to \(\mathbb H\), while the full
presentation maps back by sending \(\varphi(T_0)\) to \(\tau\) and fixing all remaining
generators. These two maps are inverse on generators, so the presentations are
isomorphic.
It remains to check that this presentation is compatible with the localization
used above. Let \(S\subset F\) be the multiplicative set generated by the
elements
\[
1-Y_{\widetilde\beta^\vee},\qquad 1-v^2Y_{\widetilde\beta^\vee}
\qquad
(\widetilde\beta^\vee\in(\Phi^\vee)_{\mathrm{aff}}^{\mathrm{re}}),
\]
so that \(S^{-1}F=F_{\mathrm{aff,loc}}\). The set \(S\) is stable under the
simple affine reflections and under the diagram automorphisms coming from
\(\Pi\). Hence the conjugation relations for the \(\gamma_r\)'s preserve
\(S^{-1}F\). The Bernstein relation~(\ref{rel:tygamma-bernstein}) also extends to
this localization: if \(g\in S\), then multiplying
\[
gT_i-T_ig^{s_i}
=
\hbar\,\frac{g-g^{s_i}}{1-Y_{-\alpha_i^\vee}}
\]
on the left by \(g^{-1}\) and on the right by \((g^{s_i})^{-1}\) gives
\[
g^{-1}T_i-T_i(g^{s_i})^{-1}
=
\hbar\,\frac{g^{-1}-(g^{s_i})^{-1}}{1-Y_{-\alpha_i^\vee}}.
\]
Thus the same Bernstein relation makes sense with coefficients in
\(F_{\mathrm{aff,loc}}\), and the localized algebra is obtained by replacing
the coefficient algebra \(F\) by \(F_{\mathrm{aff,loc}}\) in the presentation
above. Therefore localizing over \(F\) gives the asserted presentation of
\(\mathbb H_{\mathrm{loc}}\).
\end{proof}

From now on we assume $O'\neq\varnothing$, and all statements are made inside the localized algebra $\mathbb H_{\mathrm{loc}}$. We first establish the identities needed to define the bar operation.

\begin{Lem}\label{lem:a-y-relation}
For any $b \in Q_{\mathrm{fin}}^\vee$, we have the relation
\[
a Y_b  = \Upsilon(Y_b) a.
\]
\end{Lem}
\begin{proof}
Set
\[
R:=\prod_{\alpha\in\Phi^+}\frac{1-v^2Y_{\alpha^\vee}}{1-Y_{\alpha^\vee}}.
\]
By Definition~\ref{def:levelzero-a}, we have
\[
a=v^{-\ell(w_0)}T_{w_0}\eta_{w_0}R.
\]
Since \(R\in F_{\mathrm{loc}}\), it commutes with \(Y_b\). Using the relation \(\eta_w f=f^w\eta_w\), we compute
\begin{align*}
a Y_b
&=v^{-\ell(w_0)}T_{w_0}\eta_{w_0}RY_b\\
&=v^{-\ell(w_0)}T_{w_0}\eta_{w_0}Y_bR\\
&=v^{-\ell(w_0)}T_{w_0}Y_{w_0b}\eta_{w_0}R.
\end{align*}
On the other hand, by
\[
\Upsilon(Y_b)=T_{w_0}Y_{w_0b}T_{w_0}^{-1},
\]
we have
\begin{align*}
\Upsilon(Y_b)a
&=T_{w_0}Y_{w_0b}T_{w_0}^{-1}
  v^{-\ell(w_0)}T_{w_0}\eta_{w_0}R\\
&=v^{-\ell(w_0)}T_{w_0}Y_{w_0b}\eta_{w_0}R;
\end{align*}
the expressions for $aY_b, \Upsilon(Y_b)a$ agree.
\end{proof}

\begin{Cor}\label{cor:a-y-inverse-relations}
For any \(b\in Q_{\mathrm{fin}}^\vee\), one has
\[
Y_b=a^{-1}\Upsilon(Y_b)a,
\qquad
\Upsilon(a)^{-1}Y_b=\Upsilon(Y_b)\Upsilon(a)^{-1}.
\]
\end{Cor}
\begin{proof}
The first identity is Lemma~\ref{lem:a-y-relation} multiplied by \(a^{-1}\)
on the left. Applying \(\Upsilon\) to Lemma~\ref{lem:a-y-relation} gives
\(\Upsilon(a)\Upsilon(Y_b)=Y_b\Upsilon(a)\), which is equivalent to the
second identity.
\end{proof}

We introduce the following shorthand, which will be convenient in
Lemma~\ref{lem:t-a-relation} and the proof of
Proposition~\ref{prop:omega0-endomorphism}.
\begin{defi}
For $1\leqslant i\leqslant n$, set
\[
f_i=\frac{1+Y_{\alpha_i^\vee}}{1-Y_{\alpha_i^\vee}}\in F_{\mathrm{loc}}.
\]
\end{defi}

\begin{Lem}\label{lem:t-a-relation}
For any $i$ with $1 \leqslant i \leqslant n$,
\begin{equation}
    T_i a = a \left( T_i + \hbar f_i \right).
\end{equation}
\end{Lem}

\begin{proof}
    Let $i'$ be the index such that $s_{i'}=w_0s_iw_0$. Then
    $T_iT_{w_0}=T_{w_0}T_{i'}$, while
    $w_0(\alpha_{i'})=-\alpha_i$ and
    $\eta_{s_{i'}}\eta_{w_0}=\eta_{w_0}\eta_{s_i}$.
    Starting with the definition of $a$, we compute $T_i a$ by commuting $T_i$ through the expression for $a$:
\begin{align*}
    T_i a &= v^{-\ell(w_0)} T_i T_{w_0} \eta_{w_0} \prod_{\alpha \in \Phi^+} \frac{1 - v^2 Y_{\alpha^\vee}}{1 - Y_{\alpha^\vee}} \\
    &= v^{-\ell(w_0)} T_{w_0} T_{i'} \eta_{w_0} \prod_{\alpha \in \Phi^+} \frac{1 - v^2 Y_{\alpha^\vee}}{1 - Y_{\alpha^\vee}} \\
    &= v^{-\ell(w_0)} T_{w_0} \left( \frac{\hbar}{1 - Y_{-\alpha_{i'}^\vee}} + v^{-1} \frac{1 - v^2 Y_{-\alpha_{i'}^\vee}}{1 - Y_{-\alpha_{i'}^\vee}} \eta_{s_{i'}} \right) \eta_{w_0} \prod_{\alpha \in \Phi^+} \frac{1 - v^2 Y_{\alpha^\vee}}{1 - Y_{\alpha^\vee}} \\
    &= v^{-\ell(w_0)} T_{w_0} \eta_{w_0} \left( \frac{\hbar}{1 - Y_{\alpha_i^\vee}} + v^{-1} \frac{1 - v^2 Y_{\alpha_i^\vee}}{1 - Y_{\alpha_i^\vee}} \eta_{s_i} \right) \prod_{\alpha \in \Phi^+} \frac{1 - v^2 Y_{\alpha^\vee}}{1 - Y_{\alpha^\vee}} \\
    &= v^{-\ell(w_0)} T_{w_0} \eta_{w_0} \left( \prod_{\alpha \in \Phi^+ \setminus \{\alpha_i\}} \frac{1 - v^2 Y_{\alpha^\vee}}{1 - Y_{\alpha^\vee}} \right) \\
    & \quad \cdot \left( \frac{\hbar}{1 - Y_{\alpha_i^\vee}} + v^{-1} \frac{1 - v^2 Y_{\alpha_i^\vee}}{1 - Y_{\alpha_i^\vee}} \eta_{s_i} \right) \frac{1 - v^2 Y_{\alpha_i^\vee}}{1 - Y_{\alpha_i^\vee}} \\
    &= a \frac{1 - Y_{\alpha_i^\vee}}{1 - v^2 Y_{\alpha_i^\vee}} \left( \frac{\hbar}{1 - Y_{\alpha_i^\vee}} + v^{-1} \frac{1 - v^2 Y_{\alpha_i^\vee}}{1 - Y_{\alpha_i^\vee}} \eta_{s_i} \right) \frac{1 - v^2 Y_{\alpha_i^\vee}}{1 - Y_{\alpha_i^\vee}}.
\end{align*}

To complete the proof, we simplify the term to the right of $a$:
\begin{align*}
    &\frac{1 - Y_{\alpha_i^\vee}}{1 - v^2 Y_{\alpha_i^\vee}} \left( \frac{\hbar}{1 - Y_{\alpha_i^\vee}} + v^{-1} \frac{1 - v^2 Y_{\alpha_i^\vee}}{1 - Y_{\alpha_i^\vee}} \eta_{s_i} \right) \frac{1 - v^2 Y_{\alpha_i^\vee}}{1 - Y_{\alpha_i^\vee}} \\
    &\quad = \frac{\hbar}{1 - Y_{\alpha_i^\vee}} + v^{-1} \eta_{s_i} \frac{1 - v^2 Y_{\alpha_i^\vee}}{1 - Y_{\alpha_i^\vee}} \\
    &\quad = \hbar \frac{1 + Y_{\alpha_i^\vee}}{1 - Y_{\alpha_i^\vee}} + T_i.
\end{align*}
\end{proof}

\begin{Cor}\label{cor:a-t-relation}
    For any $i \in \{1, \dots, n\}$, we have
    \begin{align}
        aT_i & = (T_i - \hbar \Upsilon(f_i))a.
    \end{align}
\end{Cor}
\begin{proof}
    By Lemma~\ref{lem:t-a-relation}, we have
    \begin{align}
        aT_i & = T_ia - \hbar af_i.
    \end{align}
    Lemma~\ref{lem:a-y-relation} gives $af_i=\Upsilon(f_i)a$, so the claim follows.
\end{proof}

\begin{Lem}\label{lem:eliminated-t0-relations}
Let \(A\) be an \(\Ring_q\)-algebra with elements \(T_i\) for \(1\leqslant i\leqslant n\),
\(Y_b\) for \(b\in Q_{\mathrm{fin}}^\vee\), and invertible elements \(\gamma_r\) for \(r\in O'\).
Fix \(r_0\in O'\), set
\[
\tau=\gamma_{r_0}^{-1}T_{r_0}\gamma_{r_0},
\qquad
\gamma_0=1,
\]
and suppose that these elements satisfy the relations in
Proposition~\ref{prop:tygamma} that do not involve \(\varphi(T_0)\), together with
\begin{equation}\label{eq:eliminated-affine-conjugation}
\gamma_r\tau\gamma_r^{-1}=T_r\qquad(r\in O').
\end{equation}
If \(O'=\{r\}\) and \(m_{0r}<\infty\), suppose in addition that
\(\tau\) and \(T_r\) satisfy their \(m_{0r}\)-term braid relation.
Then all relations in the presentation of Proposition~\ref{prop:tygamma} hold
in \(A\) after substituting \(\tau\) for \(\varphi(T_0)\).
\end{Lem}
\begin{proof}
The case \(i=0\) of relation~(\ref{rel:tygamma-quadratic}) follows by conjugating
the case \(i=r_0\) of relation~(\ref{rel:tygamma-quadratic}). The case \(i=0\)
of relation~(\ref{rel:tygamma-bernstein}) is obtained in the same way, using
relation~(\ref{rel:tygamma-y-conjugation}). The
case \(\pi_r(0)=r\) of relation~(\ref{rel:tygamma-diagram}) is
exactly \eqref{eq:eliminated-affine-conjugation}. If
\(\pi_s(i)=0\), then choose \(r\in O'\); relation
(\ref{rel:tygamma-gamma-product}), the instance
\((\gamma_r\gamma_s)T_i(\gamma_r\gamma_s)^{-1}=T_r\) of
relation~(\ref{rel:tygamma-diagram}), and
\eqref{eq:eliminated-affine-conjugation} give the case of
relation~(\ref{rel:tygamma-diagram}) with target \(\varphi(T_0)\). Finally, for
relation~(\ref{rel:tygamma-braid}), the only remaining case involves \(\varphi(T_0)\) and
some \(T_i\), with \(m_{0i}<\infty\). If \(|O'|>1\), choose \(r\in O'\)
with \(r^*\neq i\), and put \(j=\pi_r(i)\). Then \(1\leqslant j\leqslant n\), and
conjugating by \(\gamma_r\), using
\eqref{eq:eliminated-affine-conjugation} and
relation~(\ref{rel:tygamma-diagram}), turns this into a braid relation
involving only $T_i$ with $1\leqslant i\leqslant n$. If \(O'=\{r\}\), the same argument works
unless \(i=r^*=r\); the remaining braid relation is precisely the additional
hypothesis above.
\end{proof}

\begin{Lem}\label{lem:singleton-eliminated-braid}
Suppose \(O'=\{r\}\) and \(m_{0r}<\infty\). In
\(\mathbb H_{\mathrm{loc}}\), set
\[
\tau=\gamma_r^{-1}T_r\gamma_r,
\qquad
f_0=\frac{1+Y_{\alpha_0^\vee}}{1-Y_{\alpha_0^\vee}},
\qquad
u=\Upsilon(\gamma_r),
\]
and
\[
D_r=T_r-\hbar\Upsilon(f_r),
\qquad
D_0=\Upsilon(\tau^{-1})-\hbar\Upsilon(f_0).
\]
Then \(m_{0r}=2\), and the element
\[
G=a^{-1}ua
\]
satisfies
\[
[G^{-1}T_rG,T_r]=0.
\]
\end{Lem}
\begin{proof}
The irreducible simply connected cases with \(O'\) a singleton are
\(A_1,B_n,C_n\), and \(E_7\). In type \(A_1\), one has
\(m_{01}=\infty\). In the other three cases, the vertices \(0\) and \(r\)
are nonadjacent, so \(m_{0r}=2\).

The original braid relation gives
\([\tau^{-1},T_r^{-1}]=0\). Since the reflections \(s_0\) and \(s_r\)
fix each other's simple coroots, the Bernstein relations also give
\[
[\tau^{-1},f_r]=0,
\qquad
[f_0,T_r^{-1}]=0.
\]
Together with \([f_0,f_r]=0\), applying \(\Upsilon\) to these four
commutation relations gives
\[
[D_0,D_r]=0.
\]

Since \(\pi_r\) exchanges \(0\) and \(r\), one has
\[
\tau^{-1}=\gamma_r^{-1}T_r^{-1}\gamma_r,
\qquad
\gamma_r^{-1}f_r\gamma_r=f_0.
\]
Applying \(\Upsilon\) gives
\[
u^{-1}D_ru=D_0.
\]
Corollary~\ref{cor:a-t-relation} gives
\[
aT_ra^{-1}=D_r.
\]
Consequently,
\[
G^{-1}T_rG=a^{-1}D_0a,
\qquad
T_r=a^{-1}D_ra,
\]
so the asserted commutation relation follows from \([D_0,D_r]=0\).
\end{proof}

\begin{Prop}\label{prop:omega0-endomorphism}
    The formulas
    \begin{align}
        \omega_0(v) & = v,\qquad \omega_0(q)=q^{-1},\\
        \omega_0(T_i) & = T_i \qquad (1\leqslant i\leqslant n),\\
        \omega_0(Y_b) & = Y_b \qquad (b\in Q_{\mathrm{fin}}^\vee),\\
        \omega_0(\gamma_r) & = a^{-1}\Upsilon(\gamma_r)a \qquad (r\in O')
    \end{align}
    give a well-defined ring automorphism
    $\omega_0:\mathbb H_{\mathrm{loc}}\to \mathbb H_{\mathrm{loc}}$.
    Its inverse fixes \(v\), sends \(q\) to \(q^{-1}\), fixes the \(T_i\)'s
    and the \(Y_b\)'s, and is given on the
    \(\gamma\)-generators by
    \[
    \omega_0^{-1}(\gamma_r)
    =
    T_{u_r}^{-1}a\gamma_r^{-1}a^{-1}T_{u_{r^*}}.
    \]
\end{Prop}
\begin{proof}
    Fix \(r_0\in O'\) as in Proposition~\ref{prop:levelzero-generators}.
    By the presentation in Proposition~\ref{prop:levelzero-generators}, it is
    enough to check the relations among the generators $\{T_i\}_{i=1}^n$,
    $\{Y_b\}_{b\in Q_{\mathrm{fin}}^\vee}$, and $\{\gamma_r\}_{r\in O'}$ obtained from
    Proposition~\ref{prop:tygamma} after eliminating \(\varphi(T_0)\).
    We regard the target as an \(\Ring_q\)-algebra through the coefficient
    involution \(v\mapsto v,\ q\mapsto q^{-1}\).
    The coefficient involution \(q\mapsto q^{-1}\) preserves the affine
    localization, since
    \[
    \omega_0(Y_{\beta^\vee+m\delta})
    =
    \omega_0(q^{-m}Y_{\beta^\vee})
    =
    q^mY_{\beta^\vee}
    =
    Y_{\beta^\vee-m\delta}.
    \]
    The relations involving only the finite Hecke generators and the $Y$-lattice
    are immediate because $\omega_0$ fixes those generators. For the
    \(Y\)-conjugation relation~(\ref{rel:tygamma-y-conjugation}) of
    Proposition~\ref{prop:tygamma}, write
    \(\pi_r(b)=c+m\delta\), with \(c\in Q_{\mathrm{fin}}^\vee\).
    Then Lemma~\ref{lem:a-y-relation} gives
    \begin{align*}
        \omega_0(\gamma_r)\omega_0(Y_b) & = a^{-1}\Upsilon(\gamma_r)aY_b\\
        & = a^{-1}\Upsilon(\gamma_rY_b)a\\
        & = a^{-1}\Upsilon(Y_{\pi_r(b)}\gamma_r)a\\
        & = q^m a^{-1}\Upsilon(Y_c)\Upsilon(\gamma_r)a\\
        & = q^mY_c a^{-1}\Upsilon(\gamma_r)a\\
        & = \omega_0(Y_{\pi_r(b)})\omega_0(\gamma_r).
    \end{align*}
    The multiplication relations~(\ref{rel:tygamma-gamma-product}) among the $\gamma_r$ coming from the group
    $\Pi$ are also preserved, because
    \[
    \omega_0(\gamma_r)\omega_0(\gamma_s)
    =
    a^{-1}\Upsilon(\gamma_r\gamma_s)a.
    \]
    For the diagram relation~(\ref{rel:tygamma-diagram}) among finite simple generators, suppose
    $\pi_r(i)=j$ with $1\leqslant i,j\leqslant n$. Then
    \begin{align*}
        \omega_0(\gamma_r)\omega_0(T_i) & = a^{-1}\Upsilon(\gamma_r)aT_i\\
        & = a^{-1}\Upsilon(\gamma_r)(T_i - \hbar \Upsilon(f_i))a\\
        & = a^{-1}\Upsilon(\gamma_rT_i^{-1})a - \hbar a^{-1}\Upsilon(\gamma_rf_i)a\\
        & = a^{-1}\Upsilon(T_j^{-1}\gamma_r)a - \hbar a^{-1}\Upsilon(f_j\gamma_r)a\\
        & = a^{-1}(T_j - \hbar \Upsilon(f_j))\Upsilon(\gamma_r)a\\
        & = T_ja^{-1}\Upsilon(\gamma_r)a\\
        & = \omega_0(T_j)\omega_0(\gamma_r).
    \end{align*}
   To see that $\omega_0$ extends to a ring endomorphism, it remains, by Lemma~\ref{lem:eliminated-t0-relations}, to check the
    conjugation relation for the eliminated affine generator and, in the
    singleton case, the additional braid relation. Let
    \(G_r=\omega_0(\gamma_r)\), and put
    \[
    \tau_\omega=G_{r_0}^{-1}T_{r_0}G_{r_0}.
    \]
    For \(r\in O'\), choose \(s\in O\) with \(\pi_s=\pi_r\pi_{r_0}^{-1}\), and
    set \(G_0=1\). Since \(\pi_s(r_0)=r\), the instance
    \(G_sT_{r_0}G_s^{-1}=T_r\) of relation~(\ref{rel:tygamma-diagram}) checked
    above gives
    \[
    G_r\tau_\omega G_r^{-1}
    =(G_rG_{r_0}^{-1})T_{r_0}(G_rG_{r_0}^{-1})^{-1}
    =G_sT_{r_0}G_s^{-1}
    =T_r.
    \]
    If \(O'\) is a singleton and \(m_{0r}<\infty\),
    Lemma~\ref{lem:singleton-eliminated-braid} supplies the additional braid
    relation required in Lemma~\ref{lem:eliminated-t0-relations}.
    Thus all relations involving the substituted \(\varphi(T_0)\) are also preserved, and
    the assignment extends to an endomorphism of $\mathbb H_{\mathrm{loc}}$.

    Now we prove that $\omega_0$ is an automorphism. Since
    \[
    \Upsilon(\gamma_r)
    =
    T_{u_{r^*}}\gamma_r^{-1}T_{u_r}^{-1},
    \]
    set
    \[
    L_r=a^{-1}T_{u_{r^*}},
    \qquad
    R_r=T_{u_r}^{-1}a.
    \]
    Then \(\omega_0(\gamma_r)=L_r\gamma_r^{-1}R_r\). Since
    \(\omega_0\) fixes \(a\) and the finite Hecke generators, it fixes
    \(L_r\) and \(R_r\), and therefore
    \[
    \omega_0(R_r\gamma_r^{-1}L_r)
    =
    R_r(L_r\gamma_r^{-1}R_r)^{-1}L_r
    =
    \gamma_r.
    \]
    Together with the coefficient action and the fixed generators \(T_i\) and
    \(Y_b\), this shows that \(\omega_0\) is surjective. 
    
    To show that $\omega_0$ is an automorphism, we will show that
    \(\mathbb H_{\mathrm{loc}}\) is Noetherian, so that every surjective ring endomorphism
    is an automorphism. 

{ Recall the elements $\eta_{s_i}\in \mathbb{H}_{\mathrm{loc}}, 1\leqslant i\leqslant n$ defined
by (\ref{eq:eta-simple-localization}). We can define $\eta_{s_0}\in \mathbb{H}_{\mathrm{loc}}$ using $i=0$ in (\ref{eq:eta-simple-localization}). Similarly to \cite[Theorem 3.2.1]{Cherednik}, we have an isomorphism between the wreath product 
$F_{\mathrm{aff},\mathrm{loc}}\# W^\vee_{\mathrm{ext}}$ and $\mathbb{H}_{\mathrm{loc}}$ that is the identity on $F_{\mathrm{aff},\mathrm{loc}}$ and sends $s_i$ to $\eta_{s_i}$ for $i=0,\ldots,n$, and the length $0$ elements to the corresponding elements $\gamma_r$.}

    So, it is sufficient to prove that the wreath product is Noetherian. Note that it is a finitely generated module over its  subalgebra $F_{\mathrm{aff},\mathrm{loc}}\#Q_{\mathrm{fin}}$. The latter is an Ore localization of a q-Weyl algebra. q-Weyl algebras are Noetherian and the Ore localization preserves the property of being Noetherian, finishing the proof. 
    
    Since
    \[
    R_r\gamma_r^{-1}L_r
    =
    T_{u_r}^{-1}a\gamma_r^{-1}a^{-1}T_{u_{r^*}},
    \]
    its inverse is the one stated above.
\end{proof}

\begin{Rem}\label{rem:not0}
    Among simply connected almost simple groups, the excluded cases with $O' = \varnothing$ are precisely $F_4$, $G_2$, and $E_8$. In those cases Proposition \ref{prop:levelzero-generators} does not apply, because there are no $\pi$- or $\gamma$-generators and one must still control \(\varphi(T_0)\) directly. Our argument therefore does not define $\overline{\bullet}$ on the full localized algebra in those types. Concretely, this would require specifying
    \[
    \omega_0(\varphi(T_0)),
    \]
    or equivalently specifying its value on the standard expression for \(\varphi(T_0)\) in
    the original presentation. While such an assignment may well exist, the
    argument in \S\ref{sec:heuristic} does not suggest one explicitly.
\end{Rem}

\begin{defi}\label{def:levelzero-bar}
    For $h \in \mathbb{H}_{\mathrm{loc}}$, set
    \[
    \overline{h} = (\omega_0^{-1} \circ \Upsilon)(h).
    \]
    On coefficients and the generators from
    Proposition~\ref{prop:levelzero-generators}, this gives
    \[
    \overline v=v^{-1},
    \qquad
    \overline q=q,
    \qquad
    \overline{T_i}=T_i^{-1},
    \qquad
    \overline{Y_b}=\Upsilon(Y_b),
    \qquad
    \overline{\gamma_r}=a\gamma_r a^{-1}.
    \]
    The last identity follows from
    \(\omega_0(a\gamma_r a^{-1})=\Upsilon(\gamma_r)\), since \(\omega_0\)
    fixes \(a\).
\end{defi}

\subsection{The localized level-zero bar is an involution at $q=1$}\label{sec:levelzero-involution}

Assume throughout this subsection that $O'\neq\varnothing$, so that
Proposition~\ref{prop:omega0-endomorphism} defines the automorphism
\[
\omega_0:\mathbb H_{\mathrm{loc}}\to \mathbb H_{\mathrm{loc}}.
\]
Set
\[
c=\Upsilon(a)a\in (\mathcal H_{\mathrm{ex}})_{\mathrm{loc}}=F_{\mathrm{loc}}\rtimes W_{\mathrm f}.
\]
We will show that the square of $\overline{\bullet}=\omega_0^{-1}\circ\Upsilon$ is the automorphism \(x\mapsto cxc^{-1}\), and that for $q=1$ this conjugation becomes trivial.

\begin{Thm}\label{thm:levelzero-q1-involution}
Assume $O'\neq\varnothing$, and let $\overline{\bullet}$ be the automorphism of $\mathbb H_{\mathrm{loc}}$ defined above. Then the element \(c=\Upsilon(a)a\) lies in $F_{\mathrm{loc}}^{W_{\mathrm f}}$. Moreover, on the generators of Proposition~\ref{prop:levelzero-generators} one has
\[
\overline{\overline{T_i}}=T_i,
\qquad
\overline{\overline{Y_b}}=Y_b,
\qquad
\overline{\overline{\gamma_r}}=c\gamma_r c^{-1}.
\]
Consequently, $\overline{\bullet}^{\,2}$ is the inner automorphism \(x\mapsto cxc^{-1}\) on all of $\mathbb H_{\mathrm{loc}}$. The automorphism $\overline{\bullet}$ is an involution if and only if
\[
\pi_r(c)=c\qquad\text{for all }r\in O'.
\]
In particular, \(\overline{\bullet}\) descends to the specialization
\(\mathbb H_{\mathrm{loc}}/(q-1)\), and there it is an involution.
\end{Thm}

\begin{proof}
From Lemma~\ref{lem:a-y-relation}, for $b\in Q_{\mathrm{fin}}^\vee$ we can deduce
\[
\Upsilon(a)aY_b=\Upsilon(a)\Upsilon(Y_b)a=Y_b\Upsilon(a)a.
\]
Writing
\[
\Upsilon(a)a=\sum_{w\in W_{\mathrm f}} g_w\eta_w\qquad (g_w\in F_{\mathrm{loc}}),
\]
we can compare coefficients in
\[
\Upsilon(a)aY_b=Y_b\Upsilon(a)a.
\]
If $w\neq 1$, choose $b\in Q_{\mathrm{fin}}^\vee$ with $wb\neq b$; then the coefficient of $\eta_w$ gives
$g_w(Y_{wb}-Y_b)=0$. Since $F_{\mathrm{loc}}$ is a localization of a domain, this
forces $g_w=0$. Thus in fact \(\Upsilon(a)a\in F_{\mathrm{loc}}\).

Next, apply $\Upsilon$ to Lemma~\ref{lem:t-a-relation}. Using $\Upsilon(T_i)=T_i^{-1}=T_i-\hbar$, we obtain
\[
T_i\Upsilon(a)=\Upsilon(a)\bigl(T_i-\hbar\,\Upsilon(f_i)\bigr).
\]
Combining this with Corollary~\ref{cor:a-t-relation}, namely
\[
aT_i=\bigl(T_i-\hbar\,\Upsilon(f_i)\bigr)a,
\]
we get
\[
T_i\Upsilon(a)a=\Upsilon(a)\bigl(T_i-\hbar\,\Upsilon(f_i)\bigr)a=\Upsilon(a)aT_i.
\]
Now \(\Upsilon(a)a\in F_{\mathrm{loc}}\), so using the localization formula
\eqref{eq:localized-Ti} for $T_i$ once more gives
\[
T_i\Upsilon(a)a-\Upsilon(a)aT_i
=
 v^{-1}\frac{1-v^2Y_{-\alpha_i^\vee}}{1-Y_{-\alpha_i^\vee}}\,\bigl((\Upsilon(a)a)^{s_i}-\Upsilon(a)a\bigr)\eta_{s_i}.
\]
The scalar factor and $\eta_{s_i}$ are invertible in the localization, hence
\((\Upsilon(a)a)^{s_i}=\Upsilon(a)a\) for every simple reflection $s_i$, and
therefore \(c=\Upsilon(a)a\in F_{\mathrm{loc}}^{W_{\mathrm f}}\).

We now compute the square of the bar. On the \(q\)-independent localized
finite Hecke expressions occurring here, \(\overline{\bullet}\) agrees with
\(\Upsilon\). Hence it fixes \(T_i\) and \(Y_b\) after two applications, and
\[
\overline{a}=\Upsilon(a).
\]
For \(r\in O'\), the generator formula in Definition~\ref{def:levelzero-bar}
gives
\[
\overline{\overline{\gamma_r}}
=
\overline{a\gamma_r a^{-1}}
=
\Upsilon(a)a\gamma_r a^{-1}\Upsilon(a)^{-1}
=
c\gamma_r c^{-1}.
\]
Thus \(\overline{\bullet}^{\,2}\) is the inner automorphism \(x\mapsto cxc^{-1}\) on the generating set
from Proposition~\ref{prop:levelzero-generators}, and hence on all of
\(\mathbb H_{\mathrm{loc}}\). Since $c$ commutes with the $T_i$ and the $Y_b$,
the displayed formulas follow. Thus $\overline{\bullet}$ is an involution if and only if $c$
commutes with each $\gamma_r$. Because $c\in F_{\mathrm{loc}}$, this is
equivalent to
\[
\gamma_r c\gamma_r^{-1}=\pi_r(c)=c
\qquad (r\in O').
\]

Finally, specialize to $q=1$. Since
\(\omega_0^{-1}(q)=\Upsilon(q)=q^{-1}\), both maps preserve the ideal
\((q-1)\), while \(\overline q=q\). Hence \(\overline{\bullet}\) descends to
the specialization \(q=1\). By
\eqref{eq:positive-length-zero}, the finite projection of \(\pi_r\) is
\(u_r^{-1}\). Thus, for each \(b\in Q_{\mathrm{fin}}^\vee\), the difference
\(\pi_r(b)-u_r^{-1}(b)\) is an integral multiple of \(\delta\). Hence, after setting
$q=1$, we have
\[
Y_{\pi_r(b)}=Y_{u_r^{-1}(b)}.
\]
So the action of $\gamma_r$ on $F_{\mathrm{loc}}$ factors through the finite Weyl group element $u_r^{-1}$. Because $c\in F_{\mathrm{loc}}^{W_{\mathrm f}}$, we obtain
\[
\pi_r(c)=u_r^{-1}(c)=c
\qquad\text{for all }r\in O'.
\]
Therefore $\overline{\bullet}^{\,2}=\mathrm{id}$ after the specialization $q=1$.
\end{proof}

\begin{Ex}[The rank 1 case]
In rank $1$ one can compute the element \(c=\Upsilon(a)a\) explicitly. Writing
$\alpha=\alpha_1$ and $\gamma=\gamma_1$, the theorem says that
$\overline{\overline{\gamma}}=\gamma$ if and only if $\gamma$ commutes with $c$.

In rank $1$ one has
\begin{align*}
    a & = 1 + T_1\frac{\hbar}{1 - Y_{\alpha^\vee}}.
\end{align*}
After applying $\Upsilon$, the expression becomes the right-handed variant. Indeed, note that
\[
\Upsilon(Y_{\alpha^\vee})=T_1Y_{-\alpha^\vee}T_1^{-1},
\]
and a direct simplification using the rank-one Bernstein relation gives
\begin{align*}
    \Upsilon(a) & = 1 + \hbar \frac{Y_{\alpha^\vee}}{1 - Y_{\alpha^\vee}}T_1^{-1}.
\end{align*}
For any rational function $g$ in one variable, the Bernstein relation reads
\begin{align*}
    g(Y_{\alpha^\vee})T_1 - T_1g(Y_{-\alpha^\vee})
    =
    \hbar \frac{g(Y_{\alpha^\vee}) - g(Y_{-\alpha^\vee})}{1 - Y_{-\alpha^\vee}},
\end{align*}
which yields the useful cancellation identity
\begin{align}
\label{eq:rank-one-cross-term}
    T_1\frac{1}{1-Y_{\alpha^\vee}}
    +
    \frac{Y_{\alpha^\vee}}{1-Y_{\alpha^\vee}}T_1^{-1}
    =
    -2\hbar\frac{Y_{\alpha^\vee}}{(1-Y_{\alpha^\vee})^2}.
\end{align}
Also,
\[
(\Upsilon(a)-1)(a-1)
=
\hbar^2\frac{Y_{\alpha^\vee}}{(1-Y_{\alpha^\vee})^2}.
\]
Therefore \eqref{eq:rank-one-cross-term} gives
\begin{align*}
    \Upsilon(a)a
    &=
    1
    +\hbar\left(
        T_1\frac{1}{1-Y_{\alpha^\vee}}
        +
        \frac{Y_{\alpha^\vee}}{1-Y_{\alpha^\vee}}T_1^{-1}
    \right)
    +\hbar^2\frac{Y_{\alpha^\vee}}{(1-Y_{\alpha^\vee})^2}\\
    &=
    1-\hbar^2\frac{Y_{\alpha^\vee}}{(1-Y_{\alpha^\vee})^2}\\
    & = 1 + \hbar^2 \frac{1}{(1 - Y_{\alpha^\vee})(1 - Y_{-\alpha^\vee})}.
\end{align*}
This expression is \(W_{\mathrm f}\)-invariant, and its coefficient is fixed by
\(v\mapsto v^{-1}\), so it is fixed by \(\Upsilon\).
Thus
\[
c=\Upsilon(a)a
=1+\hbar^2\frac{1}{(1 - Y_{\alpha^\vee})(1 - Y_{-\alpha^\vee})}.
\]
Via inverse Cherednik duality, write the corresponding element in the dual
\(X\)-presentation as \(\varphi^{-1}(c)\). Then
\begin{align*}
    \varphi^{-1}(c) & = 1 + \hbar^2 \frac{1}{(1 - X_{\alpha^\vee})(1 - X_{-\alpha^\vee})}.
\end{align*}
The question is therefore whether \(\varphi^{-1}(c)\) commutes with
\(\varphi^{-1}(\gamma)=\pi\). In rank one,
\eqref{eq:pi-affine-action} gives
\(\pi(\alpha^\vee)=-\alpha^\vee+\delta\). Thus
relation~\ref{rel:daha-pi-x} in Definition~\ref{def:daha} gives
\(\pi X_{\alpha^\vee}\pi^{-1}=qX_{-\alpha^\vee}\), and hence
\begin{align*}
    \pi \varphi^{-1}(c)\pi^{-1} & = 1 + \hbar^2 \frac{1}{(1 - qX_{-\alpha^\vee})(1 - q^{-1}X_{\alpha^\vee})}.
\end{align*}
This is \(\varphi^{-1}(c)\) if and only if $q=1$. Equivalently, in rank $1$ the condition $\pi(c)=c$ from Theorem \ref{thm:levelzero-q1-involution} is exactly the condition $q=1$.
\end{Ex}

\begin{Rem}
The obstruction to involutivity is therefore uniform in all types with
$O'\neq\varnothing$, and it is a localized phenomenon: it is exactly the failure
of the $W_{\mathrm f}$-invariant rational function \(c=\Upsilon(a)a\) to be invariant
under the full extended affine action coming from the elements $\pi_r$. In rank $1$, one has
\[
c=1+\hbar^2\frac{1}{(1-Y_{\alpha^\vee})(1-Y_{-\alpha^\vee})},
\]
and the preceding example shows that the condition $\pi(c)=c$ is precisely the condition $q=1$.
\end{Rem}

\bibliographystyle{amsalpha}
\bibliography{bard_bib}

@misc{GG,
      title={Affine {H}ecke algebras associated to {K}ac-{M}oody groups}, 
      author={H. Garland and I. Grojnowski},
      year={1995},
      eprint={q-alg/9508019},
      archivePrefix={arXiv},
      primaryClass={q-alg},
      url={https://arxiv.org/abs/q-alg/9508019}, 
}

@misc{DMB,
      title={{Hodge} theory, intertwining functors, and the {Orbit Method} for real reductive groups}, 
      author={Dougal Davis and Lucas Mason-Brown},
      year={2025},
      eprint={2503.14794},
      archivePrefix={arXiv},
      primaryClass={math.RT},
      url={https://arxiv.org/abs/2503.14794}, 
}

@book{Kac_infdim,
  author    = {Kac, Victor G.},
  title     = {Infinite-Dimensional Lie Algebras},
  edition   = {3},
  publisher = {Cambridge University Press},
  year      = {1990},
  doi       = {10.1017/CBO9780511626234}
}

@book{CG,
  author    = {Chriss, Neil and Ginzburg, Victor},
  title     = {Representation {T}heory and Complex {G}eometry},
  series    = {Modern Birkh{\"a}user Classics},
  publisher = {Birkh{\"a}user Boston},
  address   = {Boston, MA},
  year      = {2010},
  note      = {Reprint of the 1997 edition},
  doi       = {10.1007/978-0-8176-4938-8},
  isbn      = {978-0-8176-4937-1}
}

@article{Garland,
author= {Garland, H.},
title= {A Cartan decomposition for p-adic loop groups},
journal = {Math. Ann.},
volume = {302},
year = {1995},
pages = {151-175}}

@article{HM,
author = {H{\'e}bert, Auguste and Muthiah, Dinakar},
title = {Completed Iwahori-Hecke algebra for Kac-Moody groups over local fields},
year = {2026},
journal = {Proceedings of the London Mathematical Society},
doi = {10.1112/plms.70189}
}

@book{Kumar,
  author    = {Kumar, Shrawan},
  title     = {{Kac}-{M}oody Groups, Their Flag Varieties and Representation Theory},
  series    = {Progress in Mathematics},
  volume    = {204},
  publisher = {Birkh{\"a}user Boston},
  address   = {Boston, MA},
  year      = {2002},
  doi       = {10.1007/978-1-4612-0105-2},
  isbn      = {978-0-8176-4227-3}
}

@article{L83,
  author  = {Lusztig, G.},
  title   = {Singularities, character formulas, and a q-analog of weight multiplicities},
  journal = {Ast{\'e}risque},
  volume  = {101--102},
  pages   = {208--229},
  year    = {1983}
}

@article {Lgeneric,
    AUTHOR = {Lusztig, G.},
     TITLE = {{H}ecke algebras and {J}antzen's generic decomposition
              patterns},
   JOURNAL = {Advances in Mathematics},
    VOLUME = {37},
      YEAR = {1980},
    NUMBER = {2},
     PAGES = {121--164},
       DOI = {10.1016/0001-8708(80)90031-6},
       URL = {https://doi.org/10.1016/0001-8708(80)90031-6},
}

@article{KL79,
  author  = {Kazhdan, David and Lusztig, George},
  title   = {Representations of {C}oxeter groups and {H}ecke algebras},
  journal = {Inventiones Mathematicae},
  volume  = {53},
  number  = {2},
  pages   = {165--184},
  year    = {1979},
  doi     = {10.1007/BF01390031}
}

@article {V,
    AUTHOR = {Vasserot, Eric},
     TITLE = {Induced and simple modules of double affine {H}ecke algebras},
   JOURNAL = {Duke Math. J.},
  FJOURNAL = {Duke Mathematical Journal},
    VOLUME = {126},
      YEAR = {2005},
    NUMBER = {2},
     PAGES = {251--323},
      ISSN = {0012-7094,1547-7398},
   MRCLASS = {20C08 (14M15 16E20 17B37)},
  MRNUMBER = {2115259},
       DOI = {10.1215/S0012-7094-04-12623-5},
       URL = {https://doi.org/10.1215/S0012-7094-04-12623-5},
}

@phdthesis{PPThesis,
  author = {Paul Philippe},
  title = {Masures et th{\'e}orie de {Kazhdan}-{L}usztig affine pour les groupes de {Kac}-{M}oody},
  school = {Universit{\'e} Jean Monnet Saint-\'Etienne},
  year = {2025},
  type = {Th{\`e}se de doctorat},
  note = {\'Ecole Doctorale No. 488, SIS -- Sciences Ing{\'e}nierie Sant{\'e}; soutenue publiquement le 30 juin 2025},
  url = {https://pphilippe.pages.math.cnrs.fr/webpage/Paul%20Philippe%20PhD%20Thesis.pdf}
}

@incollection {K,
    AUTHOR = {Kashiwara, M.},
     TITLE = {The flag manifold of {K}ac-{M}oody {L}ie algebra},
 BOOKTITLE = {Algebraic analysis, geometry, and number theory ({B}altimore,
              {MD}, 1988)},
     PAGES = {161--190},
 PUBLISHER = {Johns Hopkins Univ. Press, Baltimore, MD},
      YEAR = {1989},
      ISBN = {0-8018-3841-X},
   MRCLASS = {17B67 (14M15 20G15)},
  MRNUMBER = {1463702},
MRREVIEWER = {Venkatramani\ Lakshmibai},
}

@article {LPeriodic,
    AUTHOR = {Lusztig, G.},
     TITLE = {Periodic {$W$}-graphs},
   JOURNAL = {Represent. Theory},
  FJOURNAL = {Representation Theory. An Electronic Journal of the American
              Mathematical Society},
    VOLUME = {1},
      YEAR = {1997},
     PAGES = {207--279},
      ISSN = {1088-4165},
   MRCLASS = {20G05 (17B10 20F55)},
  MRNUMBER = {1464171},
MRREVIEWER = {R.\ W.\ Carter},
       DOI = {10.1090/S1088-4165-97-00033-2},
       URL = {https://doi.org/10.1090/S1088-4165-97-00033-2},
}

@book {Cherednik,
    AUTHOR = {Cherednik, Ivan},
     TITLE = {Double affine {H}ecke algebras},
    SERIES = {London Mathematical Society Lecture Note Series},
    VOLUME = {319},
 PUBLISHER = {Cambridge University Press, Cambridge},
      YEAR = {2005},
     PAGES = {xii+434},
      ISBN = {0-521-60918-6},
   MRCLASS = {32G34 (05E15 11L05 17B20 20C08 33D52 33D80)},
  MRNUMBER = {2133033},
       DOI = {10.1017/CBO9780511546501},
       URL = {https://doi.org/10.1017/CBO9780511546501},
}

@article {Yun,
    AUTHOR = {Yun, Zhiwei},
     TITLE = {Weights of mixed tilting sheaves and geometric {R}ingel
              duality},
   JOURNAL = {Selecta Math. (N.S.)},
  FJOURNAL = {Selecta Mathematica. New Series},
    VOLUME = {14},
      YEAR = {2009},
    NUMBER = {2},
     PAGES = {299--320},
      ISSN = {1022-1824,1420-9020},
   MRCLASS = {14F20 (14L30 20G25)},
  MRNUMBER = {2480718},
MRREVIEWER = {Ada\ Boralevi},
       DOI = {10.1007/s00029-008-0066-8},
       URL = {https://doi.org/10.1007/s00029-008-0066-8},
}

@article{Braverman_Kazhdan_2011,
  author  = {Braverman, Alexander and Kazhdan, David},
  title   = {The spherical {H}ecke algebra for affine {K}ac--{M}oody groups {I}},
  journal = {Annals of Mathematics},
  volume  = {174},
  number  = {3},
  pages   = {1603--1642},
  year    = {2011},
  doi     = {10.4007/annals.2011.174.3.5},
  url     = {https://doi.org/10.4007/annals.2011.174.3.5}
}

@article{Braverman_Kazhdan_Patnaik,
  author  = {Braverman, Alexander and Kazhdan, David and Patnaik, Manish M.},
  title   = {{Iwahori}--{H}ecke algebras for p-adic loop groups},
  journal = {Inventiones Mathematicae},
  volume  = {204},
  number  = {2},
  pages   = {347--442},
  year    = {2016},
  doi     = {10.1007/s00222-015-0612-x},
  url     = {https://doi.org/10.1007/s00222-015-0612-x},
  eprint  = {1403.0602},
  archivePrefix = {arXiv},
  primaryClass = {math.RT}
}

@misc{Bouthier_Vasserot,
  author  = {Bouthier, Alexis and Vasserot, Eric},
  title   = {On the geometric {S}atake equivalence for {K}ac--{M}oody groups},
  year    = {2025},
  eprint  = {2510.11466},
  archivePrefix = {arXiv},
  primaryClass = {math.RT},
  url     = {https://arxiv.org/abs/2510.11466}
}

@article{Gaussent_Rousseau,
  author  = {Gaussent, St{\'e}phane and Rousseau, Guy},
  title   = {Spherical {H}ecke algebras for {K}ac--{M}oody groups over local fields},
  journal = {Annals of Mathematics},
  volume  = {180},
  number  = {3},
  pages   = {1051--1087},
  year    = {2014},
  doi     = {10.4007/annals.2014.180.3.5},
  url     = {https://doi.org/10.4007/annals.2014.180.3.5}
}

@article{BPGR,
  author  = {Bardy-Panse, Nicole and Gaussent, St{\'e}phane and Rousseau, Guy},
  title   = {{Iwahori}--{H}ecke algebras for {K}ac--{M}oody groups over local fields},
  journal = {Pacific Journal of Mathematics},
  volume  = {285},
  number  = {1},
  pages   = {1--61},
  year    = {2016},
  doi     = {10.2140/pjm.2016.285.1},
  url     = {https://doi.org/10.2140/pjm.2016.285.1}
}

@article{Muthiah_2018,
  author  = {Muthiah, Dinakar},
  title   = {On {Iwahori}--{H}ecke algebras for p-adic loop groups: double coset basis and {B}ruhat order},
  journal = {American Journal of Mathematics},
  volume  = {140},
  number  = {1},
  pages   = {221--244},
  year    = {2018},
  doi     = {10.1353/ajm.2018.0004},
  url     = {https://doi.org/10.1353/ajm.2018.0004}
}

@article{Muthiah_Orr,
  author  = {Muthiah, Dinakar and Orr, Daniel},
  title   = {On the double-affine {B}ruhat order: the {$\epsilon=1$} conjecture and classification of covers in {ADE} type},
  journal = {Algebraic Combinatorics},
  volume  = {2},
  number  = {2},
  pages   = {197--216},
  year    = {2019},
  doi     = {10.5802/alco.37},
  url     = {https://doi.org/10.5802/alco.37}
}

@article{Deodhar_1985,
  author  = {Deodhar, Vinay V.},
  title   = {On some geometric aspects of {B}ruhat orderings. {I}. {A} finer decomposition of {B}ruhat cells},
  journal = {Inventiones Mathematicae},
  volume  = {79},
  number  = {3},
  pages   = {499--511},
  year    = {1985},
  doi     = {10.1007/BF01388520},
  url     = {https://doi.org/10.1007/BF01388520}
}

@misc{Muthiah_2019,
  author  = {Muthiah, Dinakar},
  title   = {Double-affine {K}azhdan--{L}usztig polynomials via masures},
  year    = {2019},
  eprint  = {1910.13694},
  archivePrefix = {arXiv},
  primaryClass = {math.RT},
  url     = {https://arxiv.org/abs/1910.13694}
}

@article{Hebert_Philippe,
  author  = {H{\'e}bert, Auguste and Philippe, Paul},
  title   = {On affine {K}azhdan--{L}usztig {R}-polynomials for {K}ac--{M}oody groups},
  journal = {Transformation Groups},
  year    = {2026},
  doi     = {10.1007/s00031-026-09974-y},
  url     = {https://doi.org/10.1007/s00031-026-09974-y},
  note    = {Published online June 13, 2026},
  eprint  = {2410.04872},
  archivePrefix = {arXiv},
  primaryClass = {math.RT}
}

@misc{SchnellMHM,
  author = {Schnell, Christian},
  title = {An overview of Morihiko Saito's theory of mixed {H}odge modules},
  year = {2014},
  eprint = {1405.3096},
  archivePrefix = {arXiv},
  primaryClass = {math.AG},
  url = {https://arxiv.org/abs/1405.3096}
}

\end{document}